\documentclass[a4paper,12pt,reqno]{amsart}
\title{A quantitative inverse theorem for the $U^4$ norm over finite fields}
\author{W.T. Gowers$^*$}
\thanks{$*$ Royal Society 2010 Anniversary Research Professor, at the University of Cambridge} 
\author{L. Mili\'cevi\'c$^\dagger$}
\thanks{$\dagger$ Mathematical Institute of the Serbian Academy of Sciences and Arts, Belgrade, Serbia}

\usepackage{amsmath, wrapfig}
\usepackage{dsfont, a4wide, amsthm, amssymb, amsfonts, graphicx}
\usepackage{fancyhdr, xspace, psfrag, setspace, supertabular, color}
\usepackage{hyperref}
\usepackage{bbm}
\usepackage{stmaryrd}
\usepackage{txfonts}

\newtheorem{theorem}{Theorem}[section]

\newtheorem{lemma}[theorem]{Lemma}

\newtheorem{corollary}[theorem]{Corollary}

\newtheorem*{definition*}{Definition}

\def\e{\epsilon}
\def\E{\mathbb{E}}
\def\Z{\mathbb{Z}}
\def\R{\mathbb{R}}

\def\C{\mathbb{C}}

\def\P{\mathbb{P}}
\def\F{\mathbb{F}}

\def\a{\alpha}
\def\be{\beta}

\def\g{\gamma}
\def\d{\delta}

\def\x{\mathbf{x}}
\def\y{\mathbf{y}}

\def\D{\Delta}
\def\b1{\mathbbm{1}}

\def \cP{\mathcal P}
\def \cQ{\mathcal Q}
\def \cA{\mathcal A}
\def\G{\Gamma}
\def\tpsi{\tilde\psi}
\def\tg{\tilde\gamma}

\def\dc{\bar{*}}
\def\vc{\mathop{\multimapdotbothvert}}
\def\hc{\mathop{\multimapdotboth}}
\def\mc{\mathop{\talloblong\!}}
\def\mcmc{\mathop{\talloblong\!\talloblong\!}}

\def\Arr{\mathop{\mathrm{Arr}}}
\def\bfv{\mathbf v}
\def\bfw{\mathbf w}
\def\pphi{{}^\prime\!\phi}

\def\rank{\mathop{\mathrm{rank}}}
\def\VP{\mathop{\mathrm{VP}}}

\def \t{\tau}
\def \tphi{\tilde{\phi'}}
\begin{document}
\maketitle

\begin{abstract} 
A remarkable result of Bergelson, Tao and Ziegler \cite{BTZ,TZ1} implies that if $c>0$, $k$ is a positive integer, $p\geq k$ is a prime, $n$ is sufficiently large, and $f:\F_p^n\to\C$ is a function with $\|f\|_\infty\leq 1$ and $\|f\|_{U^k}\geq c$, then there is a polynomial $\pi$ of degree at most $k-1$ such that $\E_xf(x)\omega^{-\pi(x)}\geq c'$, where $\omega=\exp(2\pi i/p)$ and $c'>0$ is a constant that depends on $c,k$ and $p$ only. A version of this result for low-characteristic was also proved by Tao and Ziegler \cite{TZ2}. The proofs of these results do not yield a lower bound for $c'$. Here we give a different proof in the high-characteristic case when $k=4$, which enables us to give an explicit estimate for $c'$. The bound we obtain is roughly doubly exponential in the other parameters. 
\end{abstract}

\section{Introduction}

An important role in additive combinatorics is played by a sequence of norms $\|.\|_{U^k}$ called \emph{(Gowers) uniformity norms}. They were introduced in \cite{gowers} as part of a proof of Szemer\'edi's theorem. It can be shown using repeated applications of the Cauchy-Schwarz inequality that if the characteristic function of a set $A$ is close in the $U^k$ norm to a constant function, then $A$ contains approximately as many arithmetic progressions of length $k+1$ as a random set of the same cardinality, and this reduces Szemer\'edi's theorem to the problem of understanding the properties of bounded functions with uniformity norms that are not small.

Let $G$ be a finite Abelian group and let $f:G\to\C$. The $U^k$ norm is given by the following formula.
\[\|f\|_{U^k}^{2^k}=\E_{x,a_1,\dots,a_k}\prod_{\e\in\{0,1\}^k}C^{|\e|}f\Bigl(x-\sum_{i=1}^k\e_ia_i\Bigr),\]
where, $C$ is the operation of taking complex conjugates, and $\E$ stands for the average when the parameters are chosen uniformly and independently from $G$. When $k=2$, this formula specializes to
\[\|f\|_{U^2}^4=\E_{x,a,b}f(x)\overline{f(x-a)f(x-b)}f(x-a-b).\]
It is not trivial that these formulae define norms (or a seminorm in the case $k=1$). However, the proof, which uses multiple applications of the Cauchy-Schwarz inequality, is not too hard.

Central to the main argument of \cite{gowers} is a ``local inverse theorem", which shows that if a function $f:\Z_N\to\C$ is bounded and has large $U^k$ norm, then there is an arithmetic progression $P$ of length $m$, where $m$ has a power dependence on $N$, such that the restriction of $f$ to $P$ correlates well with a polynomial phase function of degree at most $k-1$: that is, a function of the form $\exp(2\pi i\alpha q(x))$, where $q$ is a polynomial of degree at most $k-1$. A substantial extension of this result was obtained by Green, Tao and Ziegler \cite{GTZ}, who were able to give a complete description of the ``obstructions to uniformity" and thereby obtain a ``global inverse theorem". That is, they were able to identify a class of functions $\mathcal G_{k-1}$, which are sophisticated generalizations of polynomial phase functions of degree at most $k-1$, with the property that if $f:\Z_N\to\C$ is a function with $\|f\|_\infty\leq 1$ and $\|f\|_{U^k}\geq c$, then there exists $g\in\mathcal G_k$ such that $\langle f,g\rangle\geq c'(c,k)$. Crucially, the converse is also true: a function that correlates well with a function in $\mathcal G_{k-1}$ has a large $U^k$ norm. This major result completed a long-standing programme to obtain asymptotic estimates for the number of arithmetic progressions of length $k$ in the first $n$ primes, as well as many other linear configurations \cite{greentao3}. (The case $k=3$ was proved by Green and Tao \cite{greentao1}.)

A few years earlier, Bergelson, Tao and Ziegler proved an analogue of this result in the case where $G$ is $\F_p^n$ rather than $\Z_N$. Writing $\omega$ for $\exp(2\pi i/p)$, we can state their result as follows. (The formulation in their paper is different, but this is the main consequence of interest.) 

\begin{theorem} \cite{BTZ}
For every $c>0$, every positive integer $k$, and every prime $p\geq k$, there is a constant $c'>0$ with the following property: for every function $f:\F_p^n\to\C$ with $\|f\|_\infty\leq 1$ and $\|f\|_{U^k}\geq c$, there is a polynomial $\pi:\F_p^n\to\F_p$ of degree at most $k-1$ such that $\E_xf(x)\omega^{-\pi(x)}\geq c'$.
\end{theorem}

\noindent Loosely speaking, this says that a bounded function with large $U^k$ norm must correlate well with a polynomial phase function of degree at most $k-1$. 

The proof of Bergelson, Tao and Ziegler is infinitary in nature and gives no estimate for $c'$, though they state that in principle their arguments can be finitized in order to yield a very weak explicit bound. The main result of this paper is an alternative proof that is finitary and quantitative in the case $k=4$. Moreover, the estimate we obtain for $c'$ is ``reasonable": the function expressing the dependence on $c$ and $p$ is of the form $\exp\exp(Q(c^{-1},p))$, where $Q$ is quasipolynomial. 

We use a variety of tools to prove this. Some are taken from \cite{gowers}, and in particular the part of that paper devoted to proving Szemer\'edi's theorem for progressions of length 5, for which the $U^4$ norm is the one that must be understood. (Of those, some we import wholesale, but others we have managed to simplify, which in particular allows us to avoid relying on the rather technical \S 10 of that paper. It should be possible to simplify \cite{gowers} itself in a similar way -- we plan to check this.) Another tool we use in slightly adapted form is a ``symmetry argument" due to Green and Tao \cite{greentao1}, which played an important part in the proof of the inverse theorem over $\Z_N$. We also make use of a lemma from \cite{gowerswolf}, which states that the ``analytic rank" of a multilinear form has a certain subadditivity property. Using these tools we also develop some new ones, of which the main two are a bilinear analogue of Bogolyubov's method, which may well have other applications, and a stability theorem for almost bilinear functions defined on level sets of high-rank bilinear functions. 

We also have a second proof of the theorem, which we shall present in a different paper, a preprint of which will appear soon \cite{GM2}. The main difference is that in this paper our Bogolyubov method describes a certain ``mixed convolution" up to a small $L_2$ error, whereas in the other argument we obtain a (more complicated) description up to a small $L_\infty$ error. In the course of thinking about that, we proved a more combinatorial statement that can also be thought of as a bilinear analogue of Bogolyubov's method \cite{GM}. That version was discovered independently Bienvenu and L\^e \cite{BL}.

While writing this draft of the paper, we spotted that the proof of the stability theorem can probably be reorganized in a way that would simplify it. We also made certain choices about how to present statements that were less convenient for later use than they could have been. So this should not be regarded as a final draft: we hope to tidy it up soon.

\medskip

\noindent \textbf{Acknowledgements.} The second author would like to express his gratitude to Trinity College and the Department of Pure Mathematics and Mathematical Statistics at the University of Cambridge for their generous support while this work was carried out. He also acknowledges the support of the Ministry of Education, Science and Technological Development of the Republic of Serbia, Grant III044006.

\iftrue
\else
It turns out that understanding ``approximately bilinear maps" is the main task we face. Since this seems to be an interesting project in its own right, we have included certain results that are not needed for the proof of the main theorem, but that do contribute to that understanding. These extra results are clearly signalled, so that the reader who just wishes to see the proof of the main theorem can skip past them.

Several steps of our proof can be straightforwardly generalized to the corresponding steps for higher $U^k$ norms, and it is likely that the entire proof can be generalized. Later in the paper, we shall discuss this prospect in more detail. It is also possible that with heavy use of Bohr-sets technology our arguments can be adapted to the $\Z_N$ case. We have not yet considered whether they can deal with the low-characteristic case without a substantial new idea.
\fi

\section{Overview of the proof}\label{overview}

To begin with, we introduce some notation. If $G$ is a finite Abelian group and $f:G\to\C$, then we write $\partial_af$ for the function defined by the formula
\[\partial_af(x)=f(x)\overline{f(x-a)}.\]
We also write $\partial_{a,b}f$ for $\partial_a(\partial_bf)$, and so on. For example,
$\partial_{a,b,c}f(x)$ is equal to
\[f(x)\overline{f(x-a)f(x-b)}f(x-a-b)\overline{f(x-c)}f(x-a-c)f(x-b-c)\overline{f(x-a-b-c)}.\]

With this notation, we have a concise expression for the $U^4$ norm, namely
\[\|f\|_{U^4}^{16}=\E_{x,a,b,c,d}\partial_{a,b,c,d}f(x).\]
However, we also have other expressions; the following two, which can be easily checked, are of particular interest.
\[\|f\|_{U^4}^{16}=\E_{a,b}\|\partial_{a,b}f\|_{U^2}^4=\E_{a}\|\partial_{a}f\|_{U^3}^8.\]
Since $\|f\|_{U^k}\leq\|f\|_\infty$ for every $k$, easy averaging arguments allow us to deduce from the above identities that if $\|f\|_{U^4}\geq c$, then the following two statements hold.
\begin{enumerate}
\item There is a set $A\subset G^2$ of density at least $c/2$ such that $\|\partial_{a,b}f\|_{U^2}^4\geq c/2$ for every $(a,b)\in A$.
\item There is a set $A'\subset G$ of density at least $c/2$ such that $\|\partial_af\|_{U^3}^8\geq c/2$ for every $a\in A'$. 
\end{enumerate}
We make use of the first of these statements. This contrasts with the approach of Green, Tao and Ziegler \cite{GTZ2} to proving their $U^4$ inverse theorem for functions defined on $\Z_N$ (which they treated separately), which begins with the second statement. Roughly speaking, Green, Tao and Ziegler use the $U^3$ inverse theorem for each $a$ such that $\|\partial_af\|_{U^3}$ is large, and then prove that the corresponding generalized quadratic phase functions ``line up in a linear way". We use the $U^2$ inverse theorem (which is a very simple calculation using Fourier expansions) for each $(a,b)$ such that $\|\partial_{a,b}\|_{U^2}$ is large and then prove that the characters we obtain ``line up in a bilinear way". We start with a rather weak bilinearity property and gradually strengthen it: this is the sense in which our proof requires a detailed study of approximate bilinearity.

We shall now explain the rough scheme of the proof, but first we need a few informal definitions, which we shall make more precise later. If $G$ is a finite Abelian group, $A\subset G$ and $\phi:A\to G$, then we define a ``derivative" $\phi'$ by $\phi'(h)=\phi(x+h)-\phi(x)$, and we say that $\phi'$ is ``approximately well-defined" if the right-hand side almost always takes the same value when $x,x+h\in A$. We can regard $\phi'$ as something like a function, but its ``domain" is the convolution $\b1_A*\b1_{-A}$. (That is, the domain is the set $A-A$, but concepts such as ``almost always" are defined in terms of the weights given by the convolution.)

If instead $A\subset G^2$, we can now define ``partial derivatives" in a similar way. We define the ``vertical derivative" to be $D_v\phi=\phi(x,y+h)-\phi(x,y)$ and the ``horizontal derivative" to be $D_h\phi(w,y)=\phi(x+w,y)-\phi(x,y)$. Again, these may be nowhere near well-defined, but sometimes they come close. The appropriate weighted domains for these are the ``vertical convolution" $C_v\b1A$ and ``horizontal convolution" $C_h\b1_A$ of $\b1_A$, which are obtained by applying the definition for single-variable functions to each column/row of $\b1_A$ separately.

Let $G=\F_p^n$. 
\begin{enumerate}
\item If $\|f\|_{U^4}\gg 1$ then there is a subset $A$ of $G^2$ with $|A|\gg|G|^2$ and a function $\phi:A\to G$ such that $|\widehat{\partial_{a,b}}(\phi(a,b))|\gg 1$ for every $(a,b)\in A$. 
\item We can pass to a subset $A'\subset A$ with $|A'|\gg|A|$ such that (writing $\phi$ for the restriction of the original $\phi$ to $A'$) the multivalued function $D_hD_vD_hD_v\phi$ is approximately well defined.
\item By an averaging argument it follows that $\psi=D_hD_v\phi$ is an ``approximate bihomomorphism", in the sense that for almost all $x,y_1,y_2,w,h$ the expression
\[\psi(x,y_1)-\psi(x,y_1+h)-\psi(x+w,y_2)+\psi(x+w,y_2+h)\]
depends on $w$ and $h$ only. Here the weights used to define ``almost all" are given by the ``mixed convolution" $C_hC_w\b1_A$.
\item A bilinear extension of Bogolyubov's method yields a bounded number of bilinear forms $\be_1,\dots,\be_k$ such that for almost all $(x,y)\in G^2$ the value of $C_hC_w\b1_A$ is approximately determined by the values of $\be_1(x,y),\dots,\be_k(x,y)$.
\item By averaging and suitable generalized Cauchy-Schwarz inequalities, we can pass to a ``bilinear Bohr set" (that is, a set where $\be_1,\dots,\be_k$ take specified values) on which $\psi$ is still an approximate bihomomorphism. We may also pass to a subspace of bounded codimension in order to ensure that every non-trivial linear combination of the $\be_i$ has high rank.
\item The high-rank condition guarantees that the bilinear Bohr set $B$ has several quasirandomness properties. These can be used to prove a stability result: an approximate bihomomorphism on $B$ agrees almost everywhere with an exact bihomomorphism. 
\item An exact bihomomorphism defined on a high-rank bilinear Bohr set can be extended to the whole of $G^2$. Moreover an exact bihomomorphism on $G^2$ is given by a formula of the form $x.Ay+Ty+F(x)$, where $A$ is a linear map from $\F_p^n$ to $\F_p^n$, $T$ is a linear map from $\F_p^n$ to $\F_p$, and $F$ is an arbitrary function from $\F_p^n$ to $\F_p$. 
\item $F$ (which appears because we have ``antidifferentiated" with respect to $y$) can be shown to agree on a large set with a linear function. 
\item Using all this, one can prove that $\phi$ agrees on a large set with a bi-affine map from $\F_p^n\times\F_p^n$ to $\F_p$.
\item A ``symmetry argument", which builds on a corresponding argument of Green and Tao for the $U^3$ norm, then shows that $\E_x\partial_{a,b,c}f(x)$ correlates non-trivially with $\omega^{\tau(a,b,c)}$ for some symmetric trilinear form $\tau$.
\item From this it is easy to deduce the existence of a cubic polynomial $\kappa$ such that the function $g(x)=f(x)\omega^{-\kappa(x)}$ has large $U^3$ norm. 
\item By the $U^3$ inverse theorem, it follows that $g$ correlates with a quadratic phase function, and therefore that $f$ correlates with a cubic phase function.
\end{enumerate}

We remark that the various steps outlined above can be thought of as natural generalizations, some routine, others requiring more thought, of a proof of an inverse theorem for the $U^3$ norm for functions defined on $\F_p^n$, with bilinear structure and bilinear functions replacing linear structure and linear functions. However the steps above do not quite correspond to the steps in Green and Tao's proof of the $U^3$ inverse theorem. The main difference is that in the linear case there are tools that allow one to pass in one step from a function that ``respects many additive quadruples" in an arbitrary dense set to a Freiman homomorphism on a dense subset of that set. We do not know how to prove the corresponding statement for bilinear functions without several steps (roughly steps (2) to (6) in the scheme above) in which one deals with approximate bihomomorphisms. Another difference is that, for reasons that we shall explain later, our bilinear version of Bogolyubov's method describes mixed single convolutions to within a small $L_2$ error, whereas the linear version describes double convolutions to within a small $L_\infty$ error.

\section{Vertical parallelograms and other arrangements of points}

In this section, we shall carry out part of the proof that is close to arguments in \cite{gowers}, and in particular in \S 12 of that paper. We shall quote a lemma from that section, modify another, and prove a technical lemma that links the two.

\subsection{Respecting additive quadruples}

Let $G$ be a finite Abelian group. For the rest of this paper, we shall follow standard practice and say that a function $f:G\to\C$ is \emph{bounded} if $\|f\|_\infty\leq 1$. If $f$ is bounded and $\|f\|_{U^4}\geq c_0$, then $\|f\|_{U^4}^{16}\geq c_0^{16}$, so, as we remarked at the beginning of the previous section, there is a subset $A\subset G^2$ of density at least $c_1$ such that $\|\partial_{a,b}f\|_{U^2}^4\geq c_1$ for every $(a,b)\in A$, where $c_1=c_0^{16}/2$. 

A standard fact, which is easy to prove, is that if $G$ is a finite Abelian group and $g:G\to\C$, then $\|g\|_{U^2}=\|\hat g\|_4$. If in addition $g$ is bounded, then $\|\hat g\|_2=\|g\|_2\leq\|g\|_\infty\leq 1$. Since we also know that $\|\hat g\|_4^4\leq\|\hat g\|_2^2\|\hat g\|_\infty^2$, it follows that $\|\hat g\|_\infty\geq\|\hat g\|_4^2$. Therefore, for each $(a,b)\in A$ we have that $\|\widehat{\partial_{a,b}f}\|_\infty\geq c_1^{1/2}$. Equivalently, we can find a function $\phi:A\to\hat G$ such that $|\widehat{\partial_{a,b}f}(\phi(a,b))|\geq c_1^{1/2}$ for every $(a,b)\in A$. 

Much later in the argument, it will be important to us that $\phi(a,b)$ is a Freiman homomorphism in $a$ for each fixed $b$. This we can achieve by passing to a suitable subset. The argument is standard, but for completeness we give proofs here the parts that we have not found in directly quotable form in the literature.

The first result we shall need is Proposition 6.1 of \cite{gowers}, which we give in a very slightly restated form.

\begin{lemma} \label{linearaddquads}
Let $G$ be a finite Abelian group. Let $\a>0$, let $f:G\to\C$ be a bounded function, let $B\subset G$, and let $\phi:B\to\hat G$ be a function such that $\E_a\b1_B(a)|\widehat{\partial_af}(\phi(a))|^2\geq\a$. Then there are at least $\a^4|G|^3$ quadruples $(a,b,c,d)\in B^4$ such that $a+b=c+d$ and $\phi(a)+\phi(b)=\phi(c)+\phi(d)$.
\end{lemma}

We shall also need the Balog-Szemer\'edi-Gowers lemma. The bounds stated below are not those obtained in \cite{gowers}; they are due to Schoen \cite{schoen} (see also \cite{sheffer} from where we obtained the constant $2^{22}$, which can be easily extracted from his proof).

\begin{lemma} \label{BSG}
Let $\d>0$, let $H$ be a finite Abelian group and let $A\subset H$ be a set such that there are at least $\d|A|^3$ quadruples $(a,b,c,d)\in A^4$ with $a+b=c+d$. Then $A$ has a subset $A'$ of size at least $\d|A|/6$ such that the difference set $A'-A'$ has cardinality at most $2^{22}|A'|/\d^4$. 
\end{lemma}

Thirdly, we need a lemma that is similar to Lemma 7.5 of \cite{gowers}, but with enough differences that we give a complete proof. 

\begin{lemma} \label{gettingahom}
Let $G=\F_p^n$, let $B\subset G$, let $\phi:B\to G$ be a function with graph $\G$, and suppose that $|\G-\G|\leq C|\G|$. Then there is a subset $B'\subset B$ of size at least $p^{-1}C^{-5}|B|$ such that the restriction of $\phi$ to $B'$ is a Freiman homomorphism.
\end{lemma}

\begin{proof}
Let $A$ be the set of all possible values of $\phi(a_1)-\phi(a_2)-\phi(a_3)+\phi(a_4)$ such that $a_1-a_2=a_3-a_4$. Since $\G$ is the graph of a function, $|A+\G|=|A||\G|$. But $A\subset 2\G-2\G$, so by Plunnecke's theorem $|A+\G|\leq C^5|\G|$. Therefore, $|A|\leq C^5$.

Let $k$ be such that $p^k\geq C^5$ and let $V$ be a random subspace of $G$ of codimension $k$. Then each non-zero point of $A$ belongs to $V$ with probability at most $p^{-k}\leq C^{-5}$, so the expected number of non-zero points of $A$ in $V$ is less than 1. Therefore, we can find $V$ of codimension $k$ such that $V\cap A=\{0\}$.

Now let $x\in G$ and suppose that $a_1-a_2=a_3-a_4$ and that $\phi(a_i)\in V+x$ for each $i$. Then $\phi(a_1)-\phi(a_2)-\phi(a_3)+\phi(a_4)\in V\cap A$, so it equals zero. This proves that for any $x$ the restriction of $\phi$ to $\{a\in B:\phi(a)\in V+x\}$ is a Freiman homomorphism. By averaging, at least one of these sets has size at least $p^{-k}|B|\geq p^{-1}C^{-5}|B|$.
\end{proof}

Putting these lemmas together in various ways gives us the following corollaries, which are very useful for this kind of proof.

\begin{corollary} \label{restricttohom}
Let $G,G'$ be finite Abelian groups, let $B\subset G$ and let $\phi:B\to G'$ be a function such that there are at least $\d|B|^3$ quadruples $(a,b,c,d)\in B^4$ with $a+b=c+d$ and $\phi(a)+\phi(b)=\phi(c)+\phi(d)$. Then $B$ has a subset $B''$ of size at least $p^{-1}2^{-113}\d^{21}|B|$ such that the restriction of $\phi$ to $B''$ is a Freiman homomorphism.
\end{corollary}

\begin{proof}
By Lemma \ref{BSG} applied to the group $H=G\times G'$ and the set $\G=\{(x,\phi(x)):x\in B\}$ (that is, the graph of $\phi$), we find that $\G$ has a subset $\G'$ of size at least $\d|\G|/6$ such that $|\G'-\G'|\leq 2^{22}|\G'|/\d^4$. Let $B'=\{x\in G:(x,\phi(x))\in\G'\}$, and note that $|B'|=|\G'|$.

We now apply Lemma \ref{gettingahom} to $B'$, $\G'$ and the restriction of $\phi$ to $B'$. We may take $C=2^{22}\d^{-4}$. We obtain a subset $B''\subset B'$ of size at least $p^{-1}C^{-5}|B'|$ such that the restriction of $\phi$ to $B''$ is a Freiman homomorphism. It is not hard to check that this implies the conclusion of the corollary.
\end{proof}

\begin{corollary} \label{restricttohom2}
Let $G$ be a finite Abelian group. Let $\a>0$, let $f:G\to\C$ be a bounded function, let $B\subset G$ with $|B|=\be$, and let $\phi:B\to\hat G$ be a function such that $\E_a\b1_B(a)|\widehat{\partial_af}(\phi(a))|^2\geq\a$. Then there is a subset $B''\subset B$ of size at least $p^{-1}2^{-113}\a^{84}\be^{-63}|B|$ such that the restriction of $\phi$ to $B''$ is a Freiman homomorphism.
\end{corollary}

\begin{proof}
By Lemma \ref{linearaddquads}, there are at least $\a^4|G|^3=\a^4\be^{-3}|B|^3$ quadruples $(a,b,c,d)\in B^4$ such that $a+b=c+d$. Corollary \ref{restricttohom}, with $\d=\a^4\be^{-3}$, now gives the result.
\end{proof}

\begin{corollary} \label{homsinrows}
Let $G=\F_p^n$, let $A$ be a subset of $G^2$ of density at least $c_1$, and let $\phi:A\to\hat G$ be a function such that $|\widehat{\partial_{a,b}f}(\phi(a,b))|\geq c_1^{1/2}$ for every $(a,b)\in A$. Then $A$ has a subset $A'$ of density at least $p^{-1}2^{-113}c_1^{170}$ such that for each fixed $b$ the restriction of $\phi$ to $A'\cap(G\times\{b\})$ is a Freiman homomorphism.
\end{corollary}

\begin{proof}
Let us write $A_{\bullet b}$ for the set $\{a\in G:(a,b)\in A\}$ and $d(b)$ for its density. We shall also write $\phi_{\bullet b}$ for the function $\phi_{\bullet b}(a)=\phi(a,b)$. For each $b$ we have that $\E_a\b1_{A_{\bullet b}}(a)|\widehat{\partial_a(\partial_bf)}(\phi(a,b))|^2\geq c_1d(b)$. By Corollary \ref{restricttohom2} it follows that $A_{\bullet b}$ has a subset $A_{\bullet b}'$ of density at least $p^{-1}2^{-113}(c_1d(b))^{84}d(b)^{-62}=p^{-1}2^{-113}c_1^{84}d(b)^{22}$ such that the restriction of $\phi_{\bullet b}$ to $A_{\bullet b}'$ is a Freiman homomorphism. 

We can take $A'$ to be the union of the sets $A_{\bullet b}'\times\{b\}$. Since the average of the densities $d(b)$ is at least $c_1$, Jensen's inequality implies that $A'$ has density at least $p^{-1}2^{-113}c_1^{106}$, which proves the result.
\end{proof}

\subsection{Respecting two-dimensional arrangements}

We now prove some bilinear generalizations of the above linear results. We shall make use of a lemma that is a slight restatement of Lemma 12.2 of \cite{gowers}. (There the lemma is stated for functions defined on $\Z_N$, but the proof is virtually identical for general finite Abelian groups. Also, there are two misprints in the statement in \cite{gowers}: it begins with the words ``Let $\g,\eta>0$" but it should say ``Let $\beta,\g>0$", and $B$ is a subset of $\Z_N^2$ and not $\Z_N$. Finally, we have given different names to some of the variables and we use different normalizations.) 

If $G$ is a finite Abelian group, we define a \emph{vertical parallelogram of width $w$ and height $h$} in $G$ to be a sequence of points $P=((x,y),(x,y+h),(x+w,y'),(x+w,y'+h))$. If $P$ is a vertical parallelogram and $\phi:G^2\to G$, we define $\phi(P)$ to be 
\[\phi(x,y)-\phi(x,y+h)-\phi(x+w,y')+\phi(x+w,y'+h).\]
Note that if $\phi$ is bilinear, then this equals $\phi(x+w,h)-\phi(x,h)=\phi(w,h)$. So for bilinear functions, $\phi(P)$ depends just on the width and height. As in the linear case, we shall be interested in understanding weakenings of this condition where we assume that it is often, but not necessarily always, the case that if $P_1$ and $P_2$ are two vertical parallelograms of the same width and height, then $\phi(P_1)=\phi(P_2)$. 

To that end, define a \emph{4-arrangement} to be a pair $(P_1,P_2)$ of vertical parallelograms of the same width and height. It is sometimes convenient to represent a 4-arrangement in an equivalent way as a sequence of the following form:
\begin{align*}
((x_1,y_1),(x_1,y_1+h),(x_1+w,y_2),&(x_1+w,y_2+h),(x_2,y_3),\\
&(x_2,y_3+h),(x_2+w,y_4),(x_2+w,y_4+h)).\\
\end{align*}
We say that $\phi$ \emph{respects} the 4-arrangement $(P_1,P_2)$ if $\phi(P_1)=\phi(P_2)$. In terms of the sequence, this is saying that $\sum_i\e_i\phi(u_i,v_i)=0$, where $(u_1,v_1),\dots,(u_8,v_8)$ are the points of the 4-arrangement in the order listed above, and $(\e_1,\dots,\e_8)$ is the initial segment $(1,-1,-1,1,-1,1,1,-1)$ of length 8 of the Morse sequence.

The lemma we quote tells us that if $\phi$ has the property we have just established, then it respects a positive proportion of 4-arrangements.

\begin{lemma} \label{somearrangements}
Let $\a,\g>0$, let $G$ be a finite Abelian group, let $f:G\to\C$ be a function with $\|f\|_\infty\leq 1$, let $A\subset G^2$ be a set of density $\a$, and let $\phi:A\to G$ be a function such that $|\widehat{\partial_{a,b}f}(\phi(a,b))|\geq\g$ for every $(a,b)\in A$. Then $\phi$ respects at least $\a^{16}\g^{48}|G|^8$ 4-arrangements in $A$. 
\end{lemma}

In our case, we can take $\a=c_1$ and $\g=c_1^{1/2}$, so the number of 4-arrangements respected is at least $c_1^{40}|G|^8$. Note that the total number of 4-arrangements in $G$ is $|G|^8$, since the parameters $x_1,x_2,y_1,y_2,y_3,y_4,w$ and $h$ can all be chosen freely (in the sequence description) and different choices give rise to different 4-arrangements.

For later purposes, we shall need to consider more elaborate structures, which we can think of as ``vertical parallelograms of vertical parallelograms". Given a vertical parallelogram $P$, write $w(P)$ and $h(P)$ for its width and height. Now define a \emph{second-order vertical parallelogram} to be a quadruple $Q=(P_1,P_2,P_3,P_4)$ of vertical parallelograms such that the quadruple
\[((w(P_1),h(P_1)),(w(P_2),h(P_2)),(w(P_3),h(P_3)),(w(P_4),h(P_4)))\]
is itself a vertical parallelogram $P$. 

This leads to a number of obvious further definitions. Define the \emph{width} of $Q$ to be $w(P_3)-w(P_1)$ and the \emph{height} of $Q$ to be $h(P_2)-h(P_1)$: that is, the width and height of $P$. Given a function $\phi:G^2\to G$, define $\phi(Q)$ to be $\phi(P_1)-\phi(P_2)-\phi(P_3)+\phi(P_4)$. Define a \emph{second-order 4-arrangement} to be a pair $(Q_1,Q_2)$ of second-order vertical parallelograms of the same width and height. Finally, if $(Q_1,Q_2)$ is a second-order 4-arrangement, say that $\phi$ \emph{respects} $(Q_1,Q_2)$ if $\phi(Q_1)=\phi(Q_2)$. 

We shall need our function $\phi$ to respect a positive proportion of second-order 4-arrangements. It turns out that this is automatically the case if it respects a positive proportion of first-order 4-arrangements, as we shall now prove. For the purposes of understanding the statement of Corollary \ref{manyarr2s} below, it is useful to note that the number of vertical parallelograms of any given width and height is $|G|^3$, and a second-order 4-arrangement is obtained by replacing each point $(x,y)$ of a first-order 4-arrangement by a vertical parallelogram of width $x$ and height $y$, so the number of second-order 4-arrangements is $|G|^{24}$ times the number of first-order 4-arrangements, or $|G|^{32}$. 

For the next lemma, we need a couple of definitions. If $G$ is a finite Abelian group, $\cA$ is the group algebra of $G$ (with convolution as its product), and $f:G^2\to\cA$, then define a quantity $\Arr(f)$ to be
\begin{align*}\E_{x_1,x_2,y_1,y_2,y_3,y_4,w,h}\big\langle f(x_1,y_1)f(x_1,y_1+h)^*&f(x_1+w,y_2)^*f(x_1+w,y_2+h),\\
&f(x_2,y_3)f(x_2,y_3+h)^*f(x_2+w,y_4)^*f(x_2+w,y_4+h)\big\rangle,\\
\end{align*}
where the inner product is defined using sums -- that is, $\langle u,v\rangle=\sum_au(a)\overline{v(a)}$ -- and $u^*$ is defined by $u^*(a)=\overline{u(-a)}$. We shall make frequent use of the fact that $\langle u,vw\rangle=\langle uv^*,w\rangle$, together with the commutativity and associativity of convolution.

The operator $\Arr$ calculates an average over 4-arrangements. To express the definition more concisely, we adopt the definition that if $P$ is the vertical parallelogram with vertices $(x,y_1),(x,y_1+h),(x+w,y_2)$ and $(x+w,y_2+h)$, then $f(P)=f(x,y_1)f(x,y_1+h)^*f(x+w,y_2)f(x+w,y_2+h)^*$. Using this notation, and writing $P_1\sim P_2$ to mean that $P_1$ and $P_2$ have the same width and height, we have
\[\Arr(f)=\mathop{\E}_{P_1\sim P_2}\langle f(P_1),f(P_2)\rangle.\] 

In a similar way we define an operator $\Arr_2$ that calculates an average over second-order 4-arrangements. As we have seen, a second-order vertical parallelogram is a quadruple $Q=(P_1,P_2,P_3,P_4)$ such that the widths and heights $(w_1,h_1), (w_2,h_2), (w_3,h_3)$ and $(w_4,h_4)$ form a vertical parallelogram. Note that this means that $w_1=w_2$, $w_3=w_4$, and $h_2-h_1=h_4-h_3$. We define the width $w(Q)$ and height $h(Q)$ of $Q$ to be $w_3-w_1$ and $h_2-h_1$, respectively. We also define $f(Q)$ to be $f(P_1)f(P_2)^*f(P_3)^*f(P_4)$, and we then set
\[\Arr_2(f)=\mathop{\E}_{Q_1\sim Q_2}\langle f(Q_1),f(Q_2)\rangle,\]
where the expectation is over pairs of second-order vertical parallelograms of the same width and height.

In the proof of the lemma we shall apply a vector-valued Cauchy-Schwarz inequality, namely
\[\E_{x\in X}\langle f(x),g(x)\rangle\leq(\E_{x\in X}\|f(x)\|_2^2)^{1/2}(\E_{x\in X}\|g(x)\|_2^2)^{1/2},\]
which follows from two applications of the Cauchy-Schwarz inequality itself, one in the group algebra and one in $\R^X$.

The way to think of the statement below is that the first three terms on the right-hand side are all ones that in applications can be bounded above easily, and therefore what we learn is that $\Arr_2(f)$ is bounded below in terms of $\Arr(f)$. Note also that the inequality is homogeneous, as it must be: if $f$ is multiplied by a positive constant $\lambda$, then $\Arr(f)$ is multiplied by $\lambda^8$, the first term by $\lambda^2$, the second by $\lambda$, the third by $\lambda$, and $\Arr_2(f)^{1/8}$ by $\lambda^4$ (the last factor coming from the fact that a second-order 4-arrangement has 32 points).

\begin{lemma} \label{algebracs}
Let $G$ be a finite Abelian group, let $\cA$ be the group algebra of $G$, and let $f:G^2\to\cA$. Then 
\[\Arr(f)\leq\|\E_{x,y}f(x,y)f(x,y)^*\|_2(\E_x\|E_yf(x,y)f(x,y)^*\|_2^2)^{1/4}(\E_P\|f(P)\|_2^2)^{1/8}\Arr_2(f)^{1/8}\]
\end{lemma}

\begin{proof}
As with many statements of this kind, the proof proceeds by several applications of the Cauchy-Schwarz inequality. Again, to keep the expressions that appear in the proof manageable we shall need to use some condensed notation. In particular, given a vertical parallelogram $P$ with vertices $(x,y_1),(x,y_1+h),(x+w,y_2)$ and $(x+w,y_2+h)$, which we treat as an ordered sequence, define $x(P)$ to be $x$, $y_1(P)$ to be $y$, and so on. Then we let $\cP(x_1,w,y_2)$ be the set of all $P$ such that $x_1(P)=x_1, w(P)=w$ and $y_2(P)=y_2$, and so on for other collections of parameters. (This notation is ambiguous, but we shall choose the letters to make it clear which parameters are being referred to.) We shall also write $P_1\sim P_2$ to mean that $P_1$ and $P_2$ have the same width and height.
\begin{align*}
\E_{P_1\sim P_2}\langle f(P_1),f(P_2)\rangle &=\E_{x_1,w,y_1,y_2}\E_{P_1\in\cP(x_1,w,y_1,y_2)}\E_{P_2\sim P_1}\langle f(P_1),f(P_2)\rangle\\
\end{align*}
Now let us adopt the convention that the four points of $P_1$ are $(x_1,y_1), (x_1,y_1+h), (x_1+w,y_2)$ and $(x_1+w,y_2+h)$. The right-hand side above can then be rewritten
\[\E_{x_1,w,y_1,y_2}\langle f(x_1,y_1)f(x_1+w,y_2)^*,\E_hf(x_1,y_1+h)f(x_1+w,y_2+h)^*\E_{P_2\sim P_1}f(P_2)^*\rangle.\]
We now apply Cauchy-Schwarz. This gives us a product of two terms, of which the square of the first is
\begin{align*}
\E_{x_1,w,y_1,y_2}\|f(x_1,y_1)f(x_1+w,y_2)^*\|_2^2&=\E_{x_1,w,y_1,y_2}\langle f(x_1,y_1)f(x_1,y_1)^*,f(x_1+w,y_2)f(x_1+w,y_2)^*\rangle\\
&=\|\E_{x,y}f(x,y)f(x,y)^*\|_2^2.\\
\end{align*}
The square of the second is
\begin{align*}\E_{x_1,w,y_1,y_2}\|\E_hf(x_1,y_1+h)f(x_1+w,y_2+h)^*\E_{P_2\in\cP(w,h)}f(P_2)^*\|^2,\\
\end{align*}
which expands to
\begin{align*}\E_{x_1,w,y_1,y_2,h_1,h_2}\langle f(x_1,y_1+h_1)&f(x_1+w,y_2+h_1)^*\E_{P_2\in\cP(w,h_1)}f(P_2)^*,\\
&f(x_1,y_1+h_2)f(x_1+w,y_2+h_2)^*\E_{P_3\in\cP(w,h_2)}f(P_3)^*\rangle.\\
\end{align*}

Now the points $(x_1,y_1+h_1), (x_1,y_1+h_2), (x_1+w,y_2+h_1)$ and $(x_1+w,y_2+h_2)$ form the vertices of a general vertical parallelogram of width $w$ and height $h_2-h_1$. It follows that the final expression above can be rewritten as
\[\E_{w,h_1,h_2}\E_{P_1\in\cP(w,h_1-h_2), P_2\in\cP(w,h_1),P_3\in\cP(w,h_2)}\langle f(P_1),f(P_2)f(P_3)^*\rangle,\]
or, more concisely,
\[\mathop{\E}_{\substack{h(P_1)=h(P_2)-h(P_3)\\ w(P_1)=w(P_2)=w(P_3)}}\langle f(P_1),f(P_2)f(P_3)^*\rangle.\]

We now do something similar in the horizontal direction. That is, we isolate from the expression above the variables that are needed to specify the first vertical edge of $P_1$ and then apply Cauchy-Schwarz. So first we rewrite the expression as
\[\E_{x_1,y_1,h}\bigl\langle f(x_1,y_1)f(x_1,y_1+h)^*,\,\E_{w,y_2}f(x_1+w,y_2)f(x_1+w,y_2+h)^*\E_{\substack{h(P_2)-h(P_3)=h\\ w(P_2)=w(P_3)=w}}f(P_2)f(P_3)^*\bigr\rangle.\]
When we apply Cauchy-Schwarz to this expression, the square of the first term is
\begin{align*}
\E_{x_1,y_1,h}\|f(x_1,y_1)f(x_1,y_1+h)^*\|_2^2&=\E_{x,y_1,y_2}\langle f(x,y_1)f(x,y_1)^*,f(x,y_2)f(x,y_2)^*\rangle\\
&=\E_x\|\E_yf(x,y)f(x,y)^*\|_2^2.\\
\end{align*}
As for the second, it is
\begin{align*}
\E_{x_1,h}\bigl\|\E_{w,y_2}f(x_1+w,y_2)f(x_1+w,y_2+h)^*\E_{\substack{h(P_2)-h(P_3)=h\\ w(P_2)=w(P_3)=w}}f(P_2)f(P_3)^*\bigr\|_2^2.\\
\end{align*}
(The variable $y_1$ is not present because the expression being averaged over does not depend on $y_1$.) To see what this expands to, note that when we expand, each of the variables $w,y_2,P_2$ and $P_3$ is duplicated, while the variables $x_1$ and $h$ are not. The product of values of $f$ becomes
\[f(x_1+w_1,y_2)f(x_1+w_1,y_2+h)^*f(x_1+w_2,y_3)^*f(x_1+w_2,y_3+h)\]
which is $f(P_1)$, where $P_1$ is a general vertical parallelogram of width $w_2-w_1$ and height $h$. So we end up with a term
\[\E_{P_1}\big\langle f(P_1),\mathop{\E}_{\substack{h(P_2)-h(P_3)=h(P_4)-h(P_5)=h(P_1)\\ w(P_2)=w(P_3), w(P_4)=w(P_5)\\ w(P_4)-w(P_2)=w(P_1)\\ }}f(P_2)f(P_3)^*f(P_4)^*f(P_5)\big\rangle.\]

For the final step, we apply Cauchy-Schwarz one more time, to the inner product that we have just obtained. The square of the first term is $\E_P\|f(P)\|_2^2$, and the square of the second works out as
\[\mathop{\E}_{\substack{h(P_2)-h(P_3)=h(P_4)-h(P_5)=h(P_6)-h(P_7)=h(P_8)-h(P_9)\\ w(P_2)=w(P_3), w(P_4)=w(P_5), w(P_6)=w(P_7), w(P_8)=w(P_9)\\ w(P_4)-w(P_2)=w(P_8)-w(P_6)\\ }}\langle f(P_2)f(P_3)^*f(P_4)^*f(P_5),f(P_6)f(P_7)^*f(P_8)^*f(P_9)\rangle.\]
But this last expression is precisely the expectation of $f(Q)$ over all second-order 4-arrangements $Q$. That is, it is $\Arr_2(f)$.

Putting all this together and taking account of the various moments where we squared the expressions we were looking at gives the result claimed.
\end{proof}

\begin{corollary} \label{manyarr2s}
Let $G$ be a finite Abelian group, let $A\subset G^2$ be a set of density $\a$ and let $\phi:A\to G$ be a function that respects at least $\theta|G|^8$ 4-arrangements. Then $\phi$ respects at least $\theta^{8}\a^{-12}|G|^{32}$ second-order 4-arrangements.
\end{corollary}

\begin{proof}
Let $\cA$ be the group algebra of $G$ and let $f:G^2\to\cA$ be defined by setting $f(x,y)$ to be $\delta_{\phi(x,y)}$ if $(x,y)\in A$ and 0 otherwise. Then $\Arr(f)$ is the average over all 4-arrangements in $G^2$ of a function that takes the value 1 if the 4-arrangement lives inside $A$ and is respected by $\phi$, and zero otherwise. That is, it is $|G|^{-8}$ times the number of 4-arrangements respected by $\phi$, which is at least $\theta$ by hypothesis.

We also have that $f(x,y)f(x,y)^*=\delta_0$ when $(x,y)\in A$ and 0 otherwise. Therefore, $\E_{x,y}f(x,y)f(x,y)^*=\a\delta_0$, from which it follows that $\|\E_{x,y}f(x,y)f(x,y)^*\|_2=\a$.

Similarly, $\E_x\|\E_yf(x,y)f(x,y)^*\|_2^2$ is the mean square density of the columns of $A$, which is at most $\a$, so the second term is at most $\a^{1/4}$. 

For the third term, $\E_P\|f(P)\|_2^2$ is the probability that a random vertical parallelogram belongs to $A$, which is at most $\a^2$ (because two points of a parallelogram can be chosen independently, and the probability that the remaining two points belong is at most 1), so we obtain a bound of at most $\a^{1/4}$ again.

Therefore, it follows from Lemma \ref{algebracs} that $\Arr_2(f)\geq\theta^8\a^{-12}\Arr(f)$, which is equivalent to the conclusion of this corollary.
\end{proof}

We remark that the factor $\a^{-12}$, though it improves the bound, does not make a significant difference to our later arguments, and there would be no problem in replacing it by 1. 

\subsection{Respecting most arrangements}

Our next target is to prove a modification of Lemma 12.5 in \cite{gowers}. We shall show that if $\phi$ respects a positive proportion of second-order 4-arrangements in $A$, then we can pass to a large subset $A'$ of $A$ where the proportion goes up to $1-\eta$ for a small absolute constant $\eta$. The main differences between this and Lemma 12.5 of \cite{gowers} are that here we work in $\F_p^n$ instead of $\Z_N$, and here we consider second-order 4-arrangements rather than first-order 8-arrangements, but the technique of proof is the same. 

In preparation for the main result of the subsection, Lemma \ref{densification} below, we shall need a technical result in linear algebra. Before stating it, we need to make a simple observation and use it to give a definition.

Let $(Q_1,Q_2)$ be a second-order 4-arrangement in a set $A$ and let $\phi:A\to G$ be a function that respects $(Q_1,Q_2)$. If $(Q_1,Q_2)=(P_1,\dots,P_8)$, then this information tells us that
\[\phi(P_1)-\phi(P_2)-\phi(P_3)+\phi(P_4)=\phi(P_5)-\phi(P_6)-\phi(P_7)+\phi(P_8).\]
If $P_i$ has vertices $(x_i,y_i), (x_i,y_i+h_i), (x_i+w_i,y_i')$, and $(x_i+w_i,y_i'+h_i)$, then
\[\phi(P_i)=\phi(x_i,y_i)-\phi(x_i,y_i+h_i)-\phi(x_i+w_i,y_i')+\phi(x_i+w_i,y_i'+h_i).\]
Putting all that together we obtain a linear equation in the values of $\phi$ at the 32 vertices of the 8 parallelograms $P_i$, and each coefficient of the equation is $\pm 1$. Clearly if we multiply all the coefficients by -1, we obtain another such equation.

Now let $\be$ be a bihomomorphism from $G^2$ to a group $H$. Then in an obvious sense $\be$ respects $(Q_1,Q_2)$ as well. To be precise, we define 
\[\be(P_i)=\be(x_i,y_i)-\be(x_i,y_i+h_i)-\be(x_i+w_i,y_i')+\be(x_i+w_i,y_i'+h_i),\]
which by the bihomomorphism property equals $\be(w_i,h_i)$, and then we define
\[\be(Q_1)=\be(P_1)-\be(P_2)-\be(P_3)+\be(P_4),\]
which, by the bihomomorphism property again, equals $\be(w(Q_1),h(Q_1))$. Since $Q_2$ has the same width and height as $Q_1$, the claim follows.

Note that the coefficients are the same as they were above. We can apply this fact to the function, $(x,y)\mapsto x\otimes y$, which we can think of as the most general bihomomorphism on $G$. 

Let $V_{4;2}$ be the vector space of all 32-tuples $((a_1,b_2),\dots,(a_{32},b_{32}))$ of points that form second-order 4-arrangements in the obvious way: that is, the vertices of $P_i$ are $(a_{4(i-1)+j},b_{4(i-1)+j})$ for $j=1,2,3,4$, with the $P_i$ and their vertices ordered as above. We have found two sequences $\e\in\{-1,0,1\}^{32}$, one equal to minus the other, with the property that $\sum_i\e_ia_i\otimes b_i=0$ for every 32-tuple in $V_{4;2}$. The way we have ordered the 32-tuples, they are the Morse sequence and minus the Morse sequence. In the next lemma, we shall show that the only sequences in $\{-1,0,1\}^{32}$ with this property are the three multiples of the Morse sequence.

\begin{lemma} \label{uniquesigns}
Let $G=\F_p^n$, let $\e\in\{-1,0,1\}^{32}$ be a sequence such that $\sum_i\e_ia_i\otimes b_i=0$ for every 32-tuple $((a_i,b_i))_{i=1}^{32}$ in $V_{4;2}$. Then $\e$ is a multiple of the Morse sequence. Furthermore, if $\e$ is not a multiple of the Morse sequence, then the proportion of $((a_i,b_i))_{i=1}^{32}\in V_{4;2}$ such that $\sum_i\e_ia_i\otimes b_i=0$ is at most $2|G|^{-1}$. 
\end{lemma}

\begin{proof}
Observe first that for any vertical parallelogram $P=((x,y),(x,y+h),(x+w,y'),(x+w,y'+h))$ we can change $y$ and $y'$ without changing the width or height of $P$. But for a linear combination 
\[\e_1x\otimes y+\e_2x\otimes(y+h)+\e_3(x+w)\otimes y'+\e_4(x+w)\otimes(y'+h)\]
to be independent of $y$ and $y'$ for every $x$ we must have $\e_1=-\e_2$ and $\e_3=-\e_4$. Furthermore, if any of these conditions fail, then either $x=x+w=0$ or the proportion of pairs $(y,y')$ for which $\e_1x\otimes y+\e_2x\otimes(y+h)+\e_3(x+w)\otimes y'+\e_4(x+w)\otimes(y'+h)$ is equal to any particular element of $G\otimes G$ is at most $|G|^{-1}$, so the proportion of vertical parallelograms $P$ giving rise to any particular element is at most $2|G|^{-1}$. 

We can also vary $x$ without changing the width or height. But
\[\e_1x\otimes y-\e_1x\otimes(y+h)+\e_3(x+w)\otimes y'-\e_3(x+w)\otimes(y'+h)\]
is equal to
\[-\e_1x\otimes h-\e_3(x+w)\otimes h=-(\e_1+\e_3)x\otimes h-\e_3w\otimes h.\]
For this to be independent of $x$, we must have $\e_3+\e_1=0$ unless $h=0$. Therefore, the sequence $(\e_1,\e_2,\e_3,\e_4)$ must be a multiple of $(1,-1,-1,1)$ for independence to hold for all $(Q_1,Q_2)$, and if it does not hold, then no value is taken with probability greater than $2|G|^{-1}$. 

Assume that this property holds. For $i=1,2,\dots,8$, let $\eta_i$ be such that $\eta_i(1,-1,-1,1)=(\e_{4(i-1)+1},\e_{4(i-1)+2},\e_{4(i-1)+3},\e_{4(i-1)+4})$. If we now write a typical second-order 4-arrangement as $(P_1,\dots,P_8)$, then 
\[\sum_{i=1}^{32}\e_ia_i\otimes b_i=\sum_{i=1}^8\eta_iw_i\otimes h_i\]
where $w_i$ and $h_i$ are the width and height of $P_i$. 

Essentially the same argument as above now shows that $(\eta_1,\eta_2,\eta_3,\eta_4)$ and $(\eta_5,\eta_6,\eta_7,\eta_8)$ are both multiplies of $(1,-1,-1,1)$. That is because the points $(w_i,h_i)$ form the vertices of two vertical parallelograms of the same width and height, and as above, if we vary these vertical parallelograms while preserving their widths and heights, we obtain the stated result. And as before we find that the proportion of arrangements for which the sum can take a particular value is at most $2|G|^{-1}$ if one of these equalities does not hold.

Writing $(\eta_1,\dots,\eta_4)=\g_1(1,-1,-1,1)$ and $(\eta_5,\dots,\eta_8)=\g_2(1,-1,-1,1)$, we find that the sum is $\g_1w(Q_1)\otimes h(Q_1)+\g_2w(Q_2)\otimes h(Q_2)$. But $Q_1$ and $Q_2$ have the same width and height, so $\g_1+\g_2=0$ if these are both non-zero. Finally, the proportion of width/height pairs for which one of the width and height is zero is at most $2|G|^{-1}$, so the proof is complete.
\end{proof}

\begin{lemma} \label{densification}
Let $\d,\eta>0$, let $G=\F_p^n$, let $A\subset G^2$, and let $\phi:A\to G$ be a function that respects at least $\d|G|^{32}$ second-order 4-arrangements in $A$. Then $A$ has a subset $A'$ that contains at least $2^{-2^{37}(\log(\eta^{-1})+\log(\d^{-1}))}|G|^{32}$ second-order 4 arrangements such that the proportion of its arrangements that are respected by $\phi$ is at least $1-\eta$.
\end{lemma}

\begin{proof}
The idea is to use a dependent random selection, choosing the dependencies so as to favour the arrangements that are respected by $\phi$. To do this, we choose a positive integer $k$, and then we choose independent random elements $s_1,\dots,s_k$ of $G$ and independent random $n\times n$ matrices $M_1,\dots,M_k$ over $\F_p$. Having made these choices, we then choose each point $(x,y)\in A$ independently with probability given by the Riesz product
\[2^{-k}\prod_{i=1}^k\Bigl(1+\cos\bigl(\frac{2\pi}p(s_i.\phi(x,y)+x.M_iy)\bigr)\Bigr).\]
Thus, for each fixed choice of the $s_i$ and $M_i$ the elements of $A$ are selected independently, but overall the distribution is far from independent.

Setting $\omega=\exp(2\pi i/p)$ as usual, we can rewrite the above probability as
\[4^{-k}\prod_{i=1}^k(2+\omega^{s_i.\phi(x,y)+x.M_iy}+\omega^{-(s_i.\phi(x,y)+x.M_iy)}).\]
Suppose now that $((a_1,b_1),\dots,(a_{32},b_{32}))$ is a second-order 4-arrangement in $A$. The probability that all 32 points are chosen is
\[4^{-32k}\E_{s_1,\dots,s_k}\E_{M_1,\dots,M_k}\prod_{i=1}^{32}\prod_{j=1}^k(2+\omega^{s_j.\phi(a_i,b_i)+a_i.M_jb_i}+\omega^{-(s_j.\phi(a_i,b_i)+a_i.M_jb_i)}).\]
This is equal to
\[4^{-32k}\Bigl(\E_s\E_M\prod_{i=1}^{32}(2+\omega^{s.\phi(a_i,b_i)+a_i.Mb_i}+\omega^{-(s.\phi(a_i,b_i)+a_i.Mb_i)})\Bigr)^k.\]
Let us now think about the product inside the brackets. If we expand it out, we obtain a sum of terms, each of which is of the form 
\[2^r\omega^{s.\sum_{i=1}^{32}\e_i\phi(a_i,b_i)}\omega^{\sum_{i=1}^{32}\e_ia_i.Mb_i}\]
for some positive integer $r$ and some $\e\in\{-1,0,1\}^{32}$.

The sum $\sum_{i=1}^{32}\e_ia_i.Mb_i$ can be thought of as the matrix inner product (over $\F_p$) of the matrix $\sum_{i=1}^{32}\e_ia_i\otimes b_i$ with the matrix $M$. Therefore, the average over all $M$ of the second term in the product above is 1 if $\sum_i\e_ia_i\otimes b_i=0$ and 0 otherwise.

By Lemma \ref{uniquesigns}, for each choice of $\e$ that is not a multiple of the particular choice calculated earlier (which happens to coincide with the Morse sequence of length 32) the proportion of elements $((a_i,b_i))_{i=1}^{32}$ of the vector space $V_{4;2}$ with $\sum_i\e_ia_i\otimes b_i=0$ is at most $2|G|^{-1}$. Therefore, for all but a proportion $2^{32}|G|^{-1}$ of second-order 4-arrangements, the average $\E_M\omega^{\sum_{i=1}^{32}\e_ia_i.Mb_i}$ is 1 if $\e$ is a multiple of the Morse sequence and 0 otherwise.

If $\e$ is a multiple of the Morse sequence, then the average $\E_s\omega^{s.\sum_{i=1}^{32}\e_i\phi(a_i,b_i)}$ is 1 if $\e=0$ or $\phi$ respects the 4-arrangement (since this is true if and only if the sum over $i$ is zero) and 0 otherwise. 

It follows that if $\phi$ does not respect the second-order 4-arrangement, then 
\[\E_s\E_M\prod_{i=1}^{32}(2+\omega^{s.\phi(a_i,b_i)+a_i.Mb_i}+\omega^{-(s.\phi(a_i,b_i)+a_i.Mb_i)})=2^{32}\]
and otherwise it equals $2^{32}+2$. 

Therefore, given a second-order 4-arrangement in $A$, this random selection procedure will choose all its points with probability $2^{-32k}$ if it is not respected by $\phi$, and $2^{-32k}(1+2^{-31})^k$ if it is.

Let $A'$ be the random set chosen by the selection procedure. Let $X$ be the number of second-order 4-arrangements in $A'$ that are respected by $\phi$ and let $Y$ be the number that are not. Since there are $\d|G|^{32}$ arrangements in $A$ that are respected by $\phi$ and at most $|G|^{32}$ that are not, we have that $\E X\geq 2^{-32k}(1+2^{-31})^k\d|G|^{32}$ and $\E Y\leq 2^{-32k}|G|^{32}$. If we choose $k$ such that $(1+2^{-31})^k\geq 2\eta^{-1}\d^{-1}$, then we may conclude that $\E(X-\eta^{-1}Y)\geq 2^{-32k}\eta^{-1}|G|^{32}$. The choice $k=2^{32}(\log(\eta^{-1})+\log(\d^{-1}))$ works, and from this the lemma follows. (The extra factor $\eta^{-1}$ does not make enough difference to be worth keeping.) 
\end{proof}

We now collect together the results of the section into one single statement, which for convenience we give in both qualitative and quantitative form. 

\begin{lemma} \label{collectsection3}
For any $c_1>0$ there exists $c_5>0$ with the following property. Let $G=\F_p^n$, let $f:G\to\C$ be a bounded function, let $A\subset G^2$ be a set of density $\a$, and let $\phi:A\to G$ be a function such that $|\widehat{\partial_{a,b}f}(\phi(a,b))|\geq c_1$ for every $(a,b)\in A$. Then $A$ has a subset $A'$ of density at least $c_6$ such that the restriction of $\phi$ to each row of $A$ is a Freiman homomorphism, and the proportion of the second-order 4-arrangements in $A'$ that are respected by $\phi$ is at least $1-\eta$. 

Moreover, we may take $c_6$ to be $2^{-2^{51}}p^{-2^{44}}\eta^{2^{37}}c_1^{2^{53}}$.
\end{lemma}

\begin{proof}
By Corollary \ref{homsinrows} $A$ has a subset $A_1$ of density at least $c_2=p^{-1}2^{-113}c_1^{170}$ such that the restrictions of $\phi$ to the rows of $A$ are Freiman homomorphisms. This property is preserved when we pass to further subsets.

By Lemma \ref{somearrangements}, $\phi$ respects at least $c_3=c_2^{16}c_1^{24}|G|^8$ 4-arrangements of $A_1$.

By Corollary \ref{manyarr2s}, $\phi$ respects at least $c_4=c_3^8c_2^{-12}|G|^{32}$ second order 4-arrangements of $A_1$.

By Lemma \ref{densification} $A_1$ has a subset $A'$ that contains at least $c_5=2^{-2^{37}(\log(\eta^{-1})+\log(c_4^{-1}))}|G|^{32}$ second-order 4-arrangements such that the proportion that are respected by $\phi$ is at least $1-\eta$. 

A set of density $\d$ cannot contain more than $\d^{11}|G|^{32}$ second-order 4-arrangements, since 11 of its points can be chosen independently. (A quick way of seeing this is to observe that we can choose three of the four parallelogram widths independently and also the first point in each parallelogram.) It follows that $A'$ has density at least $c_6=c_5^{1/11}\geq 2^{-2^{34}(\log(\eta^{-1})+\log(c_4^{-1}))}$.

The statement of the lemma now follows from a back-of-envelope calculation, which we omit.
\end{proof}

We shall choose $\eta$ to be an absolute constant, so the important aspect of the bound is that $c_6$ has a power dependence on $c_1$. 

\section{Mixed convolutions and a bilinear Bogolyubov-type method}

Let us now see why second-order vertical parallelograms arise naturally. Given a finite Abelian group $G$ and functions $f,g:G\to\C$, it is often useful to consider a variant of the usual convolution operator, namely the operator $\dc$, defined by the formula
\[f\dc g(x)=\E_uf(u)\overline{g(u-x)}=\E_{u-v=x}f(u)\overline{g(v)}.\]
It is easy to check the modified convolution law $\widehat{f\dc g}(\chi)=\hat f(\chi)\overline{\hat g(\chi)}$. 

Now suppose that we are given a function $f:G^2\to\C$. Define cross-sectional functions $f_{x\bullet}$ and $f_{\bullet y}$ by $f_{x\bullet}(u)=f(x,u)$ and $f_{\bullet y}(v)=f(v,y)$. We define the \emph{vertical convolution} of two functions $f,g:G^2\to\C$ to be the function
\[(f\vc g)(x,h)=f_{x\bullet}\dc g_{x\bullet}(h)=\E_yf(x,y)\overline{g(x,y-h)}.\]
Similarly, define the \emph{horizontal convolution} of $f$ and $g$ to be the function
\[(f\hc g)(w,y)=f_{\bullet y}\dc g_{\bullet y}(w)=\E_xf(x,y)\overline{g(x-w,y)}.\]
Finally, define the \emph{mixed convolution} $\mc(f_1,f_2,f_3,f_4)$ to be the function $(f_1\vc f_2)\hc(f_3\vc f_4)$. It is easy to check that
\begin{align*}
\mc(f_1,f_2,f_3,f_4)(w,h)&=\E_{x,y,y'}f_1(x,y)\overline{f_2(x,y-h)f_3(x-w,y')}f_4(x-w,y'-h)\\
&=\E_{x,y,y'}f_4(x,y)\overline{f_3(x,y+h)f_2(x+w,y')}f_1(x+w,y'+h).\\
\end{align*}
In particular, if $f=\b1_A$ for a subset $A$ of $G^2$, then $\mc(f,f,f,f)(w,h)$ is the probability that a random vertical parallelogram of width $w$ and height $h$ has all four of its points in $A$. From this it follows that the average value of $\mc(f,f,f,f)$ is the \emph{vertical parallelogram density} of $A$: that is, the probability that a random vertical parallelogram has all its points in $A$. We shall use the notation $\mc f$ for $\mc(f,f,f,f)$. 

It is now easy to describe where second-order vertical parallelograms enter the picture: they are what is counted if one applies the mixed convolution twice. More precisely, if $f=\b1_A$, then $\mcmc f(w,h)$ is the probability that a second-order vertical parallelogram of width $w$ and height $h$ has all its points in $A$. We also have that the number of second-order 4-arrangements in $A$ is $|G|^{32}\bigl\|\mcmc f\bigr\|_2^2$.

\subsection{A measure of well-definedness}

In the next subsection, we shall prove a lemma that is closely related to a simple fact about Freiman homomorphisms, namely that a Freiman homomorphism $\phi$ of order $2k$ on a subset $A$ of an Abelian group induces a Freiman homomorphism $\psi$ of order $k$ on the difference set $A-A$ (or indeed on the sumset $A+A$, but the difference set is more closely analogous to what we are doing here). This is useful, because typically difference sets have more additive structure than arbitrary sets. 

The proof of the lemma will be trivial once the appropriate definitions are in place. We shall again make use of functions from $G^2$ to $\cA$, the group algebra of $G$. (Actually, all we really care about is the subset of $\cA$ that consists of non-negative real-valued functions.) We do this because it gives us a convenient way of handling functions that are almost, but not quite, well defined. A similar approach was taken in \cite{gowersconlon}.

To explain it first in a simpler context, suppose that we have a finite Abelian group $G$, a subset $A\subset G$, and a Freiman homomorphism $\phi:A\to G$. Then $\phi$ induces a well-defined function $\psi$ on the difference set $A-A$, namely the function $\psi(a-b)=\phi(a)-\phi(b)$. The condition that this is well-defined is equivalent to the statement that if $a-b=c-d$, then $\phi(a)-\phi(b)=\phi(c)-\phi(d)$. 

Suppose now that $\phi$ satisfies the weaker property that $\phi(a)-\phi(b)=\phi(c)-\phi(d)$ only for a proportion $1-\eta$ of the additive quadruples in $A$ (that is, the quadruples $(a,b,c,d)\in A^4$ such that $a-b=c-d$). Let us call such a map a $(1-\eta)$-homomorphism. We would like to say that $\phi$ induces an ``almost well-defined" function on the difference set, in some suitable sense. 

If one wants an exact equivalence, then the appropriate sense is as follows. Let $\mu$ be a non-negative function defined on $G$ and let $\psi:G\to\cA$. Suppose also that each $\psi(x)$ is a probability measure: that is, a non-negative function that sums to 1. Then we shall say that $\psi$ is $(1-\eta)$-\emph{well defined with respect to $\mu$} if 
\[\E_x\|\mu(x)\psi(x)\|_2^2\geq(1-\eta)\E_x\mu(x)^2.\]

To see why this definition makes sense, let us regard the function $\phi$ as taking values not in $G$ but in $\cA$, by composing it with the map that takes each $g\in G$ to the function $\d_g\in\cA$. And let us extend $\phi$ to the whole of $G$ by setting $\phi(x)=0$ when $x\notin A$. (This $0$ is the zero of $\cA$, not the identity of $G$.) When $a,b,c,d\in A$, the condition $\phi(a)-\phi(b)=\phi(c)-\phi(d)$ becomes the condition $\langle\phi(a)\phi(b)^*,\phi(c)\phi(d)^*\rangle=1$, since
\[\langle\phi(a)\phi(b)^*,\phi(c)\phi(d)^*\rangle=\langle\phi(a-b),\phi(c-d)\rangle,\]
which is 1 if $a-b=c-d$ and 0 otherwise. The condition that $\phi$ is a $(1-\eta)$-homomorphism can therefore be written as
\[\E_{a-b=c-d}\langle\phi(a)\phi(b)^*,\phi(c)\phi(d)^*\rangle\geq(1-\eta)\E_{a-b=c-d}\b1_A(a)\b1_A(b)\b1_A(c)\b1_A(d).\]
The left-hand side of the above inequality can be rewritten as 
\[\E_u\|\E_{a-b=u}\phi(a)\phi(b)^*\|^2.\]
Now let us define a convolution $\phi*\phi^*:G\to\cA$ in a natural way, namely,
\[\phi*\phi^*(u)=\E_{a-b=u}\phi(a)\phi(b)^*.\]
Then we can rewrite the left-hand side further as $\|\phi*\phi^*\|_2^2$.

Now we set $\mu$ to be $\b1_A*\b1_{-A}$ and $\psi(u)$ to be $\mu(u)^{-1}\phi*\phi^*(u)$. Then the expression becomes $\|\mu\psi\|_2^2$, and the inequality becomes
\[\|\mu\psi\|_2^2\geq(1-\eta)\|\mu\|_2^2,\]
or, to expand it slightly,
\[\E_x\|\mu(u)\psi(u)\|_2^2\geq(1-\eta)\E_u\mu(u)^2.\]
This is saying that $\phi*\phi^*$ is $(1-\eta)$-well defined with respect to $\b1_A*\b1_{-A}$.

As one final remark, note that each $\psi(x)$ is a probability distribution on $G$: the value of $\psi(u)(v)$ is the probability that $\phi(x)\phi(y)^*=\d_v$ (which corresponds to the statement $\phi(x)-\phi(y)=v$ for the original function from $G$ to $G$) given that $x-y=u$ and $x,y\in A$. Since $\|\psi\|_2^2\leq\|\psi\|_1\|\psi\|_\infty=\|\psi\|_\infty$, the only way for the inequality above to hold is if for most $u$ (with respect to the measure $\mu$) there is one value of $\phi(x)\phi(y)^*$ that predominates. It is in this sense that we are measuring how well defined $\psi$ is.

\subsection{From second-order 4-arrangements on an arbitrary set to first-order 4-arrangements on a structured set}

Now let us obtain a similar statement in the context of mixed convolutions of two-variable functions. Given a finite Abelian group $G$, a subset $A\subset G^2$, and a function $\phi:A\to G$, we say that it is a \emph{bihomomorphism} if its ``mixed derivative" well defined, meaning that
\[\phi(x,y_1)-\phi(x,y_1+h)-\phi(x+w,y_2)+\phi(x+w,y_2+h)\]
depends on $w$ and $h$ only. We remark that the use of the word ``bihomomorphism" is slightly misleading, since one can add to a bihomomorphism an arbitrary function that depends on the first variable only and it will remain a bihomomorphism. As we mentioned in the overview of the proof, this is essentially the same phenomenon as the familiar fact that the partial derivative of a function with respect to one variable determines the function only up to an arbitrary function of the other variables. It will turn out not to matter later, but it does mean that it will be necessary at some point to prove that the one-variable function we end up with behaves well.

As in the linear case, we now want to define an ``almost bihomomorphism", and we do this once again by replacing functions into $G$ by functions into $\cA$. The precise definition we choose is as follows. Let $\Sigma(\cA)$ be the subset of $\cA$ that consists of all non-negative functions that sum to 1. (We use the letter $\Sigma$ to stand for ``simplex".) Then if $\mu$ is a non-negative function defined on $G^2$, $\phi:G^2\to\cA$, and $0\leq\eta\leq 1$, we say that $\phi$ is a $(1-\eta)$-\emph{bihomomorphism with respect to} $\mu$ if 
\[\E_{w,h}\bigl\|\E_{P\in\cP(w,h)}\mu(P)\phi(P)\bigr\|_2^2\geq(1-\eta)\E_{w,h}\bigl|\E_{P\in\cP(w,h)}\mu(P)\bigr|^2.\]
Here, if $P$ has vertices $(x,y_1), (x,y_1+h), (x+w,y_2)$ and $(x+w,y_2+h)$, then 
\[\mu(P)=\mu(x,y_1)\mu(x,y_1+h)\mu(x+w,y_2)\mu(x+w,y_2+h)\]
and
\[\phi(P)=\phi(x,y_1)\phi(x,y_1+h)^*\phi(x+w,y_2)^*\phi(x+w,y_2+h).\]
Note that $\E_{w,h}\bigl|\E_{P\in\cP(w,h)}\mu(P)\bigr|^2$ is the \emph{4-arrangement density} of $\mu$: that is, the expected product of $\mu$ over the eight vertices of a random 4-arrangement. The left-hand side is similar, except that now we multiply this product by a quantity that measures how close $\phi(P)\phi(P')^*$ is to a delta-function when $P$ and $P'$ are the two parts of the random 4-arrangement.

We also need something similar for second-order 4-arrangements. We say that $\phi:G\to\Sigma(\cA)$ is a \emph{second-order $(1-\eta)$-bihomomorphism with respect to} $\mu$ if 
\[\E_{w,h}\bigl\|\E_{Q\in\cQ(w,h)}\mu(Q)\phi(Q)\bigr\|_2^2\geq(1-\eta)\E_{w,h}\bigl|\E_{Q\in\cQ(w,h)}\mu(Q)\bigr|^2,\]
where the definitions of $\mu(Q)$ and $\phi(Q)$ are similar to the definitions of $\mu(P)$ and $\phi(P)$ but now the products are over the vertices of second-order vertical parallelograms (with appropriate adjoints taken in the case of $\phi$). Roughly speaking, $\phi$ is a second-order $(1-\eta)$-bihomomorphism for small $\eta$ if the sum of the values of $\phi$ over the vertices of a second-order vertical parallelogram $Q$, with appropriate signs, (or the product with appropriate adjoints when we are talking about functions to $\Sigma(\cA)$), mostly depends only on the width and height of $Q$.

\begin{lemma} \label{isoonmixedconv}
Let $G$ be a finite Abelian group, let $\cA$ be its group algebra, let $\mu$ be a non-negative function on $G^2$, and let $\phi:G^2\to\Sigma(\cA)$ be a second-order $(1-\eta)$-bihomomorphism with respect to $\mu$. Then $\phi$ induces a function $\psi$ that is a $(1-\eta)$-bihomomorphism with respect to the function~$\mc\mu$.
\end{lemma}

\begin{proof}
The hypothesis can be written
\begin{align*}
\E_{w,h}\bigl\|\mathop{\E}_{(P_1,\dots,P_4)\in\cQ(w,h)}\mu(P_1)\mu(P_2)\mu(P_3)\mu(P_4)&\phi(P_1)\phi(P_2)^*\phi(P_3)^*\phi(P_4)\bigr\|_2^2\\
&\geq(1-\eta)\E_{w,h}\bigl|\mathop{\E}_{(P_1,\dots,P_4)\in\cQ(w,h)}\mu(P_1)\mu(P_2)\mu(P_3)\mu(P_4)\bigr|^2,\\
\end{align*}
where $\mu(P_i)$ is the product of $\mu(x,y)$ over the four vertices of $P_i$. 

Now we define $\psi$ by setting $\E_{P\in\cP(w',h')}\mu(P)\phi(P)$ to equal $(\mc\mu)(w',h')\psi(w',h')$ for each pair $(w',h')$, noting that the $\ell_1$-norm of the left-hand side is $(\mc\mu)(w',h')$ and therefore that $\psi\in\Sigma(\cA)$. (If $(\mc\mu)(w',h')=0$, then we choose an arbitrary probability distribution and call it $\psi(w',h')$.) Note that $\psi$ has a simple probabilistic interpretation: if we choose a vertical parallelogram $P\in\cP(w',h')$ at random with probability proportional to $\mu(P)$, then $\psi(w',h')$ is the expectation of $\phi(P)$.

Since $(P_1,\dots,P_4)\in\cQ(w,h)$ if and only if the widths and heights of $P_1,\dots,P_4$ form a vertical parallelogram of width $w$ and height $h$, we have that
\[\bigl\|\mathop{\E}_{(P_1,\dots,P_4)\in\cQ(w,h)}\mu(P_1)\mu(P_2)\mu(P_3)\mu(P_4)\phi(P_1)\phi(P_2)^*\phi(P_3)^*\phi(P_4)\bigr\|_2^2=\bigl\|\E_{P\in\cP(w,h)}(\mc\mu)(P)\psi(P)\bigr\|_2^2\]
for every $(w,h)$. For similar reasons, 
\[\mathop{\E}_{(P_1,\dots,P_4)\in\cQ(w,h)}\mu(P_1)\mu(P_2)\mu(P_3)\mu(P_4)=(\mc\mu)(w,h).\]
Thus, our assumption is equivalent to the statement that
\[\E_{w,h}\bigl\|\E_{P\in\cP(w,h)}(\mc\mu)(P)\psi(P)\bigr\|_2^2\geq(1-\eta)\E_{w,h}\bigl|(\mc\mu)(w,h)\bigr|^2,\]
or in other words that $\psi$ is a $(1-\eta)$-bihomomorphism with respect to $\mc\mu$.
\end{proof}

\iftrue
\else
In this subsection we prove a lemma that is closely related to a simple fact about Freiman homomorphisms, namely that a Freiman homomorphism $\phi$ of order $2k$ on a subset $A$ of an Abelian group induces a Freiman homomorphism $\psi$ of order $k$ on the difference set $A-A$ (or indeed on the sumset $A+A$, but the difference set is more closely analogous to what we are doing here). This is useful, because typically difference sets have more additive structure than arbitrary sets. 

We begin with a very simple technical lemma.

\begin{lemma}\label{convolution}
Let $X_1,\dots,X_8$ be probability distributions on a finite Abelian group $G$. Then 
\[\|X_1*\dots*X_8\|_\infty\leq(\|X_1\|_\infty\dots\|X_8\|_\infty)^{1/8}.\]
\end{lemma}

\begin{proof}
This is really about convolutions of two distributions. We have for each $g\in G$ that
\[|X_1*X_2(g)|\leq\|X_1\|_2\|X_2\|_2,\]
by Cauchy-Schwarz. But for a probability distribution $X$, $\|X\|_2\leq\|X\|_1^{1/2}\|X\|_\infty^{1/2}=\|X\|_\infty^{1/2}$, so $\|X_1*X_2\|_\infty\leq\|X_1\|_\infty^{1/2}\|X_2\|_\infty^{1/2}$. Iterating this result three times proves the lemma.
\end{proof}

To state the lemma, we need some definitions. Let $G$ be a finite Abelian group and let $\mu$ be a measure (that is, a non-negative function) on $G^2$. Given a structure $S$ such as a vertical parallelogram or a 4-arrangement of order 2, write $\mu(S)$ for the product of $\mu(x,y)$ over all the points $(x,y)$ that make up $S$. For example, if $S$ is the vertical parallelogram 
\[((x,y),(x,y+h),(x+w,y'),(x+w,y'+h)),\]
then
\[\mu(S)=\mu(x,y)\mu(x,y+h)\mu(x+w,y')\mu(x+w,y'+h).\]
Now let $\phi$ be a function that is defined for every $(x,y)$ such that $\mu(x,y)\ne 0$. Let $k=1$ or $2$. Then $\phi$ is a $(1-\eta)$-\emph{bihomomorphism of order} $k$ \emph{on} $\mu$ if 
\[\sum\{\mu(S):\phi\ \hbox{respects}\ S\}\geq(1-\eta)\sum\mu(S)\]
where $S$ ranges over all 4-arrangements of order $k$. (Clearly this definition can be generalized to other kinds of arrangements and higher $k$, but these are the ones we shall need here.) If $A$ is a set, we say that $\phi$ is a $(1-\eta)$-bihomomorphism of order $k$ on $A$ if it is a $(1-\eta)$-bihomomorphism on $\mu=\b1_A$.

We remark that our use of the word ``bihomomorphism", though convenient, is slightly misleading, for the following reason. Let $G$ be a finite Abelian group and let $\phi:G^2\to H$ be any function that is a Freiman homomorphism in each variable separately. It is easy to check that $\phi$ is a 1-bihomomorphism of all orders. If we now take an arbitrary function $\lambda:G\to H$ and define $\tilde\phi(x,y)$ to be $\phi(x,y)+\lambda(x)$, then $\tilde\phi(x,y+h)-\tilde\phi(x,y)=\phi(x,y+h)-\phi(x,y)$ for every $x,y,h$, so $\tilde\phi$ is also a 1-bihomomorphism of all orders. Slightly surprisingly, this will turn out not to matter later: roughly speaking, we do not mind ``perturbations" as long as they are functions of one variable. (These functions of one variable are analogous to the arbitrary function of one variable that one has to introduce after solving an equation of the form $\frac{\partial f}{\partial x}=g(x,y)$, where the general solution is of the form $f(x,y)=u(x,y)+v(y)$.)

Given a vertical parallelogram $P$, we shall write $d(P)$ for the pair $(w(P),h(P))$. (The letter d is for ``dimensions".)

\begin{lemma}
Let $G$ be a finite Abelian group, let $A\subset G^2$, let $0<\eta\leq 1/100$ and let $\phi:A\to G$ be a $(1-\eta)$-bihomomorphism of order $2$ on $A$. Then $\phi$ induces a $(1-\eta^{1/2})$-bihomomorphism $\psi$ of order 1 on $\mc\b1_A$. 
\end{lemma}

\begin{proof}
For each pair $(w,h)$, let $\psi(w,h)$ be the most frequently occurring value of $\phi(P)$ over all vertical parallelograms $P$ in $A$ of width $w$ and height $h$, with an arbitrary choice of winner in the case of a tie, and for each $(w,h)$ let the proportion of vertical parallelograms $P$ of width $w$ and height $h$ such that $\phi(P)=\psi(w,h)$ be $\theta(w,h)$. 

Let us write a typical 4-arrangement $S$ as $(s_1,\dots,s_8)$, where each $s_i\in G^2$. Then the statement that $\phi$ is a $(1-\eta)$-bihomomorphism of order 2 on $A$ can be expressed by means of the following inequality, where $(\e_1,\dots,\e_8)$ is the Morse sequence of length 8.
\begin{align*}
\sum_S\bigl|\{(P_1,\dots,P_8):d(P_i)=s_i,\sum\e_i\phi(P_i)=0\}\bigr|
\geq(1-\eta)\sum_S\bigl|\{(P_1,\dots,P_8):d(P_i)&=s_i\}\bigr|\\
\end{align*}
Now for any $S$ the size of the set $\{(P_1,\dots,P_8):d(P_i)=s_i\}$ is $(\mc\b1_A)(S)|G|^{24}$, while the size of the set $\{(P_1,\dots,P_8):d(P_i)=s_i,\sum\e_i\phi(P_i)=0\}$ is the same but multiplied by the probability that $\sum_i\e_i\phi(P_i)=0$ given that $d(P_i)=s_i$. Let us write this probability as $p(S)$, so we have that
\[\sum_S(\mc\b1_A)(S)p(S)\geq(1-\eta)\sum_S(\mc\b1_A)(S).\]
Therefore, by averaging, 
\[\sum\{(\mc\b1_A)(S):p(S)\geq 1-\eta^{1/2}\}\geq(1-\eta^{1/2})\sum_S(\mc\b1_A)(S).\]

Fix a 4-arrangment $S$ and let $X_i$ be the probability distribution on $G$ obtained by choosing a random vertical parallelogram $P_i$ in $A$ with $d(P_i)=s_i$ and taking the element $\e_i\phi(P_i)$ of $G$. Note that $\|X_i\|_\infty=\theta_i$, where $\theta_i=\theta(w(P_i),h(P_i))$. Then $p(S)$ is the probability that $\sum_i\e_i\phi(P_i)=0$ when the $P_i$ are chosen according to the distributions $X_i$, which is the value of their convolution at 0. By Lemma \ref{convolution}, this is at most $\prod_{i=1}^8\theta_i^{1/8}$. Note also that the probability that $\phi(P_i)=\psi(s_i)$ for each $i$ is $\prod_{i=1}^8\theta_i$, so if
\[p(S)+\prod_{i=1}^8\theta_i>1,\]
then $\sum_{i=1}^8\e_i\psi(s_i)=0$. Thus, it is sufficient if $p(S)+p(S)^8>1$. Since $\eta\leq 1/100$, this holds if $p(S)\geq 1-\eta^{1/2}$. 

Therefore, we may conclude that
\[\sum\{(\mc\b1_A)(S):\psi\ \hbox{respects}\ S\}\geq(1-\eta^{1/2})\sum_S(\mc\b1_A)(S),\]
which is what the lemma claims.
\end{proof}
\fi

\subsection{A motivating example}

A convolution $f*g$ of two bounded functions $f$ and $g$ defined on $\F_p^n$ has ``linear structure" in the following sense: there is a linear map $T:\F_p^n\to\F_p^k$ with $k$ not too large such that $f*g$ is approximately constant on the level sets $L_u=\{x:T_x=u\}$ of $T$ in the following sense: if $P_T$ is the averaging projection with respect to these level sets (that is, the value at $x$ of a function is replaced by the average over the level set that contains $x$), then $P_T(f*g)$ will be close in the $L_2$ norm to $f*g$. The key point is that if we want $\|P_T(f*g)-f*g\|_2$ to be at most $\e$, then $k$ depends on $\e$ only.

Our aim in the rest of this section will be to obtain a bilinear generalization this result: we shall show that if $f_1,f_2,f_3,f_4$ are bounded functions defined on $\F_p^n\times\F_p^n$ and $F$ is the mixed convolution $\mc(f_1,f_2,f_3,f_4)$, then there is a bilinear function $\be:\F_p^n\times\F_p^n\to\F_p^k$ with $k$ not too large such that $F$ is close to $P_\be F$ in $L_2$, where $P_\be$ is the averaging projection with respect to the level sets of $\be$. 

In the linear case, one can do better by convolving more often. If, for example, $f_1,f_2,f_3,f_4$ are bounded functions on $\F_p^n$ and we take the convolution $F=(f_1\dc f_2)\dc(f_3\dc f_4)$, then there is a low-rank linear function $T$ such that $F$ and $P_TF$ are close not just in $L_2$ but \emph{uniformly}. 

This uniformity is useful in a number of applications, but it is sometimes possible to make do with just an $L_2$ approximation, as we shall see in this paper. We are not doing that simply because we can, but because in a certain sense we have no choice. The following simple example demonstrates that a double mixed convolution of two bounded functions does not have to be uniformly close to a function that is constant on the level sets of a bilinear function to $\F_p^k$ for some small $k$. We shall indicate why the example works, but leave the full proof as an exercise for the reader.

Define $f(x,y)=\omega^{x.y}+\omega^{2x.y}$. Then 
\begin{align*}
(f\vc f)(x,h)&=\E_y(\omega^{x.y}+\omega^{2x.y})(\omega^{-x.(y-h)}+\omega^{-2x.(y-h)})\\
&=\omega^{x.h}+\omega^{2x.h}+\E_y(\omega^{x.(y-h)}+\omega^{-x.(y-h)}).\\
\end{align*}
The last term is zero if $x\ne 0$ and 2 otherwise.

Thus, essentially we obtain the function we started with, but there is a perturbation owing to the different behaviour when $x=0$. This perturbation is tiny in $L_2$ but it is not tiny in $L_\infty$. 

If we now set $g=f\vc f$ and calculate $g\hc g$, we find that much the same thing happens. The perturbation turns out to have very little effect on the result of the convolution, so what we are left with is again extremely close to the original function, but this time there is a significant (in $L_\infty$) difference when $y=0$. 

This behaviour persists however often one takes a mixed convolution, and because the perturbation is large just for one value of $x$ or $y$, it is not possible to find a bilinear function $\be:\F_p^n\times\F_p^n\to\F_p^k$ with $k$ small and with the mixed convolution uniformly approximating its average over the level sets: the level sets are much too large.

It is worth pointing out that this example seems to be quite robust, in the sense that we cannot get round it with a simple modification of the definition of a bilinear convolution. For instance, instead of first doing a vertical convolution and then a horizontal one, one might consider doing them both together, defining a function
\[g(w,h)=\E_{x,y}f_1(x,y)\overline{f_2(x,y-h)f_3(x-w,y)}f_4(x-w,y-h).\]
But this kind of modification does not seem to help.

As it happens, we are able to take into account the phenomenon above and give a more complicated description of double mixed convolutions that is valid up to a uniform approximation. We shall present that result in another paper, where it will be used to give a different proof of the inverse theorem \cite{GM2}. For the approach of this paper, it is more convenient to use a simpler description and make do with $L_2$ approximations.

\subsection{$L_2$ approximation in the linear case}

While many additive combinatorialists are aware that a convolution of two bounded functions can be approximated in $L_2$ by a ``nice" function, there does not seem to be a standard reference, so we briefly give the proof here. 

Recall the definition of a \emph{Bohr set} in a finite Abelian group $G$. If $K$ is a set of characters on $G$, then $B(K;\d)$ is the set $\bigcap_{\chi\in K}\{x: |\chi(x)-1|\leq\d\}$. Also, if $B$ is a set of density $\be$, then we define its \emph{characteristic measure} to be the function $\mu_B$ that takes the value $\be^{-1}$ on $B$ and 0 elsewhere.

\begin{lemma}
Let $G$ be a finite Abelian group, let $f,g:G\to\C$ be bounded functions, and let $\e>0$. Then there is a Bohr set $B(K;\e/4)$ with $|K|\leq 8/\e$ such that $\|f*g-\mu_B*f*g\|_2\leq\e$.
\end{lemma}

\begin{proof}
Let $\g=\e/8$ and $\d=\e/4$, and let $K$ be the set of all characters $\chi\in G^*$ such that $|\hat f(\chi)\hat g(\chi)|\geq\g$. Since $\sum_\chi|\hat f(\chi)\hat g(\chi)|\leq\|f\|_2\|g\|_2$ by Cauchy-Schwarz and Parseval's identity, and $\|f\|_2$ and $\|g\|_2$ are both at most 1, it follows that $|K|\leq\g^{-1}$.

Now let $B$ be the Bohr set $B(K;\d)$. Note first that for each $\chi\in K$ we have that
\[|1-\mu_B(\chi)|=|1-\E_{x\in B}\chi(x)|\leq\E_{x\in B}|1-\chi(x)|\leq\d.\]
Therefore,
\begin{align*}
\|f*g-\mu_B*f*g\|_2^2&=\sum_\chi|\hat f(\chi)|^2|\hat g(\chi)|^2|1-\widehat{\mu_B}(\chi)|^2\\
&=\sum_{\chi\in K}|\hat f(\chi)|^2|\hat g(\chi)|^2|1-\widehat{\mu_B}(\chi)|^2+\sum_{\chi\notin K}|\hat f(\chi)|^2|\hat g(\chi)|^2|1-\widehat{\mu_B}(\chi)|^2\\
&\leq\d^2|K|+4\g\sum_{\chi\in K}|\hat f(\chi)||\hat g(\chi)|\\
&\leq\d^2\g^{-1}+4\g,\\
\end{align*}
where in the last step we used Cauchy-Schwarz and the fact that $\|\hat f\|_2$ and $\|\hat g\|_2$ are both at most 1 (which follows from Parseval and the boundedness of $f$ and $g$). We have chosen $\g$ and $\d$ so that $\d^2\g^{-1}+4\g=\e$, so the result follows.
\end{proof}

When $G=\F_p^n$, which is the case that concerns us, it is convenient to set $\d=0$ instead. Then $B$ becomes the subspace $\bigcap_{\chi\in K}\{x:\chi(x)=1\}$, which has codimension at most $|K|$. We can then take $\g$ to be $\e/4$. Let us state this conclusion separately, so we can refer to it more conveniently.

\begin{lemma} \label{weakbog}
Let $G=\F_p^n$, let $f,g:G\to\C$ be bounded functions, and let $\e>0$. Then there is a subspace $B$ of $\F_p^n$ of codimension at most $4/\e$ such that $\|f*g-\mu_B*f*g\|_2\leq\e$.
\end{lemma}

\noindent We remark that in this case the map $F\mapsto\mu_B*F$ is precisely the averaging projection $P_B$ discussed at the beginning of the previous subsection.

\subsection{$L_2$ approximation in the bilinear case: preliminaries}

We begin with some simple and standard lemmas.

\begin{lemma} \label{easyyoung}
Let $G$ be a finite Abelian group and let $f,g:G\to\C$. Then
\[\|\widehat{f\dc g}\|_1\leq\|f\|_2\|g\|_2.\]
 \end{lemma}

\begin{proof}
By the convolution law, Cauchy-Schwarz, and Parseval's identity,
\[\|\widehat{f*g}\|_1=\sum_\chi|\hat f(\chi)||\hat g(\chi)|\leq\|\hat f\|_2\|\hat g\|_2=\|f\|_2\|g\|_2.\]
\end{proof}

The next lemma is the convolution law but with the product on the physical side and the convolution on the Fourier side. We give the proof just to make clear that the normalization is correct, and also because we use the slightly less standard convolution $\dc$.

\begin{lemma} \label{product}
Let $f$ and $g$ be two functions defined on an Abelian group $G$. Then 
\[\widehat{f.\overline g}=\hat f\dc\hat g.\]
\end{lemma}

\begin{proof}
\begin{align*}
\hat f\dc\hat g(\chi)&=\sum_{\chi_1\chi_2^{-1}=\chi}\hat f(\chi_1)\overline{\hat g(\chi_2)}\\
&=\sum_{\chi_1\chi_2^{-1}=\chi}\E_{x,y}f(x)\overline{\chi_1(x)}\overline{g(y)}\chi_2(y)\\
&=\E_{x,y}f(x)\overline{g(y)}\sum_{\chi_1\chi_2^{-1}=\chi}\overline{\chi_1(x)\chi_2(y)^{-1}}\\
&=\E_{x,y}f(x)\overline{g(y)}\overline{\chi(x)}\sum_{\chi_2}\chi_2(y-x)\\
&=\E_{x,y}f(x)\overline{g(y)}\overline{\chi(x)}\Delta_{xy}\\
&=\E_xf(x)\overline{g(x)}\overline{\chi(x)}\\
&=\widehat{f.\overline g}(\chi),\\
\end{align*}
where $\D_{xy}=|G|$ if $x=y$ and 0 otherwise, so $\E_{x,y}F(y)\Delta_{xy}=F(x)$ for any $F$.
\end{proof}

The next lemma is a special case of a standard fact about convolutions of functions in $\ell_1$ (to be precise, it is a very easy case of Young's inequality). Again we give the proof just for the sake of establishing our normalizations in the context that concerns us, and demonstrating that the nonstandard convolution makes no difference.

\begin{lemma} \label{ell1convolution}
Let $f$ and $g$ be two functions defined on an Abelian group $G$. Then 
\[\|\hat f\dc\hat g\|_1\leq\|\hat f\|_1\|\hat g\|_1.\]
\end{lemma}

\begin{proof}
We have
\begin{align*}
\sum_\chi|\hat f\dc\hat g(\chi)|&=\sum_\chi\bigl|\sum_{\chi_1\chi_2^{-1}=\chi}\hat f(\chi_1)\overline{\hat g(\chi_2)}\bigr|\\
&\leq\sum_\chi\sum_{\chi_1\chi_2^{-1}=\chi}|\hat f(\chi_1)||\hat g(\chi_2)|\\
&=\sum_{\chi_1}|\hat f(\chi_1)|\sum_{\chi_2}|\hat g(\chi_2)|\\
&=\|\hat f\|_1\|\hat g\|_1.\\
\end{align*}
\end{proof}

From these lemmas we can draw a simple conclusion about the mixed convolution of four bounded functions.

\begin{lemma} \label{columndecay}
Let $G$ be a finite Abelian group and let $f_1,f_2,f_3$ and $f_4$ be bounded functions from $G^2$ to $\C$. Let $F=\mc(f_1,f_2,f_3,f_4)$. Then $\|\widehat{F_{x\bullet}}\|_1\leq 1$ for every $x\in G$. 
\end{lemma}

\begin{proof}
Let $F_1=f_1\vc f_2$. Then $(F_1)_{x\bullet}=(f_1)_{x\bullet}\dc(f_2)_{x\bullet}$ for every $x$, so by Lemma \ref{easyyoung} we have that $\|\widehat{(F_1)_{x\bullet}}\|_1\leq 1$. Setting $F_2=f_3\vc f_4$ we obtain the same conclusion for $F_2$. Now for each $x$ we have
\[F_{x\bullet}=(F_1\hc F_2)_{x\bullet}=\E_u(F_1)_{u\bullet}\overline{(F_2)_{(u-x)\bullet}}.\]
Therefore, taking Fourier transforms and using Lemma \ref{product}, we have
\[\widehat{F_{x\bullet}}=\E_u\widehat{(F_1)_{u\bullet}}\dc\widehat{(F_2)_{(u-x)\bullet}}.\]
By Lemma \ref{ell1convolution} the right-hand side is an average of functions with $\ell_1$-norm at most 1, so by the triangle inequality $\|\widehat{F_{x\bullet}}\|_1\leq 1$ as claimed.
\end{proof}

The above lemma gives us another way of understanding why we do not obtain a uniform approximation in the bilinear case. In the proof above we made crucial use of the triangle inequality in $\ell_1$. But to obtain uniform approximations, one typically has to obtain Fourier decay that is faster than $\ell_1$ decay. That is, one needs bounds on the $\ell_q$ norm for some $q<1$. Unfortunately, the triangle inequality does not hold when $q<1$, so the averaging step above fails. Thus, without an extra idea one cannot do better than $\ell_1$ decay and approximations that are valid almost everywhere rather than everywhere. 

\subsection{Relationships between the large spectra of the rows of a mixed convolution}

The key to obtaining bilinear structure is another lemma from \cite{gowers} (which was also used to obtain bilinear structure). That in its turn depends on some tools that have become standard in additive combinatorics. 

First we need a lemma that is quite similar to Lemma \ref{linearaddquads}. It is a very slight reformulation of Lemma 13.1 of \cite{gowers}.

\begin{lemma} \label{addquads}
Let $\a>0$, let $G$ be a finite Abelian group, let $f:G^2\to\C$ be a bounded function, let $g=f\vc f$, let $B\subset G$, and let $\sigma:B\to G^*$ be a function such that $\E_h\b1_B(h)|\widehat{g_{\bullet h}}(\sigma(h))|^2\geq\a$. Then there are at least $\a^4|G|^3$ quadruples $(a,b,c,d)\in B^4$ such that $a+b=c+d$ and $\sigma(a)+\sigma(b)=\sigma(c)+\sigma(d)$.
\end{lemma}

We shall also use Lemma \ref{BSG} again. And thirdly, we shall use technology related to the Freiman-Ruzsa theorem, with the greatly improved bounds of Sanders \cite{sanders}, who proved the following result. (The bound $2^{42}$ we claim for the absolute constant does not appear in the literature, but can be obtained by going carefully through Sanders's arguments. We are very grateful to Tom Sanders for supplying us with it.)

\begin{lemma} \label{sanders}
Let $A$ be a subset of $\F_p^m$ such that $|A-A|\leq K|A|$. Then there is a subspace $V$ of $\F_p^m$ of cardinality at most $|A|$ such that $|A\cap V|\geq\exp(-2^{42}(\log K+\log p)^6)|A|$.
\end{lemma}

Now let us put these results together to obtain structural information about the graph $\G$ of the function $\sigma$ in Lemma \ref{addquads} in the case $G=\F_p^n$. The conclusion of that lemma tells us that $\G$ contains at least $\a^4|G|^3$ additive quadruples. This implies that $\G$ has cardinality at least $\a^{4/3}|G|$, since the number of additive quadruples cannot exceed the number of ordered triples. Since $\G$ is the graph of a function, we also know that it has cardinality at most $|G|$, and therefore it satisfies the hypothesis of Lemma \ref{BSG} with $H=G^2$ and $\d=\a^4$. Therefore, by that lemma it has a subset $\G'$ of size at least $\a^{16/3}|G|/6$ such that $|\G'-\G'|\leq 2^{22}|\G'|/\a^{16}$. It then follows from Lemma \ref{sanders} that there is a subspace $V\subset G^2$ of size at most $|\G'|$ such that $|\G'\cap V|\geq\theta|\G'|$, where $\theta=\exp(-2^{42}(22\log 2+16\log(\a^{-1})+\log p)^6)\geq\exp(-2^{48}(\log(\a^{-1})+\log p)^6)$. (The main thing here is that $\theta$ has a quasipolynomial dependence on $\a$.) 

Write $V_0$ for the inverse image of $\{0\}$ under this projection. Then pick a subspace $V_1\subset V$ such that $V$ is a direct sum $V_0+V_1$. Since $V$ contains at least $\theta|\G'|$ points of $\G'$, $|V_1|\geq\theta|\G'|\geq\theta|V|$, and therefore $|V_0|\leq\theta^{-1}$. By averaging, we can find $v\in V_0$ such that $|\G'\cap(v+V_1)|\geq\theta^2|\G'|$. 

Now $v+V_1$ is the graph of an affine map from $\F_p^n$ to $\F_p^n$, so this is implies that there is a subset $B'$ of $B$ of size at least $\theta^2\a^{16/3}|G|/6$ such that the restriction of $\sigma$ to $B'$ is affine.

From this, we obtain the following lemma (which is also a lemma of \cite{gowers}, but there it is proved for $\Z_N$ and with a significantly worse bound, because the results of Sanders were not yet known).

\begin{lemma} \label{manylinear}
Let $\g,\e>0$, let $G=\F_p^n$, let $f:G^2\to\C$ be a bounded function, let $g=f\vc f$, and let $\Sigma$ be the set of all pairs $(h,u)$ such that $|\widehat{g_{\bullet h}}(u)|^2\geq\g$. Then there are affine functions $T_1,\dots,T_m$ such that for all but at most $\e|G|$ points $(h,u)$ of $\Sigma$ there exists $i$ with $T_ih=u$. Moreover, we may take $m$ to be $\exp(2^{56}(\log(\e^{-1})+\log(\g^{-1})+\log p)^6)$. In particular, $m$ has a quasipolynomial dependence on $\g$ and $\e$.
\end{lemma}

\begin{proof}
The idea of the proof is simple: we use the results just discussed to remove from $\Sigma$ large pieces that are restrictions of graphs of affine functions until we have removed almost everything.

The inductive step is as follows. Suppose we have a subset $\Sigma'\subset\Sigma$ of size at least $\e|G|$. Since for each $h$ there are at most $\g^{-1}$ values of $u$ with $(h,u)\in\Sigma$ (by Parseval), we can find a set $B$ of size at least $\g\e|G|$ and a function $\sigma:\F_p^n\to(\F_p^n)^*$ such that $\E_h\b1_B(h)|\widehat{g_{\bullet h}}(\sigma(h))|^2\geq\g^2\e$. Setting $\a=\g^2\e$, we may now apply the result from the discussion above to obtain a subset $B'$ of size at least $\theta^2\a^{16/3}|G|/6$, where $\theta=\exp(-2^{48}(\log(\a^{-1})+\log p)^6)$, such that the restriction of $\sigma$ to $B'$ is affine.

Since $|\Sigma|\leq\g^{-1}|G|$ (again by Parseval), we cannot repeat this process more than $6\theta^{-2}\a^{-16/3}\g^{-1}=6\theta^{-2}\e^{-16/3}\g^{-35/3}$ times, and when it stops we are left with a set of size at most $\e|G|$. We may therefore take $m$ to be 
\[6\exp(2^{49}(\log(\e^{-1})+2\log(\g^{-1})+\log p)^6)\e^{-16/3}\g^{-35/3}\leq\exp(2^{56}(\log(\e^{-1})+\log(\g^{-1})+\log p)^6).\] 
This proves the lemma.
\end{proof}

\subsection{$L_2$ approximation in the bilinear case: the main result}

We now have the ingredients we need to prove that a mixed convolution $F$ of a bounded function on $\F_p^n\times\F_p^n$ can be approximated in $L_2$ by $T_\beta F$, where $T_\beta$ is the averaging projection on to the level sets of a bilinear map $\be:\F_p^n\times\F_p^n\to\F_p^k$ for some not too large $k$.

To recap, we begin with a bounded function $f:\F_p^n\times\F_p^n\to\C$ and we let $F=\mc f$. Lemma \ref{columndecay} implies that $\|\widehat{F_{x\bullet}}\|_1\leq 1$ for every $x$.

Let $\e,\g>0$ be constants to be chosen later. Letting $g=f\vc f$ (so that $F=g\hc g$), we also have from Lemma \ref{manylinear} some affine functions $T_1,\dots,T_m:\F_p^n\to\F_p^n$ such that for all but at most $\e|G|$ values of $h$, the $\g$-large spectrum of $g_{\bullet h}$ (meaning here the set of $u$ such that $|\widehat{g_{\bullet h}}(u)|^2\geq\g$) is contained in the set $\{T_1h,\dots,T_mh\}$. Moreover, $m$ has the quasipolynomial dependence on $\e$ and $\g$ given in Lemma \ref{manylinear}. 

Next, let us define for each $i$ a function $u_i:\F_p^n\to\C$ by setting 
\[u_i(y)=\begin{cases} 0 & T_jy=T_iy\ \mbox{for some } j<i \\
\widehat{F_{\bullet y}}(T_iy) & \mbox{otherwise} \\
\end{cases}\]
Then define a function $G$ by the formula
\[G(x,y)=G_{\bullet y}(x)=\sum_{i=1}^mu_i(y)\omega^{x.T_iy}=\sum_{v\in\{T_1y,\dots,T_my\}}\widehat{F_{\bullet y}}(v)\omega^{x.v}.\]
Then for each $y$, $\widehat{G_{\bullet y}}$ is the restriction of $\widehat{F_{\bullet y}}$ to the set $\{T_1y,\dots,T_my\}$. If $y$ is such that the $\g$-large spectrum of $g_{\bullet y}$ is contained in this set, then 
\[\|\widehat{G_{\bullet y}}-\widehat{F_{\bullet y}}\|_2^2\leq\|\widehat{G_{\bullet y}}-\widehat{F_{\bullet y}}\|_1\|\widehat{G_{\bullet y}}-\widehat{F_{\bullet y}}\|_\infty\leq\g,\]
where the second inequality follows from the fact that $\hat F_{\bullet y}=|\hat g_{\bullet y}|^2$. Also, for every $y$ we have that $\|G_{\bullet y}-F_{\bullet y}\|_2^2\leq 1$. Since the first inequality holds for all but at most $\e|G|$ values of $y$, it follows that $\|G-F\|_2^2\leq\g+\e$.

Define $v_i(y)$ to be $\E_xF(x,y)\omega^{-x.T_iy}$. Writing $S_i(y)$ for the set $\{j:T_jy=T_iy\}$, we have $u_i=v_i.\b1_{\{y:i=\min S_i(y)\}}$.

The rough idea now is that the functions $v_i$ are ``nice", and each $u_i$ is built out of a few restrictions of $v_i$ to subspaces, which makes the $u_i$ ``nice" too (but less so). The details are as follows.

\begin{lemma} \label{vsmallnorm}
Let $v_i$ be defined as above. Then $\|\widehat{v_i}\|_1\leq 1$.
\end{lemma}

\begin{proof}
This follows from the fact that $\|\widehat{F_{x\bullet}}\|_1\leq 1$ for every $x$. Indeed,
\begin{align*}
\widehat{v_i}(w)&=\E_{x,y}F(x,y)\omega^{-x.T_iy-w.y}\\
&=\E_x\widehat{F_{x\bullet}}(T_i^*x+w).\\
\end{align*}
It follows that
\[\|\widehat{v_i}\|_1\leq\E_x\sum_w|\widehat{F_{x\bullet}}(T_i^*x+w)|=\E_x\|\widehat{F_{x\bullet}}\|_1\leq 1,\]
as claimed.
\end{proof}

That is the sense in which the $v_i$ are nice. Now we turn to the $u_i$. To begin with, we recall a simple fact about Fourier transforms of subspaces.

\begin{lemma} \label{transformsubspace}
Let $V$ be a subspace of $G=\F_p^n$. Then $\|\widehat{\b1_V}\|_1=1$.
\end{lemma}

\begin{proof}
This is a direct calculation.
\[\b1_V(u)=\E_x\b1_V(x)\omega^{-x.u}=\begin{cases} |V|/|G| & u\in V^\perp \\ 0 & \mbox{otherwise}\\ \end{cases}.\]
Since $|V||V^\perp|=|G|$, the result follows.
\end{proof}

\begin{corollary} \label{restricttosubspace}
Let $V$ be a subspace of $\F_p^n$ and let $f:\F_p^n\to\C$. Then $\|\widehat{f\b1_V}\|_1\leq\|\hat f\|_1$.
\end{corollary}

\begin{proof}
The result follows from Lemma \ref{transformsubspace} and Lemma \ref{ell1convolution}, since 
$\widehat{f\b1_V}=\hat f\dc\widehat{\b1_V}$.
\end{proof}

Now observe that for each $i$ the characteristic function of the set $\{y:i=\min S_i(y)\}$ can be written as a $\pm 1$ combination of at most $2^m$ characteristic functions of subspaces, by the inclusion-exclusion formula. Indeed, 
\[\b1_{\{y:i=\min S_i(y)\}}=\prod_{j<i}(1-\b1_{\{y:j\in S_i\}}).\]
The right-hand side is a $\pm 1$ combination of characteristic functions of intersections of subspaces of the form $\{y:T_jy=T_iy\}$, and there are at most $2^m$ such intersections.

Therefore, the $\ell_1$-norm of the Fourier transform of $\b1_{\{y:i=\min S_i(y)\}}$ is at most $2^m$. By Lemmas \ref{vsmallnorm} and \ref{restricttosubspace} and the fact that $u_i=v_i.\b1_{y:i=\min S_i(y)}$, it follows that $\|\widehat{u_i}\|_1\leq 2^m$.

Let us summarize what we have proved in the form of a lemma. 

\begin{lemma} \label{firstapproximation}
Let $f:\F_p^n\times\F_p^n\to\C$ be a bounded function and let $F$ be the mixed convolution $\mc f$. Then for every $\d>0$ there exist linear maps $T_1,\dots,T_m:\F_p^n\to\F_p^n$ and functions $u_1,\dots,u_m:\F_p^n\to\C$ such that $\|\widehat{u_i}\|_1\leq 2^m$ for each $i$ and such that setting $G(x,y)=\sum_{i=1}^mu_i(y)\omega^{x.T_iy}$ we have that $\|G-F\|_2^2\leq\d$. Moreover, $m$ can be taken to equal $\exp(2^{63}(\log(\d^{-1})+\log p)^6)$.
\end{lemma}

\begin{proof}
Set $\e=\g=\d/2$ in the discussion above. Then, as we showed, $\|G-F\|_2^2\leq\e+\g=\d$. The bound for $m$ comes from setting $\e=\g=\d/2$ in the bound in Lemma \ref{manylinear}.
\end{proof}

We are ready for one of the main results of this section, which may be of independent interest.

\begin{theorem} \label{approxbyaveproj}
For every $\zeta>0$ there exists a positive integer $k$ with the following property. Let $G=\F_p^n$ and let $f:G\times G\to\C$ be any bounded function. Then there is a bi-affine map $\be:G^2\to\F_p^k$ such that, writing $F$ for the mixed convolution $\mc f$ and $P_\be$ for the averaging projection on to the level sets of $\be$, we have the approximation $\|F-P_\be F\|_2\leq\zeta$. Moreover, $k$ can be taken to be $4m^32^m/\zeta^2$, where $m=\exp(2^{69}(\log(\zeta^{-1})+\log p)^6)$.
\end{theorem}

\begin{proof}
We do a further approximation by truncating the Fourier transforms of the $u_i$. Let $\d=\zeta/2$. For each $i$, $\|\hat u_i\|_1\leq 2^m$, so we can find a set $K_i$ of size at most $m^22^m/\d^2$ such that $|\widehat{u_i}(v)|<\d^2/m^2$ for every $v\notin K_i$. Let $w_i(y)=\sum_{v\in K_i}\widehat{u_i}(v)\omega^{v.y}$. Then 
\[\|u_i-w_i\|_2^2=\|\widehat{u_i}-\widehat{w_i}\|_2^2\leq\|\widehat{u_i}-\widehat{w_i}\|_1\|\widehat{u_i}-\widehat{w_i}\|_\infty\leq\d^2/m^2\]
for each $i$. It follows that the function $y\mapsto(u_i(y)-w_i(y))\omega^{x.T_iy}$ has $L_2$ norm at most $\d/m$, so if we set $H(x,y)=\sum_{i=1}^mw_i(y)\omega^{x.T_iy}$, then $\|H-G\|_2\leq\d$, by the triangle inequality, and therefore $\|H-F\|_2\leq 2\d$, also by the triangle inequality.

Let us write $K_i=\{v_{i1},\dots,v_{ik_i}\}$. Then 
\[H(x,y)=\sum_{i=1}^m\sum_{j=1}^{k_i}\widehat{u_i}(v_{ij})\omega^{(T_i^*x+v_{ij}).y}.\]
Writing $\lambda_{ij}$ for $\widehat{u_i}(v_{ij})$ and $\be_{ij}(x,y)$ for $(T_i^*x+v_{ij}).y$, we can rewrite this as $\sum_{i=1}^m\sum_{j=1}^{k_i}\lambda_{ij}\omega^{\be_{ij}(x,y)}$. That is, we have approximated $F$ in $L_2$ by a linear combination of a bounded number of bilinear phase functions.

Now define a bilinear map $\be:\F_p^n\to\F_p^k$, where $k=\sum_{i=1}^mk_i$ by concatenating all the $\be_{ij}$. That is, we set
\[\be(x,y)=(\be_{11}(x,y),\dots,\be_{1k_1}(x,y),\dots,\be_{m1}(x,y),\dots,\be_{mk_m}(x,y)).\]
Then $H$ is constant on the level sets of $\be$. But the closest approximation to $F$ in $L_2$ by a function that is constant on these level sets is $P_\be F$, where $P_\be$ is the averaging projection. It follows that $\|F-P_\be F\|_2\leq 2\d=\zeta$. 

Since each $K_i$ has size at most $m^22^m/\d^2$, we have the bound $k\leq m^32^m/\d^2=4m^32^m/\zeta^2$. The bound now follows from substituting $\zeta/2$ for $\d$ in the bound for $m$ in Lemma \ref{firstapproximation}.
\end{proof} 

The dependence of $k$ on $\zeta$ is given by a polylogarithmic function followed by an exponential one. The exponential part of this comes from our use of the inclusion-exclusion argument just after Corollary \ref{restricttosubspace}. It is not obvious that it is needed, so it is conceivable that an exponential can be saved in this part of the argument.

In a forthcoming paper \cite{GM2} we shall prove a theorem that describes a 

\subsection{Obtaining a near bihomomorphism with respect to a linear combination of the level sets of a bilinear function}

We have done two things in this section so far: we have obtained a near bihomomorphism $\psi$ with respect to a measure of the form $F=\mc\b1_A$, and we have shown that the mixed convolution $\mc\b1_A$ can be approximated in $L_2$ by $P_\be F$, where $P_\be$ is the averaging projection to the level sets of a bilinear function $\be:\F_p^n\times\F_p^n\to\F_p^k$, with $k$ depending only on the density of $A$. We complete this section by checking that if the $L_2$ approximation is good enough, then $\psi$ is a near bihomomorphism on $P_\be F$ as well. 

First we give a standard argument to show that a reasonably dense measure contains many 4-arrangements.

\begin{lemma} \label{manyarrangements}
Let $G$ be a finite Abelian group and let $\mu$ be a non-negative function defined on $G^2$. Then $\bigl\|\mc\mu\bigr\|_2^2\geq\|\mu\|_1^8$.
\end{lemma}

\begin{proof}
First, note that if $F$ is any non-negative function defined on $G^2$, then 
\begin{align*}
\|F\hc F\|_1&=\E_y\|F_{\bullet y}\dc F_{\bullet y}\|_1\\
&=\E_y\|F_{\bullet y}\|_1^2\\
&\geq(\E_y\|F_{\bullet y}\|_1)^2\\
&=\|F\|_1^2\\
\end{align*}
and that the same proof gives the corresponding inequality for vertical convolutions as well.

We therefore have
\begin{align*}
\bigl\|\mc\mu\,\bigr\|_2^2&\geq\bigl\|\mc\mu\,\bigr\|_1^2\\
&=\|(\mu\vc\mu)\hc(\mu\vc\mu)\|_1^2\\
&\geq\|\mu\vc\mu\|_1^4\\
&\geq\|\mu\|_1^8,\\
\end{align*}
which proves the lemma.
\end{proof}

\begin{lemma} \label{simplebound}
Let $G$ be a finite Abelian group and let $f_1,\dots,f_4$ be bounded functions from $G^2$ to the group algebra $\cA$ of $G$. Then $\|\mc(f_1,\dots,f_4)\|_2\leq\max_i\|f_i\|_2$.
\end{lemma}

\begin{proof}
Recall that if $f$ and $g$ are two functions from $G$ to $\cA$, then we define the convolution $f\dc g$ by 
$f\dc g(x)=\E_{u-v=x}f(u)g(v)^*$, which then allows us to define horizontal and vertical convolutions in the same way as for scalar-valued functions.

Now note that if $F$ and $F'$ are bounded functions on $G^2$, then 
\begin{align*}\|F\hc F'\|_2^2&=\E_y\|F_{\bullet y}\dc F'_{\bullet y}\|_2^2\\
&\leq\E_y\|F_{\bullet y}\dc F'_{\bullet y}\|_\infty^2\\
&\leq\E_y\|F_{\bullet y}\|_2^2\|F'_{\bullet y}\|_2^2\\
&\leq\E_y\|F_{\bullet y}\|_2^2\\
&=\|F\|_2^2.\\
\end{align*}
The same proof also gives the corresponding inequality for vertical convolutions. Therefore, since all the relevant functions are bounded, 
\begin{align*}
\|\mc(f_1,\dots,f_4)\,\|_2&=\|(f_1\vc f_2)\hc(f_3\vc f_4)\|_2\\
&\leq\|f_1\vc f_2\|_2\\
&\leq\|f_1\|_2.\\
\end{align*}
By symmetry, the inequality holds for the other $f_i$, and the lemma is proved.
\end{proof}

Now let $\mu$ and $\nu$ be non-negative bounded functions on $G^2$ and let $\psi$ be a $(1-\eta)$-bihomomorphism with respect to $\mc\mu$. Suppose also that $\psi:G^2\to\cA$ is a $(1-\eta)$-bihomomorphism with respect to $\mc\mu$. That is,
\[\E_{w,h}\bigl\|\E_{P\in\cP(w,h)}(\mc\mu)(P)\psi(P)\bigr\|_2^2\geq(1-\eta)\E_{w,h}\bigl|(\mc\mu)(w,h)\bigr|^2.\]
Note that the expression $\E_{P\in\cP(w,h)}(\mc\mu)(P)\psi(P)$ is equal to $\mc(\mu\phi)(w,h)$, so a more concise formulation of the above inequality is
\[\bigl\|\,\mc(\mu\psi)\bigr\|_2^2\geq(1-\eta)\bigl\|\mc\mu\bigr\|_2^2\]
(where the $L_2$ norm on the left-hand side is a sum of squares of $L_2$-norms of elements of $\cA$). 

We would like to deduce a similar inequality for $\nu$ when $\|\mu-\nu\|_2$ is sufficiently small. To do this, note first that $\bigl\|\,\mc(\mu\psi)\bigr\|_2^2=\langle\mc(\mu\psi),\mc(\mu\psi)\rangle$ and similarly for $\nu$, so $\bigl\|\,\mc(\mu\psi)\bigr\|_2^2-\bigl\|\,\mc(\nu\psi)\bigr\|_2^2$ is equal to a sum of eight terms of the form
\[\langle \mc(\xi_1,\xi_2,\xi_3,\xi_4),\mc(\xi_5,\xi_6,\xi_7,\xi_8)\rangle,\]
where for the $i$th such term we have $\xi_1=\dots=\xi_{i-1}=\nu\psi$, $\xi_i=(\mu-\nu)\psi$, and $\xi_{i+1}=\dots=\xi_8=\mu\psi$.
Both sides of the inner product are bounded functions, and hence have $L_2$ norm at most 1. But by Lemma \ref{simplebound} one of the sides of each term has $L_2$-norm at most $\|\mu-\nu\|_2$, so by Cauchy-Schwarz each term is at most $\|\mu-\nu\|_2$ and therefore the sum of the eight terms is at most $8\|\mu-\nu\|_2$.

If we take $\psi$ to be the function taking the constant value $\d_0$ in the above argument, we obtain also the inequality $\Bigl|\bigl\|\mc\mu\bigr\|_2^2-\bigl\|\mc\mu\bigr\|_2^2\Bigr|\leq 8\|\mu-\nu\|_2$ as well.

This gives us the following lemma.

\begin{lemma} \label{perturbmeasure}
Let $G$ be a finite Abelian group with group algebra $\cA$, let $\mu$ and $\nu$ be non-negative bounded functions on $G^2$ and suppose that $\|\mu-\nu\|_2\leq(\eta/16)\|\mu\|_1^8$. Let $\psi:G^2\to\cA$ be a $(1-\eta)$-bihomomorphism with respect to $\mu$. Then $\psi$ is a $(1-2\eta)$-bihomomorphism with respect to $\nu$.
\end{lemma}

\begin{proof}
By hypothesis, we have the inequality
\[\bigl\|\,\mc(\mu\psi)\bigr\|_2^2\geq(1-\eta)\bigl\|\mc\mu\bigr\|_2^2.\]
If we replace $\mu$ by $\nu$, then the calculations above show that neither side of this inequality changes by more than $8\|\mu-\nu\|_2$. By asumption, this is at most $(\eta/2)\|\mu_1\|^8$, which in turn is at most $(\eta/2)\bigl\|\mc\mu\bigr\|_2^2$, by Lemma \ref{manyarrangements}. Therefore, 
\[\bigl\|\,\mc(\nu\psi)\bigr\|_2^2\geq(1-\eta/2)\bigl\|\,\mc(\mu\psi)\bigr\|_2^2\geq(1-3\eta/2)\bigl\|\,\mc\mu\bigr\|_2^2\geq(1-3\eta/2)(1+\eta/2)^{-1}\bigl\|\mc\nu\bigr\|_2^2,\]
which gives us our conclusion, since $(1-3\eta/2)(1+\eta/2)^{-1}\geq 1-2\eta$.
\end{proof}

Now let us draw everything together and give the second main result of this section.

\begin{theorem} \label{bihomwrtbilinear}
For every $\a,\eta,\zeta>0$ such that $\zeta\leq\a^4\eta/16$ there exists $k$ with the following property. Let $G=\F_p^n$, let $\cA$ be its group algebra, let $A\subset G$ be a set of density $\a$, and let $\phi:G^2\to\cA$ be a second-order $(1-\eta)$-bihomomorphism with respect to $\b1_A$. Let $\mu=\mc\b1_A$ and let $\psi:G^2\to\Sigma(\cA)$ be defined by the equation $(\mc\mu)\psi=\mc(\b1_A\phi)$ (with $\psi$ taking an arbitrary value at $(w,h)$ if $\mc\mu(w,h)=0$). Then there is a bi-affine map $\be:G^2\to\F_p^k$ and a non-negative function $\nu:G^2\to\R$ such that $\nu(x,y)$ depends only on $\be(x,y)$, $\psi$ is a $(1-2\eta)$-bihomomorphism with respect to $\nu$, and $\|\mu-\nu\|_2\leq\zeta$. Moreover, we may take $k$ to be $2^{10}m^32^m/\a^4\eta$, where $m=\exp(2^{69}(\log(\zeta^{-1})+\log p)^6)$.
\end{theorem}

\begin{proof}
Lemma \ref{isoonmixedconv} (with the definition of $\psi$ in its proof) tells us that $\psi$ is a $(1-\eta)$-bihomomorphism with respect to $\mu=\mc\b1_A$. The proof of Lemma \ref{manyarrangements} gives us that $\|\mu\|_1\geq\|\b1_A\|_1^4=\a^4$. Now we apply Theorem \ref{approxbyaveproj} to obtain a positive integer $k$ and a bilinear map $\be:G^2\to\F_p^k$ such that, writing $P_\be$ for the averaging projection to the level sets of $\be$, we have that $\|\mu-P_\be\mu\|_2\leq\zeta$. Setting $\nu=P_\be\mu$, we have that $\nu$ is constant on the level sets of $\be$ and, by Lemma \ref{perturbmeasure} and the bound on $\zeta$, that $\psi$ is a $(1-2\eta)$-bihomomorphism with respect to $\nu$. The bound for $k$ comes from Theorem \ref{approxbyaveproj}.
\end{proof}

\section{Obtaining a near bihomomorphism on a high-rank bilinear Bohr set}

Our rough aim now is to deduce from the conclusion of Theorem \ref{bihomwrtbilinear} that $\psi$ is a near bihomomorphism on one of the level sets of $\be$, and then, using structural information about these level sets, to prove a stability result that allows us to approximate $\psi$ by the restriction of an exact bihomomorphism defined on all of $(\F_p^n)^2$. However, before we do this, we shall need a few results about bilinear maps. In particular, as several other authors have found, there are many arguments that work only for bilinear maps of sufficiently large rank. Ours are no exception, so we shall need a lemma that tells us that we can obtain this condition if we do not have it already.

\subsection{The rank of a bilinear map}\label{rank}

Let $\be:\F_p^n\times\F_p^n\to\F_p$ be a bilinear form. We can write it in the form $\be(x,y)=x.Ty$ for some linear map $T:\F_p^n\to\F_p^n$: the \emph{rank} of $\be$ is defined to be the rank of $T$. If $\be$ has rank $t$, then $\E_{x,y}\omega^{\be(x,y)}=p^{-t}$, since if $Ty=0$, then the expectation over $x$ is 1, and otherwise it is 0, and the probability that $Ty=0$ is $p^{-t}$. Thus, we can define the rank of $\be$ in an equivalent analytic way as $-\log_p(\E_{x,y}\omega^{\be(x,y)})$. 

If $\be$ is bi-affine rather than bilinear, meaning that it has a formula of the form 
\[\be(x,y)=x.Ty+x.u+v.y+\lambda,\]
where $T$ is a linear map as before, $u,v\in\F_p^n$, and $\lambda\in\F_p$, then we again define the rank of $\be$ to be the rank of $T$. The analytic characterization no longer works, but it does if we modify it in an obvious way. If we define $f(x,y)$ to be $\omega^{\be(x,y)}$, then 
\[f(x,y)\overline{f(x-a,y)f(x,y-b)}f(x-a,y-b)=\omega^{a.Tb}\]
for every $x,y$, so a suitable definition is
\[-\log_p\bigl(\E_{x,y,a,b}f(x,y)\overline{f(x-a,y)f(x,y-b)}f(x-a,y-b)\bigr)=-\log_p\|f\|_\square^4,\]
where $\|.\|_\square$ is the box norm for functions defined on $\F_p^n\times\F_p^n$. As we shall see later, this analytic definition of rank is worth making because it can be generalized straightforwardly and usefully to multilinear maps. 

The box norm is a well-known measure of quasirandomness. Thus, if we have a bilinear form of high rank, it enables us to use quasirandomness properties in our proofs.

We are concerned not just with bilinear forms but with bilinear maps to $\F_p^k$ for small $k$. It would be natural to call $k$ the rank (or at least an upper bound for the rank) of such a map. But since that word is taken, we shall call $k$ the \emph{dimension} of a bilinear map $\be:\F_p^n\times\F_p^n\to\F_p^k$. By a \emph{bilinear set}, we shall mean a set of the form $B_{\be}=\{(x,y):\be(x,y)=0\}$, where $\be$ is such a map. We shall call $k$ the \emph{codimension} of $B_{\be}$. We shall use the same terminology for bi-affine sets, which are defined as above but with $\be$ a bi-affine map instead. Rather than continually mentioning bi-affine maps, we shall talk about bilinear maps on the understanding that what we say applies to bi-affine maps as well (unless we make clear that this is not the case).

We define the \emph{rank} of a bilinear map $\be:\F_p^n\times\F_p^n\to\F_p^k$ to be the smallest rank of any bilinear form $(x,y)\mapsto u.\be(x,y)$ such that $u$ is a non-zero element of $\F_p^k$. That is, if $\be=(\be_1,\dots,\be_k)$, it is the smallest rank of a non-trivial linear combination of the bilinear forms $\be_i$. 

Let $X$ and $Y$ be vector spaces over $\F_p$ and let $X'\subset X$ and $Y'\subset Y$ be affine subspaces. Define a \emph{bilinear Bohr set of codimension} $k$ in $X'\times Y'$ to be a set of the form $\{(x,y)\in X'\times Y':\be(x,y)=z\}$, where $\be$ is a bi-affine map from $X'\times Y'$ to some other space $Z$. In other words, it is one of the ``level sets" we talked about earlier. The \emph{rank} of a bilinear Bohr set is the highest possible rank of a bi-affine map of which it is a level set. (The reason the definition is phrased like that is that, as easy but artificial examples show, two bi-affine maps of different ranks can have the same level sets: for instance, just take the direct sum of a bilinear map with the zero map.) 

Let $X,Y$ and $Z$ be vector spaces over $\F_p$ and let $\be:X\times Y\to Z$ be a bilinear map. Suppose we have direct-sum decompositions $X=X_0+X_1$ and $Y=Y_0+Y_1$. If we write a typical element $(x,y)$ of $X\times Y$ as $(x_0+x_1,y_0+y_1)$ in the obvious way, then we have that
\[\be(x,y)=\be_{00}(x_0,y_0)+\be_{01}(x_0,y_1)+\be_{10}(x_1,y_0)+\be_{11}(x_1,y_1),\]
where $\be_{ij}$ is the restriction of $\be$ to $X_i\times Y_j$. For each $(v,w,z)\in X_0\times Y_0\times Z$, let $B_{v,w,z}$ be the set of all $(x_0+x_1,y_0+y_1)\in X\times Y$ such that $x_0=v$, $y_0=w$, and $\be(x,y)=z$. This is a bilinear Bohr set inside the affine subspace $\{(x,y):x_0=v,y_0=w\}$. We call the decomposition of $X\times Y$ into the bilinear Bohr sets $B_{v,w,z}$ a \emph{bilinear Bohr decomposition}. Its \emph{rank} is the smallest rank of any of the sets $B_{v,w,z}$ (with the convention that the rank of the empty set is infinite, as is the rank of a product of two affine subspaces).

The basic idea behind Corollary \ref{gethighrank} below is reasonably standard, even if the details are less so. Given a bi-affine map $\be:X\times Y\to\F_p^k$, where $X$ and $Y$ are copies of $\F_p^n$, our aim is to find a bilinear Bohr decomposition of $X\times Y$ of high rank, using $\be$, with the dimensions of $X_0$ and $Y_0$ not too large. 

We make one more set of definitions before starting on our argument. Given vector spaces $X,Y$ and $Z$ over $\F_p$, every bi-affine map $\be:X\times Y\to Z$, it has a formula of the form 
\[\be(x,y)=\g(x,y)+Ax+By+z\]
where $\g$ is bilinear, $A:X\to Z$ and $B:Y\to Z$ are linear and $z\in Z$. We can say explicitly what these various components are. Indeed, they are given by the following formulae.
\begin{enumerate}
\item $\g(x,y)=\be(x,y)-\be(x,0)-\be(0,y)+\be(0,0)$.
\item $Ax=\be(x,0)-\be(0,0)$.
\item $By=\be(0,y)-\be(0,0)$.
\item $z=\be(0,0)$.
\end{enumerate}
We shall call $\g$ the \emph{bilinear part} of $\be$. We shall say that two bi-affine maps are \emph{equivalent} if they have the same bilinear part, and we shall say that a bi-affine map is \emph{linear} if its bilinear part is zero. We have already commented that two equivalent bi-affine maps have the same rank.

The next result will serve as an inductive step.

\begin{lemma} \label{indstep}
Let $V,W$ and $X$ be vector spaces over $\F_p$ and let $\be:V\times W\to X$ be a bi-affine map of rank at most $t$. Then there are subspaces $V'\subset V$ and $W'\subset W$ of codimension at most $t$ and a non-zero vector $u\in X$ such that $u.\be$ is linear on every product $(V'+v)\times(W'+w)$ of translates of $V'$ and $W'$.
\end{lemma}

\begin{proof}
By hypothesis, there is some non-zero $u\in\F_p^k$ such that the bi-affine form $(x,y)\mapsto u.\be(x,y)$ has rank at most $t$. Let $u.\be(x,y)$ be given by the formula $x.Ty+x.a+b.y+\lambda$, where $T$ is a linear map of rank at most $t$. Let $V'=\ker T^*$ and let $W'=\ker T$. Then $x.Ty=0$ for every $(x,y)\in V'\times W'$. (We also have that $x.Ty=0$ for every $(x,y)\in V'\times W$, but we prefer a more symmetrical conclusion.) It follows that for all such $(x,y)$ and for every $(v,w)$, 
\[u.\be(x+v,y+w)=x.(Tw+a)+(T^*v+b).y+u.\be(v,w),\]
and therefore that $\be$ is linear on each product $(V'+v)\times(W'+w)$. Since the spaces $V'$ and $W'$ have codimension at most $t$, we are done.
\end{proof}

A simple fact that will be useful in the next proof is that the rank of the restriction of a biaffine map $\be$ to a product of translates $(V'+v)\times(W'+w)$ does not depend on $v$ and $w$. Indeed, if the restriction of $u.\be$ to $V'\times W'$ is given by the formula $u.\be(x,y)=x.Ty+x.a+b.y+\lambda$, then, as we essentially saw at the end of the proof above, the restriction to $(V'+v)\times(W'+w)$ is given by the formula $u.\be(x+v,y+w)=x.Ty+x.(Tw+a)+(T^*v+b).y+u.\be(v,w)$, so it is an equivalent bi-affine map to the restriction to $V'\times W'$.

\begin{corollary} \label{gethighrank}
Let $X,Y$ and $Z$ be vector spaces over $\F_p$, let $\be:X\times Y\to Z$ be a bilinear map, and let $t$ be a positive integer. Then there are decompositions $X=X_0+X_1$ and $Y=Y_0+Y_1$ such that $X_0$ and $Y_0$ have dimension at most $t\dim Z$ and the resulting Bohr decomposition has rank at least $t$.
\end{corollary}

\begin{proof}
Let $V_0=X$, $W_0=Y$ and $\be_0=\be$. If $\be$ is of rank at least $t$, then we can take the decompositions $\{0\}+V_0$ and $\{0\}+W_0$. Otherwise, by Lemma \ref{indstep} there exist $u\in\F_p^k$ and subspaces $V_1\subset V_0$ and $W_1\subset W_0$ of codimension at most $t$ such that $u.\be$ is linear on $V_1\times W_1$ and hence on every product $(V_1+v)\times(W_1+w)$. Let $U_1$ be the subspace generated by $u$ and write $U_1^\perp$ for the subspace $\{v:u.v=0\}$. (Any complementary subspace would do here, but for notational purposes it is more convenient to choose a specific one.) Let $P_1$ and $Q_1$ be the projections to $U_1$ and $U_1^\perp$ arising from the decomposition $\F_p^k=U_1+U_1^\perp$. So far we know that $P_1\circ\be$ is linear on every product $(V_1+v)\times(W_1+w)$. 

Now let $\be_1$ be the restriction of $Q_1\circ\be$ to $V_1\times W_1$. If $\be_1$ has rank at least $t$, then we can choose decompositions $X=X_0+X_1$ and $Y=Y_0+Y_1$ with $X_1=V_1$ and $Y_1=W_1$. To see why this works, pick $(v,w,z)\in V_1\times W_1\times Z$. Then $B_{v,w,z}$ is the set of all $(v+x_1,w+y_1)$ such that $(x_1,y_1)\in X_1\times Y_1$ and $\be(v+x_1,w+y_1)=z$. But 
\[\be(v+x_1,w+y_1)=\be_1(x_1,y_1)+(P_1\circ\be)(x_1,y_1)+\be(x_1,w)+\be(v,y_1)+\be(v,w).\]
The restriction of $P_1\circ\be$ to $X_1\times Y_1$ is linear, as are the last three terms, so the rank of the restriction of $\be$ to $(X_1+v)\times(Y_1+u)$ is equal to the rank of $\be_1$, which is at least $t$. Therefore, for each $z$ the set $B_{v,w,z}$ is of rank at least $t$, by definition.

Suppose that at the $r$th stage of this process we have found subspaces $V_r\subset X$ and $W_r\subset Y$ of codimension at most $tr$ and a subspace $U_r\subset\F_p^k$ of dimension $r$ such that, writing $P_r$ and $Q_r$ for the projections to $U_r$ and $U_r^\perp$, we have that $P_r\circ\beta$ is linear on every product $(V_r+v)\times(W_r+w)$. If the restriction of $\be_r=Q_r\circ\beta$ to $V_r\times W_r$ has rank at least $t$, then we are done, by the argument just given when $r=1$. Otherwise, there exist $u\in U_r^\perp$ and subspaces $V_{r+1}\subset V_r$ and $W_{r+1}\subset W_r$ of codimension at most $t$ such that the restriction of $u.\be_r$ to $V_{r+1}\times W_{r+1}$ has rank at most $t$. Since $P_r\circ\beta$ is also linear when restricted to $V_{r+1}\times W_{r+1}$, it follows that $u.\beta$ is linear when restricted to $V_{r+1}\times W_{r+1}$. Therefore, if we set $U_{r+1}=U_r+\langle u\rangle$ and define $P_{r+1}$ and $Q_{r+1}$ to be the projections to $U_{r+1}$ and $U_{r+1}^\perp$, we have obtained subspaces $V_{r+1}\subset X$ and $W_{r+1}\subset Y$ of codimension at most $t(r+1)$ and a subspace $U_{r+1}\subset\F_p^k$ of dimension $r+1$ such that $P_{r+1}\circ\be$ is linear on every product $(V_{r+1}+v)\times(W_{r+1}+w)$. 

This proves the inductive step. The induction stops either with some $\be_r$ such that $Q_r\circ\be_r$ has rank at least $t$ when restricted to $V_r\times W_r$ or with $r=k$. But $Q_k$ will be the zero map, so by our convention if $r=k$ then $Q_r\circ\be_r$ has infinite rank.
\end{proof}

\subsection{Some of the main consequences of high rank} \label{highrank}

As we have already mentioned, there are several useful statements concerning linear maps that have analogues for bilinear maps that are not true in general but are true when the bilinear maps have high rank. For example, if $\alpha:\F_p^n\to\F_p^k$ is a surjective linear map, then the level sets $\{x:\a x=z\}$ all have density $p^{-k}$, whereas if $\be:\F_p^n\times\F_p^n\to\F_p^k$ is a surjective bilinear map, the level sets $\{(x,y):\be(x,y)=z\}$ may have quite different sizes. An instance of this is the map
\[(x,y)\mapsto(x_1y_1,\dots,x_ky_k).\]
This takes the value $(0,0,\dots,0)$ with probability $((2p-1)/p^2)^k$, so the density of the corresponding level set is significantly greater than the average of $p^{-k}$.

The next two lemmas encapsulate what we can get from a high-rank assumption. (The second can be deduced from the first, but with a slightly worse bound.)

\begin{lemma} \label{inductivehighrank}
Let $X,Y$ and $Z$ be finite-dimensional vector spaces over $\F_p$, with $\dim Z=k$, let $\be:X\times Y\to Z$ be a bi-affine map of rank $t$, and let $U\subset X$ and $V\subset Y$ be subspaces of codimensions $u$ and $v$. Then if $x$ is chosen uniformly at random from $U$, the probability that the restriction of the linear map $\be_{x\bullet}:y\mapsto\be(x,y)$ to $W$ is a surjection is at least $1-p^{u+v+k-t}$.
\end{lemma}

\begin{proof}
Let $\be$ be given by the formula
\[\be(x,y)=(T_1x.y+a_1.y+x.b_1+\lambda_1,\dots,T_kx.y+a_k.y+x.b_k+\lambda_k),\]
where $T_1,\dots,T_k$ are linear maps from $X$ to $Y$, $a_1,\dots,a_k$ are elements of $Y$, $b_1,\dots,b_k$ are elements of $X$, and $\lambda_1,\dots,\lambda_k$ are elements of $\F_p$. The rank assumption is that every non-trivial linear combination of the $T_i$ has rank at least $t$.

Let $V^\perp\subset Y$ be the annihilator of $V$ (with respect to the dot product), and let $q:Y\to Y/V^\perp$ be the quotient map. Then the restriction of $\be_{x\bullet}$ is a surjection if and only if the vectors $q(T_ix+a_i)$ are linearly independent in $Y/V^\perp$.

Let $\mu_1,\dots,\mu_k$ be scalars, not all zero. Then by hypothesis the map $\sum_i\mu_iT_i$ has rank at least $t$, which implies that $q\circ(\sum_i\mu_iT_i)$ has rank at least $t-v$. Therefore, the probability that $q(\sum_iT_ix)=-q(\sum_i\mu_ia_i)$ is at most $p^{v-t}$. Since there are fewer than $p^k$ choices for the $\mu_i$, it follows that the probability that the vectors $q(T_ix+a_i)$ are linearly independent is at most $p^{v+k-t}$. Finally, that was the probability for a random $x\in X$. If instead we choose $x$ randomly from $U$, this probability increases to at most $p^{u+v+k-t}$, since $U$ has density $p^{-u}$ in $X$. 
\end{proof}

\begin{lemma} \label{usinghighrank}
Let $X$, $Y$ and $Z$ be finite-dimensional vector spaces over $\F_p$ and let $\dim Z=k$. Let $(z_1,\dots,z_r)$ be a sequence of elements of $Z$, and for $1\leq i\leq r$ let $\be_i:X\times Y\to Z$ be a bi-affine map of rank at least $t$. Then if $x_1,\dots,x_r$ are chosen uniformly at random from $X$, the probability that the equations $\be_i(x_i,y)=z_i$ have exactly $p^{-rk}|Y|$ solutions in $y$ is at least $1-p^{rk-t}$. 
\end{lemma}

\begin{proof}
For each $x_1,\dots,x_r$, consider the linear map $\a:Y\to Z^r$ (which depends on $x_1,\dots,x_r$ but we suppress this in the notation) given by the formula
\[\a:y\mapsto(\be_1(x_1,y),\dots,\be_r(x_r,y)).\]
For each $i$, let $\be_i$ be given by the formula
\[\be_i(x,y)=(T_{i1}x.y+a_{i1}.y+x.b_{i1}+\lambda_{i1},\dots\dots,T_{ik}x.y+a_{ik}.y+x.b_{ik}+\lambda_{ik}).\]
The statement that the equations $\a(y)=(z_1,\dots,z_r)$ each have $p^{-rk}|Y|$ solutions is equivalent to the statement that $\a$ has full rank, which in turn is equivalent to the statement that the vectors $T_{ij}x_i+a_{ij}$ are linearly independent. 

It remains to bound the probability that there is a linear dependence between these vectors when $x_1,\dots,x_r$ are chosen uniformly at random. We do this by looking at each possible non-zero linear combination, bounding the probability that it is zero, and then using a union bound.

Let $\mu_{ij}\in\F_p$ for each $1\leq i\leq r$ and $1\leq j\leq k$, and suppose that $i',j'$ are such that $\mu_{i'j'}\ne 0$. Then the rank assumption for $\be_{i'}$ implies that the linear map $\sum_j\mu_{i'j}T_{i'j}$ has rank at least $t$. It follows that for any fixed choice of $x_1,\dots,x_{i'-1},x_{i'+1},\dots,x_r$, the probability that $\sum_{ij}\mu_{ij}(T_ix_j+a_{ij})=0$ is at most $p^{-t}$, since there is only one possible value for $(\sum_j\mu_{i'j}T_{i'j})x_{i'}$ that will cause the equation to be satisfied.

Since there are fewer than $p^{rk}$ choices for the $\mu_{ij}$, we obtain the result stated.
\end{proof}

We note here a very important consequence of this lemma, which is that a high-rank bilinear Bohr set is quasirandom when thought of as the adjacency matrix of a bipartite graph. To see this, fix a biaffine map $\be$ of dimension $k$, let $B_z=\{(x,y):\be(x,y)=z\}$, and take $r=1$ and $r=2$ in the lemma above. From the $r=1$ case we find that for almost all $x$ the set $B_{x\bullet}=\{y:(x,y)\in B\}=\{y:\be(x,y)=z\}$ has density $p^{-k}$ and from the $r=2$ case we find that for almost all $x_1,x_2$ the intersection $B_{x_1\bullet}\cap B_{x_2\bullet}$ has density $p^{-2k}$. This establishes one of the many equivalent conditions for quasirandomness. We shall use this fact later.

We remark here that it is also possible to prove quasirandomness using the analytic definition of rank. This is important too, as it enables one to generalize the above lemma to multilinear maps.

\subsection{A generalized inner product and its basic properties} \label{generalizedip}

In this section we examine a generalized inner product that arises naturally in the context of mixed convolutions and 4-arrangements. Using repeated applications of Cauchy-Schwarz, just as one does with the $U^k$ norms and box norms, we prove a generalized Cauchy-Schwarz inequality for this inner product, which allows us to define a ``respecting 4-arrangements norm", though for our purposes in this paper we shall just need the generalized Cauchy-Schwarz inequality.

Let $G$ be a finite Abelian group, let $\cA$ be its group algebra, and let $\phi_1,\dots,\phi_8$ be functions from $G^2$ to $\cA$. Then we make the definition
\[[\phi_1,\phi_2,\phi_3,\phi_4,\phi_5,\phi_6,\phi_7,\phi_8]=\langle\mc(\phi_1,\phi_2,\phi_3,\phi_4),\mc(\phi_5,\phi_6,\phi_7,\phi_8)\rangle.\]
We have in fact already met the quantity on the right-hand side: it appeared in the discussion leading up to Lemma \ref{perturbmeasure}. 

We also write $[\phi]$ as shorthand for $[\phi,\phi,\phi,\phi,\phi,\phi,\phi,\phi]^{1/8}$. Interpreted appropriately, $[\phi]$ is a measure of the extent to which $\phi(P)$ depends only on the width and height of $P$ when $P$ is a random vertical parallelogram, and $[\phi_1,\dots,\phi_8]$ measures the extent to which the product of the $\phi_i$ over the vertices of a random 4-arrangement is typically close to $\d_0$. (The words ``interpreted appropriately" are alluding here to the fact that the distributions over the parallelograms and 4-arrangements must depend appropriately on $\phi$.)

We now state and prove the generalized Cauchy-Schwarz inequality, which is similar to several other such inequalities that have appeared in the literature.

\begin{lemma} \label{generalizedCS}
Let $G$ be a finite Abelian group with group algebra $\cA$ and let $\phi_1,\dots,\phi_8$ be functions from $G^3$ to $\cA$. Then 
\[[\phi_1,\phi_2,\phi_3,\phi_4,\phi_5,\phi_6,\phi_7,\phi_8]\leq[\phi_1][\phi_2][\phi_3][\phi_4][\phi_5][\phi_6][\phi_7][\phi_8].\]
\end{lemma}

\begin{proof}
By the usual Cauchy-Schwarz inequality and the definition of the generalized inner product, we have immediately that
\[[\phi_1,\dots,\phi_8]^2\leq[\phi_1,\dots,\phi_4,\phi_1,\dots,\phi_4][\phi_5,\dots,\phi_8,\phi_5,\dots,\phi_8].\]
Secondly, if one expands out the definition of the generalized inner product, one can easily verify that
\[[\phi_1,\phi_2,\phi_3,\phi_4,\phi_5,\phi_6,\phi_7,\phi_8]=[\phi_1,\phi_2,\phi_5,\phi_6,\phi_3,\phi_4,\phi_7,\phi_8],\]
so we also have the inequality
\[[\phi_1,\dots,\phi_8]^2\leq[\phi_1,\phi_2,\phi_5,\phi_6,\phi_1,\phi_2,\phi_5,\phi_6][\phi_2,\phi_3,\phi_7,\phi_8,\phi_2,\phi_3,\phi_7,\phi_8].\]

We now prove an inequality that is slightly subtler (but still just an application of Cauchy-Schwarz). Let us, very temporarily, define a different mixed convolution by the formula
\[\{\phi_1,\phi_2,\phi_3,\phi_4\}(u_1,u_2,u_3,u_4)=\E_{x_1-x_2=x_3-x_4, y}\phi(x_1,y+u_1)\phi(x_2,y+u_2)^*\phi(x_3,y+u_3)^*\phi(x_4,y+u_4).\]
Then
\[[\phi_1,\dots,\phi_8]=\langle\{\phi_1,\phi_3,\phi_5,\phi_7\},\{\phi_2,\phi_4,\phi_6,\phi_8\}\rangle,\]
and from this and Cauchy-Schwarz we obtain the inequality
\[[\phi_1,\dots,\phi_8]^2\leq[\phi_1,\phi_3,\phi_5,\phi_7,\phi_1,\phi_3,\phi_5,\phi_7][\phi_2,\phi_4,\phi_6,\phi_8,\phi_2,\phi_4,\phi_6,\phi_8].\]

If we now apply the first inequality, then the second inequality to each of the two resulting terms, and the third inequality to each of the four resulting terms of that, we obtain the result claimed.
\end{proof}

As usual, it follows easily from the generalized Cauchy-Schwarz inequality above that $[.]$ is a norm. To obtain the triangle inequality, one observes that $[g_1+g_2]^8$ is a sum of $2^8$ terms. Each one can be bounded by the inequality, and using those bounds one obtains the inequality $[g_1+g_2]^8\leq([g_1]+[g_2])^8$. This gives the triangle inequality, and the remaining norm properties are very easy.

\iftrue
\else
In this section we define a generalized inner product that is appropriate to the problem at hand. Using repeated applications of Cauchy-Schwarz, just as one does with the $U^k$ norms and box norms, we prove a generalized Cauchy-Schwarz inequality for this inner product, which allows us to define a ``respecting 4-arrangements norm", though for our purposes in this paper we shall just need the generalized Cauchy-Schwarz inequality.

Let $G$ be a finite Abelian group and let $g_1,\dots,g_8$ be functions from $G^3$ to $\C$. Then let $(\e_1,\dots,\e_8)$ be the Morse sequence $(1,-1,-1,1,-1,1,1,-1)$, let $C$ be the complex-conjugation operator, and define 
\[[g_1,\dots,g_8]=\mathop{\E}_{\substack{x_1-x_2=x_3-x_4\\ h, y_1,y_2,y_3,y_4}}\sum_{\sum_i\e_iz_i=0}\prod_{i=1}^4C^{\e_{2i-1}}g_{2i-1}(x_i,y_i,z_{2i-1})C^{\e_{2i}}g_{2i}(x_i,y_i-h,z_{2i})\]
and let $[g]$ be shorthand for $[g,g,\dots,g]^{1/8}$.

To understand the point of this expression, it is helpful to consider the main case for which it is designed. Suppose that there are functions $f_i:G^2\to\C$ and $\phi:G\to G$ such that 
\[g_i(x,y,z)=\begin{cases}f_i(x,y)&z=\phi(x,y)\\ 0&\mbox{otherwise}\\ \end{cases}\] 
Then $[g_1,\dots,g_8]$ works out to be
\begin{align*}
\mathop{\E}_{\substack{x_1-x_2=x_3-x_4\\ h, y_1,y_2,y_3,y_4}}f_1(x_1,y_1)\overline{f_2(x_1,y_1-h)f_3(x_2,y_2)}&f_4(x_2,y_2-h)\overline{f_5(x_3,y_3)}\\
&f_6(x_3,y_3-h)f_7(x_4,y_4)\overline{f_8(x_4,y_4-h)}\,\b1_{[\phi\mbox{ respects }S]}\\
\end{align*}
where $S$ in the above expression is short for the 4-arrangement defined by the points $(x_i,y_i)$ and $(x_i,y_i-h)$. In particular, if every $f_i$ is the characteristic function of a set $A\subset G^2$, then we obtain the probability that a random 4-arrangement in $G^2$ belongs to $A$ and is respected by $\phi$.

\begin{lemma} \label{generalizedCS}
Let $G$ be a finite Abelian group and let $g_1,\dots,g_8$ be functions from $G^3$ to $\C$. Then 
\[[g_1,\dots,g_8]\leq[g_1]\dots[g_8].\]
\end{lemma}

\begin{proof}
First, we prove the inequality
\[[g_1,\dots,g_8]^2\leq[g_1,g_1,g_3,g_3,g_5,g_5,g_7,g_7][g_2,g_2,g_4,g_4,g_6,g_6,g_8,g_8].\]
To see this, note first that $[g_1,\dots,g_8]$ is equal to
\[\mathop{\E}_{\substack{x_1-x_2=x_3-x_4\\ y_1,y_2,y_3,y_4}}\sum_u\sum_{z_1-z_2-z_3+z_4=u}\mathop{\E}_h\prod_{i=1}^4C^{\e_{2i-1}}g_{2i-1}(x_i,y_i-h,z_{2i-1})\sum_{z_5-z_6-z_7+z_8=u}\mathop{\E}_{h'}\prod_{i=1}^4C^{\e_{2i}}g_{2i}(x_i,y_i-h',z_{2i}).\]
If we apply the Cauchy-Schwarz inequality and square the result, we obtain a product of two terms. The first of these is
\[\mathop{\E}_{\substack{x_1-x_2=x_3-x_4\\ y_1,y_2,y_3,y_4}}\sum_u\Bigl|\sum_{z_1-z_2-z_3+z_5=u}\mathop{\E}_h\prod_{i=1}^4C^{\e_{2i-1}}g_{2i-1}(x_i,y_i-h,z_{2i-1})\Bigr|^2.\]
If we expand this out, then we obtain the formula like the one just given for $[g_1,\dots,g_8]$ but for $[g_1,g_1,g_3,g_3,g_5,g_5,g_7,g_7]$. Similarly, the other term is equal to $[g_2,g_2,g_4,g_4,g_6,g_6,g_8,g_8]$. 

Next, we show the inequality
\[[g_1,\dots,g_8]^2\leq[g_1,g_2,g_3,g_4,g_1,g_2,g_3,g_4][g_5,g_6,g_7,g_8,g_5,g_6,g_7,g_8].\]
This time, we rewrite the formula for $[g_1,\dots,g_8]$ as
\begin{align*}
\mathop{\E}_{h,u}\sum_v\Bigl(\mathop{\E}_{\substack{x_1-x_2=u\\ y_1,y_2}}&\sum_{z_1-z_2-z_3+z_4=v}\prod_{i=1}^2C^{\e_{2i-1}}g_{2i-1}(x_i,y_i,z_{2i-1})C^{\e_{2i}}g_{2i}(x_i,y_i-h,z_{2i})\Bigr)\\
&\Bigl(\mathop{\E}_{\substack{x_3-x_4=u\\ y_3,y_4}}\sum_{z_5-z_6-z_7+z_8=v}\prod_{i=3}^4C^{\e_{2i-1}}g_{2i-1}(x_i,y_i,z_{2i-1})C^{\e_{2i}}g_{2i}(x_i,y_i-h,z_{2i})\Bigr)\\
\end{align*}
Applying Cauchy-Schwarz again proves the claim.

Observe now that $[g_1,\dots,g_8]=[g_1,g_2,g_5,g_6,g_3,g_4,g_7,g_8]$. Combining this with the previous inequality we can deduce that
\[[g_1,\dots,g_8]^2\leq [g_1,g_2,g_1,g_2,g_5,g_6,g_5,g_6][g_3,g_4,g_3,g_4,g_7,g_8,g_7,g_8].\]

It is straightforward to combine these three inequalities into a proof of the lemma.
\end{proof}

As usual, it follows easily from the generalized Cauchy-Schwarz inequality above that $[.]$ is a norm. To obtain the triangle inequality, one observes that $[g_1+g_2]^8$ is a sum of $2^8$ terms. Each one can be bounded by the inequality, and using those bounds one obtains the inequality $[g_1+g_2]^8\leq([g_1]+[g_2])^8$. This gives the triangle inequality, and the remaining norm properties are very easy.

\textbf{***I think the results of this paper should easily give an inverse theorem for this norm. It's probably worth checking this, as it might be quite interesting.***}
\fi

\subsection{Restricting to a single high-rank bilinear set}

Before continuing with the argument, let us briefly recap. We began with a function $\phi$ that respected almost all second-order 4-arrangements on a dense set $A$. We obtained from that a function $\psi$ that respects almost all first-order 4-arrangements on the mixed convolution $\mc\b1_A$. We then obtained a low-dimensional bilinear map $\be$ such that $\mc\b1_A$ is close in $L_2$ to $P_\be(\mc\b1_A)$, where $P_\be$ is the averaging projection to the level sets of $\be$. We also showed that if the $L_2$ distance between these two functions is sufficiently small, then $\psi$ will respect almost all first-order 4-arrangements on $P_\be(\mc\b1_A)$ as well.

The main result of Subsection \ref{rank} allows us to create from $\be$ a bilinear Bohr decomposition into high-rank bilinear Bohr sets $B_{v,w,z}$, each of which is contained in a level set of $\be$. Thus, $P_\be(\mc\b1_A)$ is constant on each $B_{v,w,z}$. In this section we shall show that $\psi$ is a near bihomomorphism on at least one (in fact, most) of the $B_{v,w,z}$. The high-rank condition will play a critical role in our argument, which we shall highlight when we come to it. (We certainly need it for the next lemma, but the main point, which is the one we shall highlight later, is that we cannot do without the lemma.) 

\begin{lemma} \label{4arrangementsinsets}
Let $X,Y$ and $Z$ be vector spaces over $\F_p$ and let $\be:X\times Y\to Z$ be a bi-affine map of dimension $k$. Let $X=X_0+X_1$ and $Y=Y_0+Y_1$ be direct-sum decompositions of $X$ and $Y$, with $\dim X_0=r$ and $\dim Y_0=s$. For each $(v,w,z)\in X_0\times Y_0\times Z$ let $B_{v,w,z}$ be the set $\{(x,y):x_0=v,y_0=w,\be(x,y)=z\}$ and let $b_{v,w,z}$ be the characteristic function of $B_{v,w,z}$. Suppose that the bilinear Bohr decomposition into the sets $B_{v,w,z}$ has rank at least $t$. 

Let $(u_1,v_1,z_1),\dots,(u_8,v_8,z_8)$ be eight triples in $X_0\times Y_0\times Z$. If the points $(v_1,w_1),\dots,(v_8,w_8)$ form a 4-arrangement and $z_1-z_2-z_3+z_4=z_5-z_6-z_7+z_8$, then 
\[\bigl|[b_{v_1,w_1,z_1},\dots,b_{v_8,w_8,z_8}]-p^{-3r-5s-7k}\bigr|\leq p^{-3r-5s}(3p^{2k-t}+p^{k-t}).\]
Otherwise, $[b_{v_1,w_1,z_1},\dots,b_{v_8,w_8,z_8}]=0$.
\end{lemma}

\begin{proof}
The quantity we are trying to estimate is equal to the probability that the eight vertices of a random 4-arrangement lie in the eight sets $B_{v_i,w_i,z_i}$ (in the right order). To understand the statement of the lemma, it is helpful to define $B_{v,w}$ to be the union of all $B_{v,w,z}$, which is the product $(X_1+v)\times(Y_1+w)$. Then $p^{-3r-5s}$ is (as we shall check in more detail below) the probability that each point of a random 4-arrangement lies in the correct set $B_{v_i,w_i}$. The high-rank condition tells us that each set $B_{v,w,z}$ should have density approximately $p^{-k}$ in $B_{v,w}$, and the term $p^{-7k}$ is telling us that the probabilities of landing in the given sets $B_{v_i,w_i,z_i}$, conditional on having landed in the sets $B_{v_i,w_i}$, are approximately independent, except that once seven of the points are in the right sets, the eighth one is in the right set automatically. Finally, the form of the bound on the right-hand side is not important: what matters is that it can be made significantly smaller than $p^{-3r-5s-7k}$ by taking a sufficiently large $t$ (which depends on $k$ only).

With these remarks out of the way, we can start the proof in earnest. If either of the two conditions on the eight triples $(u_i,v_i,w_i)$ fails to hold, then the sets $B_{v_i,w_i,z_i}$ do not support any 4-arrangements. That is because the projection of a 4-arrangement in $X\times Y$ to $X_0\times Y_0$ must also be a 4-arrangement, and because $\be$ respects all 4-arrangements.

Now let us consider a random 4-arrangement
\[(x_1,y_1),(x_1,y_1-h),(x_2,y_2),(x_2,y_2-h),(x_3,y_3),(x_3,y_3-h),(x_4,y_4),(x_4,y_4-h),\]
which is obtained by choosing a random quadruple $(x_1,x_2,x_3,x_4)$ with $x_1-x_2=x_3-x_4$, a random $h$, and random $y_1,y_2,y_3,y_4$. If we exclude $x_4$, then the remaining variables are independent, which makes it easy to use Lemma \ref{usinghighrank} to estimate probabilities.

To avoid confusion with indices, write $P$ and $Q$ for the projections from $X$ to $X_0$ and from $Y$ to $Y_0$, respectively. Let us begin by estimating the probability that $(x_1,y_1)\in B_{v_1,w_1,z_1}$ and $(x_1,y_1-h)\in B_{v_2,w_2,z_2}$. A necessary condition for this event is that $Qh=w_1-w_2$, which happens with probability $p^{-s}$. If we condition on this event, then the probability that $(x_1,y_1)\in B_{v_1,w_1}$ and $(x_2,y_1-h)\in B_{v_2,w_2}$ is equal to the probability that $Px_1=v_1$ and $Qy_1=w_1$, since this automatically implies that $Px_1=v_2$ (because the $(u_i,v_i)$ form a 4-arrangement). This probability is $p^{-r-s}$. We now claim that if we condition further on the event that $(x_1,y_1)\in B_{v_1,w_1}$ and $(x_1,y_1-h)\in B_{v_2,w_2}$, then Lemma \ref{usinghighrank} implies that the probability that $\be(x_1,y_1)=z_1$ and $\be(x_1,y_1-h)=z_2$ differs from $p^{-2k}$ by at most $p^{2k-t}$.

To justify this last claim, note first that 
\[\be(x_1,y_1)=\be(v_1+(1-P)x_1,w_1+(1-Q)y_1),\]
which has a bi-affine dependence on $(1-P)x_1$ and $(1-Q)y_1$, and
\[\be(x_1,y_1-h)=\be(v_2+(1-P)x_1,w_2+(1-Q)(y_1-h)),\]
which has a bi-affine dependence on $(1-P)x_1$ and $(1-Q)(y_1-h)$. Moreover, the two bi-affine maps in question, since they are restrictions of $\be$ to two products of translates of $X_1$ and $Y_1$, are equivalent, and have rank at least $t$ by assumption, and the points $(1-Q)y_1$ and $(1-Q)(y_1-h)$ are (with the conditionings we have made) independent uniformly distributed random elements of $Y_1$. Therefore, by Lemma \ref{usinghighrank} (with the roles of the $x$ and $y$ variables switched), with probability at least $1-p^{2k-t}$, $y_1$ and $h$ are such that with probability exactly $p^{-2k}$, $\be(x_1,y_1)=z_1$ and $\be(x_1,y_1-h)=z_2$. For the other $y_1$ and $h$, this probability cannot differ from $p^{-2k}$ by more than 1, and the claim follows. For later reference, let us call the exceptional pairs $(y_1,h)$ \emph{bad} pairs.

To summarize what we have shown so far, the probability that $Ph=w_1-w_2$ is $p^{-s}$, and the conditional probability that $(x_1,y_1)\in B_{v_1,w_1,z_1}$ and $(x_1,y_1-h)\in B_{v_2,w_2,z_2}$ differs from $p^{-r-s-2k}$ by at most $p^{-r-s+2k-t}$. 

Since $w_1-w_2=w_3-w_4=w_5-w_6=w_7-w_8$, the condition on $Ph$ gives us the corresponding conditions for each $i=1,2,3,4$. Furthermore, the corresponding conditional probabilities of the pairs of points lying in the sets we want also differ from $p^{-r-s-2k}$ by at most $p^{-r-s+2k-t}$. 

However, these four conditional probabilities are not independent, for a small reason and a big reason.

The big reason is that once we know that seven of the points belong to their respective sets $B_{v_i,w_i,z_i}$, we have automatically that the eighth does as well. We shall take account of that in due course. The small reason is that the events that $(y_i,h)$ is a bad pair do not have any reason to be independent. However, this is not much of a problem since only a very small proportion of pairs are bad.

We are ready to complete the proof. We first condition on the event that $Qh=w_1-w_2$, which, as we have noted, has probability $p^{-s}$. Next, we condition further on the event that $(x_i,y_i)\in B_{v_{2i-1},w_{2i-1}}$ and $(x_i,y_i-h)\in B_{v_{2i},w_{2i}}$ for each $i$. This has probability $p^{-3r-4s}$, since the events that $Px_i=v_{2i-1}$ occur with probability $p^{-r}$ and are independent except that the fourth one follows from the first three, while the events that $Qy_i=w_{2i-1}$ occur with probability $p^{-s}$ and are fully independent.

Having conditioned on this event, we observe that the probability that any of $(y_1,h)$, $(y_2,h)$ or $(y_3,h)$ is a bad pair is at most $3p^{2k-t}$. Let us also say that $y_4$ is bad if the probability that $\be(x_4,y_4)=z_7$ is not $p^{-k}$. Lemma \ref{usinghighrank} tells us that the probability that $y_4$ is bad is at most $p^{k-t}$. If there is no badness, then for each $i$ the probability that $\be(x_i,y_i)=z_{2i-1}$ and $\be(x_i,y_i-h)=z_{2i}$ is $p^{-2k}$. For fixed $y_1,y_2,y_3,y_4,h$ the first three of these events are independent. Given that they hold, the probability that $\be(x_4,y_4)=z_7$ is $p^{-k}$, and given that, it is automatic that $\be(x_4,y_4-h)=z_8$. So the probability that everything is in the right set differs from $p^{-7k}$ by at most $3p^{2k-t}+p^{k-t}$. Putting all this together gives the result.
\end{proof}


To prepare for the next result, we prove a simple inequality concerning real numbers.

\begin{lemma} \label{trivinequality}
Let $a_1,\dots,a_n$ and $b_1,\dots,b_n$ be real numbers with $0\leq a_i\leq b_i$ for every $i$, and suppose that $\sum_{i=1}^na_i\geq(1-\eta)\sum_{i=1}^nb_i$. Then there exists $j$ such that $a_j\geq(1-2\eta)b_j$ and $b_j\geq\eta n^{-1}\sum_{i=1}^nb_i$.
\end{lemma}

\begin{proof}
Let $E=\{j:b_j<\eta n^{-1}\sum_{i=1}^nb_i\}$. If the conclusion of the lemma is false, then for every $j\notin E$ we have that $a_j<(1-2\eta)b_j$. It follows that
\[\sum_{i=1}^na_i=\sum_{i\in E}a_i+\sum_{i\notin E}a_i\leq\sum_{i\in E}b_i+\sum_{i\notin E}a_i<\eta\sum_{i=1}^nb_i+(1-2\eta)\sum_ib_i=(1-\eta)\sum_{i=1}^nb_i,\]
which contradicts our main assumption.
\end{proof}

\begin{corollary} \label{restricttooneset}
Let $G=\F_p^n$, let $\cA$ be the group algebra of $G$, and let the sets $B_{v,w,z}$ be as in Lemma \ref{4arrangementsinsets} with $X=Y=G$, let $\mu$ be a function taking values in $[0,1]$ that is constant on each $B_{v,w,z}$, and let $\phi:G^2\to\Sigma(\cA)$ be a $(1-\eta)$-bihomomorphism with respect to $\mu$. Let $\xi$ be another function taking values in $[0,1]$ that is constant on each $B_{v,w,z}$ and suppose that $\E\xi=\zeta$. Let $0<\g\leq\eta[\mu]^8/8$. Suppose that the rank $t$ of the bilinear Bohr decomposition satisfies the inequality $p^{-t}\leq\eta p^{-9k}/8$. Then there exists $(v,w,z)$ such that $\phi$ is a $(1-4\eta)$-bihomomorphism with respect to (the characteristic function of) $B_{v,w,z}$, the value of $\mu$ on $B_{v,w,z}$ is at least $[\mu]^8/2$, and the value of $\xi$ on $B_{v,w,z}$ is at most $\gamma^{-1}\zeta$.
\end{corollary}

\begin{proof}
Write $\mu_{v,w,z}$ for the value taken by $\mu$ on the set $B_{v,w,z}$. Then $\mu=\sum_{v,w,z}\mu_{v,w,z}b_{v,w,z}$, where, once again, $b_{v,w,z}$ is the characteristic function of $B_{v,w,z}$. 

By Markov's inequality, for any $\g$, if we choose $(x,y)$ at random, then the probability that $\xi(x,y)\geq \g^{-1}\zeta$ is at most $\g$. Therefore, if we set $E$ to be the set where $\xi(x,y)\geq\g^{-1}\zeta$ we have that the $L_1$ norm of the restriction $\nu$ of $\mu$ to $E$ is at most $\g$. 

\iftrue
\else
As in the discussion at the beginning of Subsection \ref{generalizedip}, from each function $b_{v,w,z}$ we now create a function from $G^3$ to $\C$ (in fact taking non-negative values). Define $f_{v,w,z}$ by the formula
\[f_{v,w,z}(x,y,u)=\begin{cases}b_{v,w,z}(x,y)& \phi(x,y)=u\\ 0 & \mbox{otherwise}\\ \end{cases}.\]
It will also be convenient to define a function $g_{v,w,z}$ by the formula
\[g_{v,w,z}(x,y,u)=\begin{cases}b_{v,w,z}(x,y)& u=0\\ 0 & \mbox{otherwise}\\ \end{cases}.\]
Then $[f_{v_1,w_1,z_1},\dots,f_{v_8,w_8,z_8}]$ is the probability that a random 4-arrangement in $G^2$ has its $i$th point in $B_{v_i,w_i,z_i}$ for each $i$ and is also respected by $\phi$, while $[g_{v_1,w_1,z_1},\dots,g_{v_8,w_8,z_8}]$ is the probability that a random 4-arrangement in $G^2$ has its $i$th point in $B_{v_i,w_i,z_i}$ for each $i$ (since the function that takes the value zero everywhere respects all 4-arrangements). 

Let us also define $F$ to be $\sum_{v,w,z}\mu_{v,w,z}f_{v,w,z}$ and $G$ to be $\sum_{v,w,z}\mu_{v,w,z}g_{v,w,z}$. Equivalently, 
\[F(x,y,u)=\begin{cases}\mu(x,y) & \phi(x,y)=u\\ 0 & \mbox{otherwise}\\ \end{cases},\]
and $G(x,y,u)$ is the same but with $\phi$ replaced by the function that is identically zero. Then the statement that $\phi$ is a $(1-\eta)$-bihomomorphism on $\mu$ is equivalent to the inequality
\[[F]\geq(1-\eta)[G].\]
\fi 

Then our hypothesis is that 
\[\|\mc(\mu\phi)\|_2^2\geq(1-\eta)\|\mc\mu\|_2^2,\]
which is equivalent to the statement 
\[[\mu\phi]^8\geq(1-\eta)[\mu]^8,\]
where $[\mu]$ has the obvious interpretation. (Strictly speaking, the norm $[.]$ is defined only for functions from $G^2$ to $\cA$, but we can convert $\mu$ into such a function by multiplying it by $\d_0$, for example.) Since $\|\nu\|_1\leq\g$, it follows that $[\nu\phi]^8\geq(1-\eta)[\mu]^8-8\g\geq(1-2\eta)[\mu]^8$.

Writing $\phi_{v,w,z}$ for the restriction of $\phi$ to $B_{x,y,z}$ (or more accurately for the function $b_{v,w,z}\phi$) and expanding both sides, this gives us the statement
\begin{align*}
\sum_{(v_1,w_1,z_1),\dots,(v_8,w_8,z_8)\ \hbox{good}}&\mu_{v_1,w_1,z_1}\dots\mu_{v_8,w_8,z_8}[\phi_{v_1,w_1,z_1},\dots,\phi_{v_8,w_8,z_8}]\\
&\geq(1-2\eta)\sum_{(v_1,w_1,z_1),\dots,(v_8,w_8,z_8)}\mu_{v_1,w_1,z_1}\dots\mu_{v_8,w_8,z_8}[b_{v_1,w_1,z_1},\dots,b_{v_8,w_8,z_8}],\\
\end{align*}
where we call $(v,w,z)$ good if $B_{v,w,z}$ is not a subset of $E$.

By Lemma \ref{trivinequality}, we can find triples $(v_i,w_i,z_i)_{i=1}^8$ such that
\begin{enumerate}
\item $[\phi_{v_1,w_1,z_1},\dots,\phi_{v_8,w_8,z_8}]\geq(1-4\eta)[b_{v_1,w_1,z_1},\dots,b_{v_8,w_8,z_8}]$,
\item both sides of the above inequality are non-zero,
\item $(v_i,w_i,z_i)$ is good for $i=1,\dots,8$,
\item $\mu_{v_1,w_1,z_1}\dots\mu_{v_8,w_8,z_8}$ is at least $2\eta$ times the average over all such products for which the $(v_i,w_i)$ form a 4-arrangement and $z_1-z_2-z_3+z_4=z_5-z_6-z_7+z_8$.
\end{enumerate}
Note that because both sides are non-zero, we must have that the pairs $(v_i,w_i)$ form a 4-arrangement and that $z_1-z_2-z_3+z_4=z_5-z_6-z_7+z_8$ for our chosen triples.

Lemma \ref{generalizedCS} allows us to bound the left-hand side of the inequality in the first condition above by $\prod_{i=1}^8[\phi_{v_i,w_i,z_i}]$, while Lemma \ref{4arrangementsinsets} implies that the right-hand side can be bounded below by $(1-4\eta)p^{-3r-5s}(p^{-7k}-4p^{2k-t})$. It follows that there exists some $i$ such that
\[[\phi_{v_i,w_i,z_i}]^8\geq(1-4\eta)p^{-3r-5s}(p^{-7k}-4p^{2k-t})\]
and $\mu_{v_i,w_i,z_i}\geq 2\eta[\mu]^8$. 

Lemma \ref{4arrangementsinsets} also implies that 
\[[b_{v_i,w_i,z_i}]^8\leq p^{-3r-5s}(p^{-7k}+4p^{2k-t}).\]
Since $4p^{2k-t}\leq \eta p^{-7k}/2$, it follows that
\[[\phi_{v_i,w_i,z_i}]^8\geq(1-4\eta)(1-\eta/2)(1+\eta/2)^{-1}[b_{v_i,w_i,z_i}]^8\geq(1-5\eta)[b_{v_i,w_i,z_i}]^8,\]
which gives us that $\phi$ is a $(1-5\eta)$ homomorphism with respect to $b_{v_i,w_i,z_i}$.

Finally, since by Lemma \ref{4arrangementsinsets} the numbers $[b_{v_1,w_1,z_1},\dots,b_{v_8,w_8,z_8}]$ are equal to within a factor of 2, a crude lower bound for the average in (3) above is $[\mu]^8/2$. (This is crude because we could improve it by taking account of the fact that many of the $[b_{v_1,w_1,z_1},\dots,b_{v_8,w_8,z_8}]$ could be zero.)
\end{proof}

Now we can see why it was crucial for the sets $B_{v,w,z}$ to have high rank (inside their respective sets $B_{v,w}$). In order for the argument to work, it was essential to have a lower bound for $[b_{v_1,w_1,z_1},\dots,b_{v_8,w_8,z_8}]$. In general, we can use Lemma \ref{generalizedCS} obtain an upper bound, but there is no matching lower bound. But when the sets $B_{v,w,z}$ have high rank, the generalized inner products $[b_{v_1,w_1,z_1},\dots,b_{v_8,w_8,z_8}]$ that are non-zero are all approximately the same, and this gives us the lower bound we need.

The role of the function $\xi$ needs explaining. Recall that the measure that is constant on the sets $B_{v,w,z}$ is an $L_2$ approximation of the mixed convolution of the characteristic function of a set $A'\subset G^2$. When we come to apply the above lemma, it is important not just that the restriction of $\phi$ should be a near bihomomorphism, but also that the mixed convolution $\mc\b1_{A'}$ should be approximately constant on the set $B_{v,w,z}$ we restrict to. So we shall set $\xi$ to be the variance of the restriction. Since the approximation is good on average, it will be good on most sets $B_{v,w,z}$, and we need to make sure that we restrict to one of those sets.

\section{Nearly bilinear functions on high-rank bilinear sets}

Let $G$ and $H$ be finite Abelian groups. Define a map $\phi:G\to H$ to be a $(1-\eta)$-homomorphism if the proportion of quadruples $x-y=z-w$ for which $\phi(x)-\phi(y)=\phi(z)-\phi(w)$ is at least $1-\eta$. It is a well-known and useful fact that if $\phi$ is a $(1-\eta)$-homomorphism, then there is a Freiman homomorphism (that is, a 1-homomorphism) that agrees with $\phi$ on a subset of $G$ of density $1-C\eta$. In the next three sections we shall prove a similar result for $(1-\eta)$-bihomomorphisms defined on bilinear Bohr sets. Whereas the proof of the linear result above is a simple exercise, the proof of the bilinear version is substantially trickier. It also requires the bilinear Bohr set to have high rank, so that it will have quasirandomness properties that we can exploit. 

This high-rank assumption is not just convenient, but necessary, as the following simple example shows. (We shall just sketch the proof that the example works.) Let $G=\F_p^n$, let $k$ be a positive integer, and let $B\subset G^2$ be the set $\bigcap_{i,j\leq k}\{(x,y):x_iy_j=0\}$. Now define $\phi:B\to\F_p$ by setting $\phi(x,y)=1$ if there exists $i\leq k$ such that $x_i\ne 0$ and 0 otherwise. 

Note that if $(x,y)\in B$ and $x_i\ne 0$, then $y_1=\dots=y_k=0$, so for every $(x,y)\in B$ we either have $x_1=\dots=x_k=0$ or $y_1=\dots=y_k=0$. Let $B_1=\{(x,y):x_1=\dots=x_k=0\}$ and $B_2=\{(x,y):y_1=\dots=y_k=0\}$. Then $B=B_1\cup B_2$. Also, the densities of $B_1$ and $B_2$ are $p^{-k}$ and the density of $B_1\cap B_2$ is $p^{-2k}$. 

There is some choice about how to define a $(1-\eta)$-bihomomorphism in this context, but for any reasonable definition, the function $\phi$ we have defined is one for $\eta$ roughly comparable to $p^{-k}$. But any 1-bihomomorphism that agrees with $\phi$ on a subset of density $1-\eta$ will have to agree with $\phi|_{B_1}$ on most of $B_1$ and with $\phi|_{B_2}$ on most of $B_2$, which implies easily that it will have to be 0 everywhere on $B_1$ and 1 everywhere on $B_2$, which is not possible.

\subsection{The linear case}

Although it is standard, we shall give a proof for the linear case, since the proof will serve as a model for the proof of the bilinear case. Also, our proof is not completely standard, since we phrase it in terms of maps to group algebras.

Let $G$ be a group with group algebra $\cA$. It will be useful to introduce a notion of distance between elements of $\Sigma(\cA)$. (Recall that this is the set of non-negative functions in $\cA$ that sum to 1.) Given $\phi,\psi\in\Sigma(\cA)$, we define $d(\phi,\psi)$ to be $1-\langle\phi,\psi\rangle$. 

The distance $d$ is not a metric, since $d(\phi,\phi)=1-\|\phi\|_2^2$ does not have to be zero. In this respect, it resembles the Ruzsa distance between two sets. This feature is actually an advantage, as the value of $d(\phi,\phi)$ is a useful parameter: if it is small, then it tells us that $\phi$ is concentrated at one point, since if $d(\phi,\phi)\leq\e$, then
\[1-\e\leq\|\phi\|_2^2\leq\|\phi\|_1\|\phi\|_\infty=\|\phi\|_\infty.\]

The main reason we call $d$ a distance is that it satisfies the triangle inequality.

\begin{lemma} \label{triangleineq}
Let $\phi,\psi$ and $\omega$ be elements of $\Sigma(\cA)$. Then $d(\phi,\omega)\leq d(\phi,\psi)+d(\psi,\omega)$.
\end{lemma}

\begin{proof}
Note first that if $x,y\in[0,1]$, then $xy\geq x+y-1$ (since $(1-x)(1-y)\geq 0$). It follows that
\[
d(\phi,\omega)=1-\langle\phi,\omega\rangle\leq 1-\langle\psi,\phi\omega\rangle\leq 1-\langle\psi,\phi+\omega-1\rangle=2-\langle\psi,\phi+\omega\rangle=d(\phi,\psi)+d(\psi,\omega),\]
as required.
\end{proof}

We also have various obvious symmetry properties, as well as the useful property that $d(fg,h)=d(f,g^*h)$. Note also that $d$ is bi-affine: that is, if $\sum_i\lambda_i=1$, then $d(\phi,\sum\lambda_i\psi_i)=\sum\lambda_id(\phi,\psi_i)$. All these properties follow easily from corresponding properties of the inner product.

A further useful lemma is a slight variant of the triangle inequality.

\begin{lemma} \label{splitup}
Let $\phi_1,\phi_2,\phi_3$ and $\phi_4$ be elements of $\Sigma(\cA)$. Then
\[d(\phi_1\phi_2,\phi_3\phi_4)\leq d(\phi_1,\phi_3)+d(\phi_2,\phi_4).\]
\end{lemma}

\begin{proof}
Observe first that 
\[d(\phi,\psi)=1-\langle\phi,\psi\rangle=1-\langle\phi\psi^*,\d_0\rangle=d(\phi,\psi,\d_0).\]
Therefore,
\begin{align*}
d(\phi_1\phi_2,\phi_3\phi_4)&=d(\phi_1\phi_3^*,\phi_2^*\phi_4)\\
&\leq d(\phi_1\phi_3^*,\d_0)+d(\d_0,\phi_2^*,\phi_4)\\
&=d(\phi_1,\phi_3)+d(\phi_2,\phi_4)\\
\end{align*}
as claimed.
\end{proof}

Of course, the above lemma and induction imply that 
\[d(\phi_1\dots\phi_k,\psi_1\dots\psi_k)\leq d(\phi_1,\psi_1)+\dots+d(\phi_k,\psi_k).\]

Another nice fact is the following, which serves as a kind of cancellation law, though we shall not actually need it in this paper.

\begin{lemma} \label{cancellation}
Let $\phi,\psi$ and $\omega$ be elements of $\Sigma(\cA)$ with $d(\phi,\psi)\leq 1/2$. Then $d(\phi,\psi)\leq d(\phi\omega,\psi\omega)$.
\end{lemma}

\begin{proof}
We shall use the fact that $\langle\phi,\psi\rangle=\phi\psi^*(0)$ for any two functions $\phi,\psi\in\cA$, as well as the rather weak inequality $\langle f,g\rangle\leq\|f\|_1\|g\|_1$. We have that
\begin{align*}
\langle\phi\omega,\psi\omega\rangle&=\langle\phi\psi^*,\omega\omega^*\rangle\\
&=\phi\psi^*(0)\omega\omega^*(0)+\langle\phi\psi^*-\phi\psi^*(0)\d_0,\omega\omega^*-\omega\omega^*(0)\d_0\rangle\\
&\leq\langle\phi,\psi\rangle\|\omega\|_2^2+\|\phi\psi^*-\langle\phi,\psi\rangle\d_0\|_1\|\omega\omega^*-\omega\omega^*(0)\d_0\|_1\\
&\leq\langle\phi,\psi\rangle\|\omega\|_2^2+(1-\langle\phi,\psi\rangle)(1-\|\omega\|_2^2)\\
&=1-\langle\phi,\psi\rangle+\|\omega\|_2^2(2\langle\phi,\psi\rangle-1).\\
\end{align*}
Since $d(\phi,\psi)\leq 1/2$, $2\langle\phi,\psi\rangle-1\geq 0$, so this last expression is at most $\langle\phi,\psi\rangle$. This implies the result.
\end{proof}

In this language, the definition of a $(1-\eta)$-homomorphism $\phi:G\to\Sigma(\cA)$ becomes that 
\[\E_{x-y=z-w} d(\phi(x)\phi(y)^*,\phi(z)\phi(w)^*)\leq\eta.\]
We can express this in many equivalent ways. For example, it is equivalent to the statement that
\[d(\E_{x-y=z-w}\phi(x)\phi(y)^*\phi(z)^*\phi(w),\d_0)\leq\eta\]
and also to the statement that
\[\E_u d(\phi*\phi^*(u),\phi*\phi^*(u))\leq\eta.\]

The next lemma tells us that a $(1-\eta)$-homomorphism on a finite Abelian group induces a ``difference function" that is an approximate homomorphism in a much stronger sense: it satisfies an approximate additivity law for \emph{all} pairs, and not just most pairs.

\begin{lemma} \label{almostadditive}
Let $G$ be a finite Abelian group with group algebra $\cA$ and let $\phi:G\to\Sigma(\cA)$ be a $(1-\eta)$-homomorphism. Let $\psi=\phi*\phi^*$. Then $d(\psi(u)\psi(v),\psi(u+v))\leq 2\eta$ for every $u,v\in G$.
\end{lemma}

\begin{proof}
This follows from a direct calculation using the basic properties of the distance $d$. Indeed,
\begin{align*}
d(\psi(u)&\psi(v),\psi(u+v))\\
&=\E_{x,y,z}d\bigl(\phi(x)\phi(x-u)^*\phi(y)\phi(y-v)^*,\phi(z)\phi(z-u-v)^*\bigr)\\
&=\E_{x,y,z}d\bigl(\phi(x)\phi(y)\phi(z)^*,\phi(x-u)\phi(y-v)\phi(z-u-v)^*\bigr)\\
&\leq\E_{x,y,z}d(\phi(x)\phi(y)\phi(z)^*,\phi(x+y-z))+\E_{x,y,z}d(\phi(x-u)\phi(y-v)\phi(z-u-v)^*,\phi(x+y-z))\\
&\leq 2\eta,\\
\end{align*}
where the last inequality follows from our starting assumption and the fact that the two terms are both equal to $\E_{x-y=z-w}d(\phi(x)\phi(y)^*,\phi(z)\phi(w)^*)$. 
\end{proof}

Next, let us use a similar argument to prove a more general result (though not a direct generalization).

\begin{lemma} \label{almosthom}
Let $G$ be a finite Abelian group with group algebra $\cA$ and let $\phi_1,\phi_2:G\to\Sigma(\cA)$ be functions such that $\E_{x-y=z-w}d(\phi_1(x)\phi_2(y)^*,\phi_1(z)\phi_2(w)^*)\leq\eta$. Let $\psi=\phi_1\dc\phi_2$. Then for every $a,b,c,d$ with $a-b=c-d$ we have that $d(\psi(a)\psi(b)^*,\psi(c)\psi(d)^*)\leq 2\eta$.
\end{lemma}

\begin{proof}
Note first that our hypothesis can be rewritten 
\[\E_{x-y=z-w}d(\phi_1(x)\phi_1(y)^*,\phi_2(z)\phi_2(w)^*)\leq\eta,\]
which says that $\E_ud(\phi_1*\phi_1^*(u),\phi_2*\phi_2^*(u))\leq\eta$. This implies that $\E_ud(\phi_i(u)\phi_i(u)^*,\phi_i(u)\phi_i(u)^*)\leq\eta$ for $i=1,2$.

Now we proceed directly in a similar way to the previous lemma. Expanding the quantity we wish to bound gives us
\[\mathop{\E}_{x,y,z,w}d\bigl(\phi_1(x)\phi_2(x-a)^*\phi_1(y)^*\phi_2(y-b),\phi_1(z)\phi_2(z-c)^*\phi_1(w)^*\phi_2(w-d)\bigr).\]
By Lemma \ref{splitup} (and a slight rearrangement of the second term), this is at most
\[\mathop{\E}_{x,y,z,w}d\bigl(\phi_1(x)\phi_1(y)^*,\phi_1(z)\phi_1(w)^*\bigr)+\mathop{\E}_{x,y,z,w}d\bigl(\phi_2(x-a)\phi_2(y-b)^*,\phi_2(z-c)\phi_2(w-d)^*\bigr).\]
The remarks in the first paragraph give us that these two terms are both at most $\eta$, which proves the lemma.
\end{proof}

Using Lemma \ref{almosthom} we can quickly obtain a (slight generalization of a) well-known stability result for approximate homomorphisms.

\begin{lemma} \label{linearstability}
Let $G$ be a finite Abelian group with group algebra $\cA$, let $0\leq\eta<1/18$, and let $\phi:G\to\Sigma(\cA)$ be a $(1-\eta)$-homomorphism. Then there exists a 1-homomorphism $\omega:G\to\Sigma(\cA)$ such that $\E_xd(\phi(x),\omega(x))\leq 5\eta$.
\end{lemma}

\begin{proof}
Let $\psi=\phi*\phi^*$. Since $\psi=\psi^*$, our starting assumption is equivalent to the statement that
\[\E_xd(\phi(x),\phi\dc\psi(x))\leq\eta,\]
which implies that
\[\E_xd(\phi*\psi^*(x),\phi*\psi^*(x))\leq\eta.\]
Setting $\theta=\phi*\psi^*$, Lemma \ref{almosthom} tells us that $d(\theta(a)\theta(b)^*,\theta(c)\theta(d)^*)\leq 2\eta$ for every $a,b,c,d$ with $a-b=c-d$. This implies that $d(\theta(x),\theta(x))\leq 2\eta$ for every $x$, and therefore that $\|\theta(x)\|_\infty\geq\|\theta(x)\|_2^2\geq 1-2\eta$ for every $x$. For each $x$, define $\omega(x)$ to be $\delta_u$, where $u$ is the unique element of $G$ at which $\theta(x)$ takes its maximum value, and observe that $\langle\theta(x),\omega(x)\rangle\geq(1-2\eta)^2\geq 1-4\eta$, and therefore that $d(\theta(x),\omega(x))\leq 4\eta$. 

By the triangle inequality, it follows that $d(\omega(a)\omega(b)^*,\omega(c)\omega(d)^*)\leq 18\eta$ for every $a,b,c,d$ with $a-b=c-d$. But since each $\omega(x)$ is a delta-function and $\eta<1/18$, this implies that $\omega(a)\omega(b)^*=\omega(c)\omega(d)^*$, so $\omega$ is a 1-homomorphism.

Finally, from the inequalities $\E_xd(\phi(x),\theta(x))\leq\eta$, $d(\theta(x),\omega(x))\leq 4\eta$, and the triangle inequality, we obtain the bound $\E_xd(\phi(x),\omega(x))\leq 5\eta$.
\end{proof}

\subsection{Quasirandom sampling}

Later we shall need to use a well-known fact about quasirandom bipartite graphs. Since the proof is short, particularly in our context where there are certain regularities that make the argument tidier, we give it in full (slightly disguised, so that bipartite graphs are not explicitly mentioned).

\begin{lemma}\label{qrsample}
Let $\a>0$ and let $A_1,\dots,A_m$ be subsets of a finite set $X$. For each $x\in X$, let $B_x=\{i:x\in A_i\}$, and suppose that the following two conditions are satisfied.
\begin{enumerate}
\item All but $\e_1|X|$ of the sets $B_x$ have density $\a$ in $[m]$.
\item All but $\e_2|X|^2$ of the intersections $B_x\cap B_y$ have density $\a^2$ in $[m]$.
\end{enumerate}
Let $f:X\to[0,1]$ be an arbitrary function and let $\eta>0$. Then for all but at most $(2\a\e_1+\e_2)m/\theta^2$ of the sets $A_i$ we have the inequality
\[\bigl|\E_x\b1_{A_i}(x)f(x)-\a\E_xf(x)\bigr|\leq\theta.\]
\end{lemma}

\begin{proof}
This is an easy second-moment argument. Let us write $\be_x$ for the density of $B_x$ and $\be_{xy}$ for the density of $B_x\cap B_y$. Then
\[\E_i\E_x\b1_{A_i}(x)f(x)=\E_xf(x)\E_i\b1_{A_i}(x)=\E_xf(x)\be_x=\a\E_xf(x)\pm\e_1,\]
and
\[\E_i\bigl(\E_x\b1_{A_i}(x)f(x)\bigr)^2=\E_{x,y}f(x)f(y)\E_i\b1_{A_i}(x)\b1_{A_i}(y)=\E_{x,y}f(x)f(y)\be_{xy}=\a^2\E_{x,y}f(x)f(y)+\e_2.\]
It follows that the variance of $\E_x\b1_{A_i}(x)f(x)$ (when $i$ is chosen uniformly at random) is at most
\[\a^2\E_{x,y}f(x)f(y)+\e_2-\bigl(\a\E_xf(x)\pm \e_1\bigr)^2\leq\e_2+2\a\e_1.\]
The result now follows from Chebyshev's inequality.
\end{proof}

We can think of the quantity $\E_x\b1_{A_i}(x)f(x)$ as an attempt to estimate $\E_xf(x)$ by sampling some of its values. (Of course, one would have to divide by $\a$ to get the actual estimate.) The lemma is saying that if $A_1,\dots,A_m$ satisfy a certain quasirandomness condition, then no matter what the function $f$ is, almost all samples will give a good estimate. This basic principle will be essential to our later argument.

\subsection{The bilinear case: obtaining linearity almost everywhere in the first variable}

Our aim in this subsection will be to prove that if $\phi$ is a $(1-\eta)$-bihomomorphism on a high-rank bilinear Bohr set $B$ of codimension $k$, then the mixed convolution $\psi=\mc\phi$ has the property that $d\bigl(\psi(w_1,h)\psi(w_2,h),\psi(w_1+w_2,h)\bigr)$ is almost everywhere bounded above by a multiple of $\eta$. It is very important that the exceptional set of bad triples $(w_1,w_2,h)$ is not just small in the way that $\eta$ is small, but is significantly smaller than $p^{-k}$, assuming that the rank is large enough. So we cannot prove the result by a simple averaging argument -- a genuine use has to be made of the high-rank assumption on $B$.

We first prove two lemmas that are needed in order to establish what the trivial upper bounds are for certain quantities and thereby understand what it is for various expressions to count as small.

To avoid unduly repetitive statements of lemmas, let us decide for the rest of this section that $G=\F_p^n$, that $\cA$ is a group algebra, that $B\subset G^2$ is a bilinear Bohr set of codimension $k$ and rank $t$, and that $\eta$ is a positive constant. Also, $\be:G^2\to\F_p^k$ will be the bi-affine map used to define $B$. That is, $B=\{(x,y):\be(x,y)=0\}$. 

The following definitions will help to make our arguments more concise.

\begin{definition*} We shall say that a sequence of points $x_1,\dots,x_r$ is \emph{x-normal} for a bi-affine map $\be$ if the proportion of $y$ such that $\be(x_1,y)=\dots=\be(x_r,y)=0$ is exactly $p^{-rk}$. Similarly, a sequence $y_1,\dots,y_r$ is \emph{y-normal} for $\be$ if the proportion of $x$ such that $\be(x,y_1)=\dots=\be(x,y_k)=0$ is exactly $p^{-rk}$.
\end{definition*}

Lemma \ref{usinghighrank} tells us that a random sequence $(x_1,\dots,x_r)$ is x-normal for $\be$ with probability at least $1-p^{rk-t}$, and similarly for y-normality. We shall use this many times.

\begin{definition*} Let the \emph{vertical difference set} $B'$ be the set $\{(x,h):
\exists y\ (x,y), (x,y+h)\in B\}$ and let the \emph{vertical derivative} $\be'$ of $\be$ be the function $\be'(x,h)=\be(x,y+h)-\be(x,y)$. Define the \emph{mixed difference set} $B''$ to be $\{(w,h):\exists x\ (x,h), (x+w,h)\in B'\}$ and the \emph{mixed derivative} $\be''$ of $\be$ by $\be''(w,h)=\be'(x+w,h)-\be'(x,h)$. \end{definition*}

Note that $\be'$ is well-defined since $\be$ is affine in each variable separately, and therefore $B'=\{(x,h):\be'(x,h)=0\}$. Note also that $B'_{x\bullet}=B_{x\bullet}-B_{x\bullet}$ for each $x\in G$. Since $\be'$ is also bi-affine, it follows that $B''=\{(w,h):\be''(w,h)=0\}$ as well. Expanding the definition of $\be''(w,h)$, one sees that it is equal to 
\[\be(x+w,y+h)-\be(x+w,y)-\be(x,y'+h)-\be(x,y')\]
for some $x,y,y'$. 

If a parallelogram $(x,y), (x,y+h), (x+w,y'), (x+w,y'+h)$ has all its vertices in $B$, then we find that 
\begin{align*}
\be''(w,h)&=\be'(x+w,h)-\be'(x,h)\\
&=\be(x+w,y'+h)-\be(x+w,y')-\be(x,y+h)+\be(x,y)\\
&=0.\\
\end{align*}
The converse is true if $h$ is y-normal for $\be'$ and the rank $t$ of $\be$ is greater than $2k$. Indeed, if $\be''(w,h)=0$, we have that $\be'(x+w,h)-\be'(x,h)=0$ for every $x$. By the normality of $h$ we can find $x$ such that $\be'(x,h)=0$, which gives us that both $\be'(x,h)$ and $\be'(x+w,h)$ are 0. Moreover, the set of $x$ with this property has density $p^{-k}$. Since $\be$ has rank $t$, the probability that $x$ and $x+w$ are both normal for $\be$ is at least $1-p^{k-t}>1-p^{-k}$, so we can find $x$ such that $\be'(x,h)=\be'(x+w,h)=0$ and $x$ and $x+w$ are normal for $\be$. But then we can find $y$ and $y'$ such that $\be(x,y)=\be(x+w,y')=0$, which implies that $\be(x,y+h)=\be(x+w,y'+h)=0$, giving us a parallelogram of width $w$ and height $h$ in $B$.

Thus, for high-rank bilinear Bohr sets $B$, the set of pairs $(w,h)$ such that $B$ contains a parallelogram of width $w$ and height $h$ is contained in $B''$ and contains almost all of $B''$. 

In the bounds that follow, we shall write $a\pm b$ for a number that lives in the interval $[a-b,a+b]$. 

\begin{lemma} \label{noof4arrs1}
Suppose that $h$ is y-normal for $\be'$ and that $\be''(w,h)=0$. Then the probability that a random 4-arrangement of height $h$ and width $w$ has all its vertices in $B$ is $p^{-6k}\pm 4p^{k-t}$.
\end{lemma}

\begin{proof}
We wish to estimate the probability that the points $(x_1,y_1)$, $(x_1,y_1+h)$, $(x_1+w,y_2)$, $(x_1+w,y_2+h)$, $(x_2,y_3)$, $(x_2,y_3+h)$, $(x_2+w,y_4)$ and $(x_2+w,y_4+h)$ all belong to $B$ when the parameters $x_1,x_2,y_1,y_2,y_3,y_4$ are chosen independently at random.

Since $h$ is y-normal for $\be'$, the probability that $\be'(x_1,h)=\be'(x_2,h)=0$ is $p^{-2k}$, and since $\be''(w,h)=0$, this event implies that
\[\be'(x_1,h)=\be'(x_2,h)=\be'(x_1+w,h)=\be'(x_2+w,h)=0.\eqno{(1)}\] 
If $x_1$, $x_1+w$, $x_2$ and $x_2+w$ are all x-normal for $\be$, which is the case with probability at least $1-4p^{k-t}$, then the probability (given those values of $x_1,x_2$) that 
\[\be(x_1,y_1)=\be(x_1+w,y_2)=\be(x_2,y_3)=\be(x_2+w,y_4)=0\eqno{(2)}\] 
is $p^{-4k}$, and otherwise it lies in $[0,1]$. But the points we wish to be in $B$ are in $B$ if and only if (1) and (2) both hold, and our argument shows that the probability that they both hold is $p^{-6k}\pm 4p^{k-t}$.
\end{proof}

\begin{corollary}\label{noof4arrs2}
Suppose that $t\geq 7k$. Then the probability that a random 4-arrangement has all its vertices in $B$ is $p^{-7k}\pm 6p^{-t}$.
\end{corollary}

\begin{proof}
Let the points be as in Lemma \ref{noof4arrs1}, but this time with $w$ and $h$ varying rather than fixed. For each $w$, if $w$ is x-normal for $\be''$, then the probability that $\be''(w,h)=0$ is $p^{-k}$. It follows that the probability over all $(w,h)$ that $\be''(w,h)=0$ is $p^{-k}\pm p^{k-t}$. Also, the probability that $h$ fails to be y-normal for $\be'$ is at most $p^{k-t}$. 

If $\be''(w,h)\ne 0$ then the conditional probability that the 4-arrangement is in $B$ is zero. Putting all this together and using Lemma \ref{noof4arrs1}, we may deduce that the probability that all the points are in $B$ is
\[(p^{-k}\pm p^{k-t})(p^{-6k}\pm 5p^{k-t})=p^{-7k}\pm 6p^{-t},\]
as claimed.
\end{proof}

We shall also need a variant of the above results that applies to slightly more elaborate structures. 

\begin{lemma}\label{noof4arrs3}
Let $h_0=0$ and let the sequence $(h_1,\dots,h_r)$ be y-normal for $\be'$ and let $w,x_1,x_2,y_1,y_2,y_3$ and $y_4$ be chosen independently at random from $G$. Then the probability that all the points $(x_1,y_1+h_i)$, $(x_1+w,y_2+h_i)$, $(x_2,y_3+h_i)$ and $(x_2+w,y_4+h_i)$ belong to $B$ is $p^{-(3r+4)k}\pm 4p^{k-t}$.
\end{lemma}

\begin{proof}
The sequence $(h_1,\dots,h_r)$ is also y-normal for $\be''$, so the probability that $\be''(w,h_i)=0$ for every $i$ is $p^{-rk}$. This is a necessary condition for all the points to be in $B$. Let us now condition on this event.

The probability that $\be'(x_1,h_i)=\be'(x_2,h_i)=0$ for every $i\geq 1$ is $p^{-2rk}$, and if all those events hold, then since $\be''(w,h_i)=0$ we also have that $\be'(x_1+w,h_i)=\be'(x_2+w,h_i)=0$, which is again a necessary condition.

If the above conditions hold, then a necessary and sufficient further condition for all the points to be in $B$ is that the points $(x_1,y_1)$, $(x_1+w,y_2)$, $(x_2,y_3)$ and $(x_2+w,y_4)$ belong to $B$. The probability that the points $x_1,x_1+w,x_2$ and $x_2+w$ are all x-normal for $\be$ is at least $1-p^{4k-t}$, and if they are all x-normal for $\be$, then the probability that those four points belong to $B$ is $p^{-4k}$. The result follows.
\end{proof}

\begin{lemma}\label{threevps}
Let $w_1,w_2,h$ be such that $\be''(w_1,h)=\be''(w_2,h)=0$ and $h$ is y-normal for $\be'$. Let $(P_1,P_2,P_3)$ be a randomly chosen triple of vertical parallelograms of widths $w_1,w_2$ and $w_1+w_2$ and heights all equal to $h$. Then the probability that all the points of $P_1,P_2$ and $P_3$ are in $B$ is $p^{-9k}\pm 6p^{k-t}$.
\end{lemma}

\begin{proof}
The proof is similar to that of Lemma \ref{noof4arrs1}. This time necessary and sufficient conditions for the three parallelograms to lie in $B$ are that
\[\be'(x_1,h)=\be'(x_1+w_1,h)=\be'(x_2,h)=\be'(x_2+w_2,h)=\be'(x_3,h)=\be'(x_3+w_1+w_2,h)=0\eqno{(1)}\]
and that
\[\be(x_1,y_1)=\be(x_1+w_1,y_2)=\be(x_2,y_3)=\be(x_2+w_2,y_4)=\be(x_3,y_5)=\be(x_3+w_1+w_2,y_6)=0.\eqno{(2)}\]
By our hypothesis on $w_1,w_2$ and $h$, the first string of equalities holds if and only if
\[\be'(x_1,h)=\be'(x_2,h)=\be'(x_3,h)=0,\]
which is true with probability $p^{-3k}$, by the y-normality of $h$.

With probability at least $1-6p^{k-t}$, all of $x_1,x_1+w_1,x_2,x_2+w_2,x_3$ and $x_3+w_3$ are x-normal for $\be$, and if they are, then the probability that (2) holds is $p^{-6k}$. It follows that the probability that all three parallelograms are in $B$ is $p^{-9k}\pm 6p^{k-t}$, as claimed. 
\end{proof}

Before we move to the main results of this subsection, it will be helpful to generalize the notion of distance. Up to now we have talked just about distances between elements of $\Sigma(\cA)$ -- that is, non-negative functions that sum to 1. We shall now extend the definition so that it applies to all real-valued functions, though we shall apply it only for functions that take non-negative values. The obvious way to do this is the unique way that makes the distance bilinear: that is,
\[d(f,g)=\Bigl(\sum_xf(x)\Bigr)\Bigl(\sum_yg(y)\Bigr)-\langle f,g\rangle,\]
where as before we define the inner product using sums. This is indeed the definition we shall adopt, but we caution that the function is even less like a conventional distance than before, since now the triangle inequality ceases to hold. An easy way to see this is to observe that every function has distance zero from the zero function. However, we shall be careful to apply the triangle inequality only to functions that belong to $\Sigma(\cA)$. 

Using this notion of distance, we can reformulate the definition of a $(1-\eta)$-bihomomorphism. Previously, we said that a function $\phi:G^2\to\Sigma(\cA)$ was a $(1-\eta)$-bihomomorphism with respect to a function $\mu:G^2\to\R_+$ if 
\[\langle\mc(\mu\phi),\mc(\mu,\phi)\rangle\geq(1-\eta)\langle\mc\mu,\mc\mu\rangle,\]
where we defined $\langle\psi_1,\psi_2\rangle$ to be $\E_{w,h}\langle\psi_1(w,h),\psi_2(w,h)\rangle$. For a function $f\in\Sigma(\cA)$, write $\sigma f$ for the sum of the values of $f$. It is an easy exercise to check that for each $(w,h)$, we have that
\[\sigma(\mc(\mu\phi)(w,h))=\mc\mu(w,h).\]
Therefore, $\phi$ is a $(1-\eta)$-bihomomorphism with respect to $\mu$ if and only if
\[d(\mc(\mu\phi),\mc(\mu\phi))\leq\eta\langle\mc\mu,\mc\mu\rangle,\]
where the expression on the left-hand side is shorthand for $\E_{w,h}d(\mc(\mu\phi)(w,h),\mc(\mu\phi)(w,h))$. In the case that interests us most, $\mu$ is the characteristic function $b$ of the bilinear Bohr set $B$. Then the values of $\phi$ outside $B$ make no difference, so we are free to replace them by 0. And if we do that, then the inequality becomes
\[d(\mc\phi,\mc\phi)\leq\eta\langle\mc b,\mc b\rangle.\]
The right-hand side is the probability that a random 4-arrangement $(P_1,P_2)$ lies in $B$, which we have already shown is approximately $p^{-7k}$, while the left-hand side is the expectation of $d(\phi(P_1),\phi(P_2))$ over all 4-arrangements (this distance being zero unless both $P_1$ and $P_2$ lie entirely in $B$). So for $\phi$ to be a near bihomomorphism, we need the latter expectation to be small compared with $p^{-7k}$. 

We are now ready for the first step in the proof of the bilinear stability result. As with the proof of the linear stability result, our basic strategy is to expand and rearrange the expression we wish to bound, making suitable use of the triangle inequality. 

We shall use the notation $P(w,h;x,y,y')$ to stand for the vertical parallelogram
\[P=\Bigl((x,y), (x,y+h), (x+w,y'), (x+w,y'+h)\Bigr).\]
Recall that $\phi(P)$ stands for
\[\phi(x,y)\phi(x,y+h)^*\phi(x+w,y')^*\phi(x+w,y'+h).\]
We shall also define $\phi_L(P)$ to be $\phi(x,y)\phi(x,y+h)^*$ and $\phi_R(P)$ to be $\phi(x+w,y')\phi(x+w,y'+h)^*$, so that $\phi(P)=\phi_L(P)\phi_R(P)^*$. (The letters L and R stand for ``left" and ``right".)

\begin{lemma}\label{averagew}
Let $\phi:G^2\to\cA$ be a function such that $\phi(x,y)\in\Sigma(\cA)$ if $(x,y)\in B$ and $\phi(x,y)=0$ otherwise. Let $w_1,w_2$ and $h$ be such that $\be''(w_1,h)=\be''(w_2,h)=0$, and suppose that $h$ is y-normal for $\be'$. Then 
\[d\bigl(\psi(w_1,h)\psi(w_2,h),\psi(w_1+w_2,h)\bigr)\leq 2p^{-2k}\E_wd(\psi(w,h),\psi(w,h))\pm 7p^{2k-t}.\]
\end{lemma}

\begin{proof}
Let $x_1,x_2,x_3,y_1,\dots,y_7$ be chosen at random. Let $P_1=P(w_1,h;x_1,y_1,y_2)$, $P_2=P(w_2,h;x_2,y_3,y_4)$ and $P_3=P(w_1+w_2,h;x_3,y_5,y_6)$. Then the triple $(P_1,P_2,P_3)$ is uniformly distributed over all triples of vertical parallelograms of widths $w_1,w_2$ and $w_1+w_2$ and heights all equal to $h$. It follows that 
\[d\bigl(\psi(w_1,h)\psi(w_2,h),\psi(w_1+w_2,h)\bigr)=\mathop{\E}_{\substack{x_1,x_2,x_3\\ y_1,\dots,y_7}}d\bigl(\phi(P_1)\phi(P_2),\phi(P_3)\bigr).\]
So far, the parameter $y_7$ has played no role, but now we define four further parallelograms, namely $P_4=P(w_2+x_2-x_1,h;x_1,y_1,y_4)$,  $P_5=P(w_2+x_2-x_1,h;x_3,y_5,y_7)$, $P_6=P(x_2-x_1-w_1,h;x_1+w_1,y_2,y_3)$ and $P_7=P(x_2-x_1-w_1,h;x_3+w_1+w_2,y_6,y_7)$. Note that the pairs $(P_4,P_5)$ and $(P_6,P_7)$ are uniformly distributed amongst all 4-arrangements of height $h$. Note also that the right edges of $P_5$ and $P_7$ coincide, since $x_3+(w_2+x_2-x_1)=x_3+w_1+w_2+(x_2-x_1-w_1)$. Let us write $E$ for this common edge: that is, for the pair of points $\bigl((x_2+x_3-x_1+w_2,y_7),(x_2+x_3-x_1+w_2,y_7+h)\bigr)$, and $\phi(E)$ as shorthand for $\phi(x_2+x_3-x_1+w_2,y_7)\phi(x_2+x_3-x_1+w_2,y_7+h)^*$.

Now for all choices of the parameters, 
\[d\Bigl(\phi(P_1)\phi(P_2),\phi(P_3)\Bigr)=d\Bigl(\phi_L(P_1)\phi_R(P_2)^*\phi_L(P_3)^*,\phi_R(P_1)\phi_L(P_2)^*\phi_R(P_3)^*\Bigr).\]
If we know that all of $P_1,P_2,P_3$ and $E$ belong to $B$, then the above identity and the triangle inequality give us that
\begin{align*}
d\Bigl(\phi(P_1)\phi(P_2),\phi(P_3)\Bigr)&\leq d\Bigl(\phi_L(P_1)\phi_R(P_2)^*\phi_L(P_3)^*,\phi(E)^*\Bigr)+d\Bigl(\phi(E)^*,\phi_R(P_1)\phi_L(P_2)^*\phi_R(P_3)^*\Bigr)\\
&=d\Bigl(\phi_L(P_1)\phi_R(P_2)^*,\phi_L(P_3)\phi(E)^*\Bigr)+d\Bigl(\phi_R(P_1)\phi_L(P_2)^*,\phi_R(P_3)\phi(E)^*\Bigr)\\
&=d\Bigl(\phi(P_4),\phi(P_5)\Bigr)+d\Bigl(\phi(P_6),\phi(P_7)\Bigr),\\
\end{align*}
where the last equality follows from the constructions of $P_4,P_5,P_6$ and $P_7$. (For example, since $x_1+(w_2+x_2-x_1)=x_2+w_2$, we see that $P_4$ is made out of the left edge of $P_1$ and the right edge of $P_2$.) 

Let us write $\x$ for $(x_1,x_2,x_3)$ and $\y$ for $(y_1,\dots,y_7)$ and $b(\x,\y)$ for the function that is 1 if all of $P_1,\dots,P_7$ are in $B$ and 0 otherwise. Then the above calculations prove that
\[\E_{\x,\y}b(\x,\y)d\Bigl(\phi(P_1)\phi(P_2),\phi(P_3)\Bigr)\leq\E_{\x,\y}b(\x,\y)d\Bigl(\phi(P_4),\phi(P_5)\Bigr)+\E_{\x,\y}b(\x,\y)d\Bigl(\phi(P_6),\phi(P_7)\Bigr).\]

We now look carefully at the two sides of this inequality. We can rewrite the left-hand side as
\[\E_{(P_1,P_2,P_3)}d\Bigl(\phi(P_1)\phi(P_2),\phi(P_3)\Bigr)\E\bigl[b(\x,\y)|(P_1,P_2,P_3)\bigr].\]
The first expectation is over all triples with the appropriate weights and heights, and the conditional expectation is the expectation of $b(\x,\y)$ over all choices of $(\x,\y)$ that are consistent with the given triple $(P_1,P_2,P_3)$. 

It is not hard to work out what the conditional expectation is, since the triple $(P_1,P_2,P_3)$ determines and is determined by the parameters $x_1,x_2,x_3,y_1,\dots,y_6$. So all that is left to choose is $y_7$. If $P_1$, $P_2$ and $P_3$ do not all live in $B$, then $b(\x,\y)$ is guaranteed to be 0, and if they do all live in $B$, then the conditional expectation of $b(\x,\y)$ is the probability that the edge $E$ lives in $B$. But if $P_1,P_2$ and $P_3$ live in $B$, then $\be'(x_1,h)=\be'(x_2+w_2,h)=\be'(x_3,h)=0$, which implies that $\be'(x_2+x_3-x_1+w_2,h)=0$, so we are looking for the probability that $\be(x_2+x_3-x_1+w_2,h)=0$. Since $x_1,x_2$ and $x_3$ are chosen uniformly, $x_2+x_3-x_1+w_2$ is x-normal for $\be$ with probability at least $1-p^{k-t}$. Therefore, if we choose a random triple $(P_1,P_2,P_3)$ the conditional expectation is $p^{-k}$ with probability at least $1-p^{k-t}$, and since all the distances are between 0 and 1, this proves that the entire expression is equal to
\[p^{-k}\E_{\x,\y}d\bigl(\phi(P_1)\phi(P_2),\phi(P_3)\bigr)\pm p^{k-t}=p^{-k}d\bigl(\psi(w_1,h)\psi(w_2,h),\psi(w_1+w_2,h)\bigr)\pm p^{k-t}.\]

The two terms on the right-hand side can be dealt with similarly. Since there is a symmetry between them, we shall discuss just the first term. We rewrite it in a similar way, this time as
\[\E_{(P_4,P_5)}d\bigl(\phi(P_4),\phi(P_5)\bigr)\E\bigl[b(\x,\y)|(P_4,P_5)\bigr].\]
Let us set $x_2'$ to be $x_2+w_2$. Then the 4-arrangement $(P_4,P_5)$ determines and is determined by the variables $x_1,x_2',x_3,y_1,y_4$ and $y_7$. Since $x_2=x_2'-w_2$, that means that the only variables that it remains to choose are $y_2,y_3$ and $y_6$. If $P_4$ and $P_5$ are not both in $B$, then $b(\x,\y)=0$. Otherwise, we have that $\be'(x_1,h)=\be'(x_2',h)=\be'(x_3,h)=0$, and since $\be''(w_1,h)=\be''(w_2,h)=0$, it follows that $\be'(x_1+w_1,h)=\be'(x_2,h)=\be'(x_3+w_1+w_2,h)=0$. So we wish to know the probability, given $P_4$ and $P_5$, that $\be(x_1+w_1,y_2)=\be(x_2,y_3)=\be(x_3+w_1+w_2,y_6)=0$.
If $x_1+w_1,x_2'-w_2$ and $x_3+w_1+w_2$ are x-normal for $\be$, which happens with probability at least $1-3p^{k-t}$, then this probability is $p^{-3k}$. Therefore, the term we are estimating is equal to
\[p^{-3k}\E_{\x,\y}d\bigl(\phi(P_4),\phi(P_5)\bigr)\pm 3p^{k-t}=p^{-3k}\E_wd\bigl(\psi(w,h),\psi(w,h)\bigr)\pm 3p^{k-t}.\]
The term $\E_{\x,\y}b(\x,\y)d\bigl(\phi(P_6),\phi(P_7)\bigr)$ can be estimated in a similar way and gives the same result. Therefore, we obtain the estimate
\[p^{-k}d\bigl(\psi(w_1,h)\psi(w_2,h),\psi(w_1+w_2,h)\bigr)=2p^{-3k}\E_wd\bigl(\psi(w,h),\psi(w,h)\bigr)\pm 7p^{k-t},\]
which gives us the inequality we wanted.
\end{proof}

We shall be assuming that $\phi$ is a $(1-\eta)$-bihomomorphism, so we need a bound in terms of the average of $d\bigl(\psi(w,h'),\psi(w,h')\bigr)$ over all $w,h'$. The lemma above has made some progress towards that goal, by allowing us to average over $w$. But now we must replace the fixed $h$ by an average. This we shall do in a similar way, but there will be an interesting complication.

Given a function $\phi:G^2\to\cA$ and a parallelogram $P=P(w,h;x,y,y')$, we shall write $\phi_D(P)$ for $\phi(x,y)\phi(x+w,y')^*$ and $\phi_U(P)$ for $\phi(x,y+h)\phi(x+w,y'+h)^*$. (The letters D and U stand for ``down" and ``up".) We have $\phi(P)=\phi_D(P)\phi_U(P)^*$ for every vertical parallelogram $P$.

\begin{lemma}\label{averagewh}
Let $G=\F_p^n$, let $\theta>0$ and let $B\subset G^2$ be a bilinear Bohr set of codimension $k$ and rank $t$. Let $\cA$ be a group algebra and let $\phi:G^2\to\cA$ be a function such that $\phi(x,y)\in\Sigma(\cA)$ if $(x,y)\in B$ and $\phi(x,y)=0$ otherwise. Then 
\[\E_w d\bigl(\psi(w,h),\psi(w,h)\bigr)\leq 2\E_{w',h'}d\bigl(\psi(w',h'),\psi(w',h')\bigr)+2\theta\]
for all $h$ outside a set of density at most $12p^{26k-t}/\theta^2$.
\end{lemma}

\begin{proof}
When expanded, the left-hand side becomes
\[\mathop{\E}_{\substack{w,x_1,x_2\\y_1,y_2,y_3,y_4}}d\bigl(\phi(P_1),\phi(P_2)\bigr),\]
where $P_1=P(w,h;x_1,y_1,y_2)$ and $P_2=P(w,h;x_2,y_3,y_4)$.
As in the proof of Lemma \ref{averagew} we begin by decomposing $d\bigl(\phi(P_1),\phi(P_2)\bigr)$ and applying the triangle inequality. To do this, we introduce a new variable $h'$ (just as we introduced a new variable $y_7$ before), and this time we let $E_1$ and $E_2$ be the ``horizontal edges" $\bigl((x_1,y_1+h'),(x_1+w,y_2+h')\bigr)$ and $\bigl((x_2,y_3+h'),(x_2+w,y_4+h')\bigr)$, writing $\phi(E_1)$ and $\phi(E_2)$ for $\phi(x_1,y_1+h')\phi(x_1+w,y_2+h')^*$ and $\phi(x_2,y_3+h')\phi(x_2+w,y_4+h')^*$. Then as long as $P_1,P_2,E_1$ and $E_2$ are all inside $B$, we have that
\begin{align*}
d\bigl(\phi(P_1),\phi(P_2)\bigr)&=d\bigl(\phi_D(P_1)\phi_U(P_1)^*,\phi_D(P_2)\phi_U(P_2)^*\bigr)\\
&=d\bigl(\phi_D(P_1)\phi_D(P_2)^*,\phi_U(P_1)\phi_U(P_2)^*\bigr)\\
&\leq d\bigl(\phi_D(P_1)\phi_D(P_2)^*,\phi(E_1)\phi(E_2)^*\bigr)+d\bigl(\phi(E_1)\phi(E_2)^*,\phi_U(P_1)\phi_U(P_2)^*\bigr)\\
&=d\bigl(\phi_D(P_1)\phi(E_1)^*,\phi_D(P_2)\phi(E_2)^*\bigr)+d\bigl(\phi(E_1)\phi_U(P_1)^*,\phi(E_2)^*\phi_U(P_2)^*\bigr)\\
&=d\bigl(\phi(P_3),\phi(P_4)\bigr)+d\bigl(\phi(P_5),\phi(P_6)\bigr),\\
\end{align*}
where we have set $P_3=P(w,h';x_1,y_1,y_2)$, $P_4=P(w,h';x_2,y_3,y_4)$, $P_5=P(w,h-h';x_1,y_1+h',y_2+h')$ and $P_6=P(w,h-h';x_2,y_3+h',y_4+h')$. (Geometrically, we have sliced each of $P_1$ and $P_2$ into parallelograms of heights $h'$ and $h-h'$.)

The pairs $(P_3,P_4)$ and $(P_5,P_6)$ are uniformly distributed amongst all 4-arrangements, but this, though extremely important to us, is not quite as useful as it might at first look. To see where the complication arises, let us now try to use the above inequality in the way that we used a similar inequality in the proof of Lemma \ref{averagew}. This time we shall set $\x=(x_1,x_2)$ and $\y=(y_1,y_2,y_3,y_4)$, and we shall let $b(w,h',\x,\y)$ equal 1 if all of $P_1,P_2,E_1$ and $E_2$ are inside $B$ and 0 otherwise. Then the above argument gives us that
\begin{align*}
\mathop{\E}_{w,h',\x,\y}b(w,h',\x,\y)d\bigl(\phi(P_1),\phi(P_2)\bigr)\leq\mathop{\E}_{w,\x,\y}b(w,h',\x,\y)&d\bigl(\phi(P_3),\phi(P_4)\bigr)\\
&+\mathop{\E}_{w,h',\x,\y}b(w,h',\x,\y)d\bigl(\phi(P_5),\phi(P_6)\bigr).\\
\end{align*}

Again we discuss the terms separately, beginning with the left-hand side. For each $(P_1,P_2)$, the weight attached to $d\bigl(\phi(P_1),\phi(P_2)\bigr)$ is the expectation of $b(w,\x,\y)$ given $(P_1,P_2)$. If $P_1$ and $P_2$ are not both in $B$, this is zero. Otherwise, it is the probability that $(x_1,y_1+h')$, $(x_1+w,y_2+h')$, $(x_2,y_3+h')$ and $(x_2+w_2,y_3+h')$ all belong to $B$ when $h'$ is chosen uniformly at random. (The pair $(P_1,P_2)$ determines and is determined by the other parameters $w$, $\x$ and $\y$.) We are given that $(x_1,y_1), (x_1+w,y_2), (x_2,y_3)$ and $(x_2+w,y_4)$ belong to $B$, so a necessary and sufficient condition for this is that
\[\be'(x_1,h')=\be'(x_1+w,h')=\be'(x_2,h')=\be'(x_2+w,h')=0.\]
The last equality here is redundant, since $x_2+w=(x_1+w)+x_2-x_1$.

The triple $(x_1,x_1+w,x_2)$ is uniformly distributed in $G^3$. Therefore, it is x-normal for $\be'$ with probability at least $1-p^{3k-t}$, and whenever it is x-normal for $\be'$, the probability that the equations hold is $p^{-3k}$. It follows that
\[\mathop{\E}_{w,h',\x,\y}b(w,h',\x,\y)d\bigl(\phi(P_1),\phi(P_2)\bigr)=p^{-3k}\mathop{\E}_{w,\x,\y}d\bigl(\phi(P_1),\phi(P_2)\bigr)\pm p^{3k-t}.\]

Now let us turn our attention to the term $\mathop{\E}_{w,\x,\y}b(w,h',\x,\y)d\bigl(\phi(P_3),\phi(P_4)\bigr)$. It is here that something a little curious happens: the 4-arrangement $(P_3,P_4)$ determines all the parameters $w,h',\x,\y$, so the expectation of $b(w,h',\x,\y)$ given $(P_3,P_4)$ is equal to the value of $b$ for that choice of parameters. If $P_3$ and $P_4$ are not both in $B$, then we get zero, as usual, but if $P_3$ and $P_4$ \emph{are} both in $B$ it does \emph{not} follow that $b(w,h',\x,\y)=1$: this happens only if the points $(x_1,y_1+h), (x_1+w,y_2+h), (x_2,y_3+h)$ and $(x_2+w,y_4+h)$ all lie in $B$.

Thus, the term $\mathop{\E}_{w,\x,\y}b(w,h',\x,\y)d\bigl(\phi(P_3),\phi(P_4)\bigr)$ is proportional to the sum of the distances $d\bigl(\phi(P_3),\phi(P_4)\bigr)$ over only a small subset of the 4-arrangements $(P_3,P_4)$ that live in $B$.

This looks like a serious problem, but what rescues us is that the subset is a quasirandom sample. This will allow us to conclude, using Lemma \ref{qrsample}, that for almost every $h$ we obtain the inequality we were hoping for.

Let $\cQ$ be the set of all 4-arrangements, let $\cQ^B$ be the set of all 4-arrangements that live in $B$, and for each $h$ let $\cQ^B(h)$ be the set of all 4-arrangements $(P_3,P_4)$ in $\cQ^B$ such that if $P_3=P(w,h';x_1,y_1,y_2)$ and $P_4=P(w,h';x_2,y_3,y_4)$, then the points $(x_1,h), (x_1+w,h)$ and $(x_2,h)$ lie in $B'$. Note that this implies that the points $(x_1,y_1+h), (x_1+w,y_2+h)$ and $(x_2,y_3+h)$ lie in $B$ (for example, $\be(x_1+w,y_2+h)=\be(x_1+w)+\be'(x_1,h)$), which in turn implies that $(x_2+w,y_4+h)$ lies in $B$ (since if you have seven points of a 4-arrangement then you must have the eighth as well). 

If we choose a random 4-arrangement $(P_3,P_4)\in\cQ$ with the parameters as in the previous paragraph, then the triple $(x_1,x_1+w,x_2)$ is uniformly distributed in $G^3$, so the probability that $(x_1,x_1+w,x_2)$ is not x-normal for $\be'$ is at most $p^{3k-t}$. Therefore, by Lemma \ref{noof4arrs2}, the probability that it is not x-normal given that $(P_3,P_4)\in\cQ^B$ is at most $2p^{10k-t}$, as long as $p^{-t}\leq p^{-7k}/12$. (With more care, one can improve on this bound, but that does not gain us anything interesting.) 

If the triple $(x_1,x_1+w,x_2)$ is x-normal for $\be'$, then the probability that $(x_1,h),(x_1+w,h)$ and $(x_2,h)$ belong to $B'$ is $p^{-3k}$ for each $h$. For such 4-arrangements $(P_3,P_4)$ it follows that if we choose $h$ at random, then the probability that $(P_3,P_4)\in\cQ^B(h)$ is $p^{-3k}$.

Similarly, if we choose $(P_3,P_4)$ and $(P_3',P_4')$ at random, and write $(x_1',y_1'), (x_1'+w',y_2')$ and $(x_2',y_3')$ for ``the first three points on the bottom row" of $(P_3',P_4')$, then the sextuple $(x_1,x_1+w,x_2,x_1',x_1'+w',x_2')$ is uniformly distributed, so has a probability at most $p^{6k-t}$ of failing to be x-normal for $\be'$, which rises to at most $4p^{20k-t}$ if we condition on both $(P_3,P_4)$ and $(P_3',P_4')$ belonging to $\cQ^B$. If the sextuple \emph{is} x-normal for $\be'$ and we choose a random $h$, the probability that $(P_3,P_4)$ and $(P_3',P_4')$ both belong to $\cQ^B(h)$ is $p^{-6k}$.

This gives us the conditions we need in order to apply Lemma \ref{qrsample}, with $X=\cQ^B$, the sets $\cQ^B(h)$ as the $A_i$, and $\theta$ replacing $\eta$. We can set $\e_1=2p^{10k-t}$, $\e_2=4p^{20-k}$, and $\a=p^{-3k}$. Then for all $h$ outside a set of density at most $6p^{26-k}/\theta^2$ we have that 
\[\bigl|\mathop{\E}_{w,h',\x,\y}b(w,h',\x,\y)d\bigl(\phi(P_3),\phi(P_4)\bigr)-p^{-3k}\mathop{\E}_{w,h',\x,\y}d\bigl(\phi(P_3),\phi(P_4)\bigr)\bigr|\leq p^{-3k}\theta.\]
The same argument works for $(P_5,P_6)$ and gives the same average, so we have shown that
\[p^{-3k}\E_{w,\x,\y}d\bigl(\phi(P_1),\phi(P_2)\bigr)\pm p^{3k-t}\leq 2p^{-3k}\E_{w,h',\x,\y}d\bigl(\phi(P_3),\phi(P_4)\bigr)+2p^{-3k}\theta\]
for all $h$ outside a set of density at most $12p^{26k-t}/\theta^2$, which implies, for all such $h$, that
\[\E_wd\bigl(\psi(w,h),\psi(w,h)\bigr)\leq 2\E_{w',h'}d\bigl(\psi(w',h'),\psi(w',h')\bigr)+2\theta.\]
This proves the lemma. \end{proof}

Note that in the above lemma there is nothing to stop $\theta$ depending on $k$, and indeed it is essential that it should be able to do so, since we need it to be small compared with $p^{-7k}$. 

\subsection{The bilinear case: obtaining linearity almost everywhere in the second variable}

We shall now prove a similar result to that of the last section, but this time obtaining linearity in the second variable. Since the proof is also similar, we shall present it more concisely.

\begin{lemma}\label{averageh}
Let $\phi:G^2\to\cA$ be a function such that $\phi(x,y)\in\Sigma(\cA)$ if $(x,y)\in B$ and $\phi(x,y)=0$ otherwise. Let $w, h_1$ and $h_2$ be such that $\be''(w,h_1)=\be''(w,h_2)=0$, and suppose that $w$ is x-normal for $\be'$. Then for every $\theta\in(0,1]$ we have that
\[d\bigl(\psi(w,h_1)\psi(w,h_2),\psi(w,h_1+h_2)\bigr)\leq 2p^{-2k}\E_hd(\psi(w,h),\psi(w,h))+10p^{13k-t}+6\theta p^{-7k}\]
provided that $w$ lies outside a set of density at most $p^{k-t}$ and $(h_1,h_2)$ lies outside a set of density at most $(25p^{16k-t})/\theta^2$.
\end{lemma}

\begin{proof}
Let $x_1,x_2,x_3,y_1,\dots,y_6$ be chosen at random. Let $P_1=P(w,h_1;x_1,y_1,y_2)$, $P_2=P(w,h_2;x_2,y_3,y_4)$ and $P_3=P(w,h_1+h_2;x_3,y_5,y_6)$. Then the triple $(P_1,P_2,P_3)$ is uniformly distributed over all triples of vertical parallelograms of widths all equal to $w$ and heights $h_1,h_2$ and $h_1+h_2$.

Now let $h$ be a new parameter and define the following six further parallelograms. Writing $h_3$ for $h_1+h_2$, we let $P_4=P(w,h_3+h;x_1,y_1,y_2)$, $P_5=P(w,h_2+h;x_1,y_1+h_1,y_2+h_1)$, $P_6=P(w,h_2+h;x_2,y_3,y_4)$, $P_7=P(w,h;x_2,y_3+h_2,y_4+h_2)$, $P_8=P(w,h_3+h;x_3,y_5,y_6)$, and $P_9=P(w,h;x_3,y_5+h_3,y_6+h_3)$. We shall also define the three ``horizontal edges" $F_1=\bigl((x_1,y_1+h_3+h),(x_1+w,y_2+h_3+h)\bigr)$, $F_2=\bigl((x_2,y_3+h_2+h),(x_2+w,y_4+h_2+h)\bigr)$, and $F_3=\bigl((x_3,y_5+h_3+h),(x_3+w,y_6+h_3+h)\bigr)$, each of which belongs to two of the six parallelograms above.

We shall now apply the triangle inequality twice. First, note that
\begin{align*}
d\bigl(\phi(P_1)\phi(P_2),\phi(P_3)\bigr)&=d\bigl(\phi_D(P_1)\phi_D(P_3)^*,\phi(P_2)^*\phi_U(P_1)\phi_U(P_3)^*\bigr)\\
&\leq d\bigl(\phi_D(P_1)\phi_D(P_3)^*,\phi(F_1)\phi(F_3)^*\bigr)+d\bigl(\phi(F_1)\phi(F_3)^*,\phi(P_2)^*\phi_U(P_1)\phi_U(P_3)^*\bigr)\\
\end{align*}
The first term in the last line is equal to $d\bigl(\phi_D(P_1)\phi(F_1)^*,\phi_D(P_3)\phi(F_3)^*\bigr)$, which is equal to $d\bigl(\phi(P_4),\phi(P_8)\bigr)$. The second term is equal to 
\begin{align*}
d\bigl(\phi_U(P_1)&\phi(F_1)^*\phi_L(P_2)^*,\phi_U(P_3)\phi(F_3)^*\phi_U(P_2)^*\bigr)\\
&\leq d\bigl(\phi_U(P_1)\phi(F_1)^*\phi_L(P_2)^*,\phi(F_2)^*\bigr)+d\bigl(\phi(F_2)^*,\phi_U(P_3)\phi(F_3)^*\phi_U(P_2)^*\bigr)\\
&=d\bigl(\phi_U(P_1)\phi(F_1)^*,\phi_L(P_2)\phi(F_2)^*\bigr)+d\bigl(\phi_U(P_2)\phi(F_2)^*,\phi_U(P_3)\phi(F_3)^*\bigr)\\
&=d\bigl(\phi(P_5),\phi(P_6)\bigr)+d\bigl(\phi(P_7),\phi(P_9)\bigr).\\
\end{align*}
Putting all this together, we have proved that for every $h$ we have the inequality
\[d\bigl(\phi(P_1)\phi(P_2),\phi(P_3)\bigr)\leq d\bigl(\phi(P_4),\phi(P_8)\bigr)+d\bigl(\phi(P_5),\phi(P_6)\bigr)+d\bigl(\phi(P_7),\phi(P_9)\bigr).\]
Note that each pair $(P_4,P_8)$, $(P_5,P_6)$ and $(P_7,P_9)$ is a 4-arrangement of width $w$. Moreover, for any fixed choice of $h_1$ and $h_2$, each pair is uniformly distributed over all such 4-arrangements when the remaining parameters $h,x_1,x_2,x_3,y_1,\dots,y_6$ are chosen independently and uniformly.

Now let us write $\x$ for $(x_1,x_2,x_3)$ and $\y$ for $(y_1,\dots,y_6)$ and set $b(h,\x,\y)$ to be 1 if $P_1,P_2,P_3,F_1,F_2$ and $F_3$ (and hence $P_4,\dots,P_9$) all live in $B$. Then the above inequality implies that
\begin{align*}\mathop{\E}_{h,\x,\y}&b(h,\x,\y)d\bigl(\phi(P_1)\phi(P_2),\phi(P_3)\bigr)\\
&\leq\mathop{\E}_{h,\x,\y}b(h,\x,\y)d\bigl(\phi(P_4),\phi(P_8)\bigr)+\mathop{\E}_{h,\x,\y}b(h,\x,\y)d\bigl(\phi(P_5),\phi(P_6)\bigr)+\mathop{\E}_{h,\x,\y}b(h,\x,\y)d\bigl(\phi(P_7),\phi(P_9)\bigr).\\
\end{align*}

As in the previous section, we shall now approximately evaluate the two sides of this inequality in order to obtain the statement in the lemma, beginning with the left-hand side.

If $P_1,P_2$ and $P_3$ all belong to $B$, then the probability that $F_1$, $F_2$ and $F_3$ also belong to $B$ is equal to the probability that $\be'(x_1,h_3+h)=\be'(x_2,h_2+h)=\be'(x_3,h_3+h)=\be''(w,h)=0$, since these equations imply that $\be'(x_1+w,h_3+h)=\be'(x_2+w,h_2+h)=\be'(x_3+w,h_3+h)=0$, and then, given that the lower edges of $P_1,P_2$ and $P_3$ belong to $B$ we can deduce that $F_1,F_2$ and $F_3$ do as well.

Recall that $w$ is x-normal for $\be''$. With probability at least $1-p^{4k-t}$, the sequence $(x_1,x_2,x_3)$ is x-normal for $\be'$ when we restrict to the $k$-codimensional subspace consisting of those $h$ for which $\be''(w,h)=0$, and if it is, then all values of $(\be'(x_1,h),\be'(x_2,h),\be'(x_3,h))$ occur with probability $p^{-3k}$. Therefore, the probability that $\be''(w,h)=0$ is $p^{-k}$, and given this event and the normality of $(x_1,x_2,x_3)$ in the subspace, the probability that $\be'(x_1,h)=-\be'(x_1,h_3)$, $\be'(x_2,h)=-\be'(x_2,h_2)$ and $\be'(x_3,h)=-\be'(x_3,h_3)$ is $p^{-3k}$. It follows that 
\[\mathop{\E}_{h,\x,\y}b(h,\x,\y)d\bigl(\phi(P_1)\phi(P_2),\phi(P_3)\bigr)=p^{-4k}\mathop{\E}_{\x,\y}d\bigl(\phi(P_1)\phi(P_2),\phi(P_3)\bigr)\pm p^{4k-t}.\]

Again as in the previous section, it is enough to look at just one term on the right-hand side. Let $(P_4,P_8)$ be an arbitrary 4-arrangement of width $w$. Note that if we are given the information that $P_4=P(w,h_3+h;x_1,y_1,y_2)$ and $P_8=P(w,h_3+h;x_3,y_5,y_6)$, then (since $w$, $h_1$ and $h_3$ are fixed) the parameters $x_1,y_1,y_2,x_3,y_5,y_6$ and $h$ are determined, leaving just the parameters $x_2,y_3$ and $y_4$ to be chosen.

If $\be'(x_1,h_1)\ne 0$, then the point $(x_1,y_1+h_1)$ does not belong to $B$, so $b(h,\x,\y)=0$. Similarly, if $\be'(x_3,h_3)\ne 0$, then $b(h,\x,\y)=0$. So $\E[b(h,\x,\y|(P_4,P_8)]=0$ unless $\be'(x_1,h_1)=\be'(x_3,h_3)=0$, so once again the best we can hope for is a quasirandom sample, this time of the 4-arrangements of width $w$ that live in $B$, which is indeed what we shall show we have.

If those two conditions hold, then the free parameters $x_2,y_3,y_4$ must satisfy the conditions $\be'(x_2,h_2)=\be'(x_2,h)=0$ and $\be(x_2,y_3)=\be(x_2+w,y_4)=0$. If $h_2$ is y-normal for $\be'$, then the probability that $h$ is y-normal for the restriction to the subspace $\{x:\be'(x,h_1)=0\}$ is at least $1-p^{2k-t}$, so the conditional probability that $(h,h_2)$ is y-normal for $\be'$ is at least $1-p^{2k-t}$. If it is, then $\be'(x_2,h)=\be'(x_2,h_2)=0$ with probability $p^{-2k}$. The probability that $x_2$ and $x_2+w$ are x-normal for $\be$ is at least $1-2p^{2k-t}$, and if we condition on $\be'(x_2,h)=\be'(x_2,h_2)=0$ it is still at least $1-2p^{4k-t}$. If they are x-normal for $\be$, then $\be(x_2,y_3)=\be(x_2+w,y_4)=0$ with probability $p^{-2k}$. 

All told, therefore, the expectation of $b(h,\x,\y)$ is zero unless $\be'(x_1,h_1)=\be'(x_3,h_3)=0$, in which case it is $p^{-4k}\pm 3p^{4k-t}$. 

We now prove that this gives us a quasirandom sample. This time we let $\cQ$ be the set of all 4-arrangements $(P_4,P_8)$ of width $w$, and we let $\cQ^B$ be the set of all these that live in $B$. And given $(h_1,h_3)$ we let $\cQ^B(h_1,h_3)$ be the set of all of them such that $\be'(x_1,h_1)=\be'(x_3,h_3)=0$, where the choice of parameters is as above.

If we choose a random 4-arrangement in $\cQ$, then the probability that it belongs to $\cQ^B$ is at least $p^{-7k}/2$, by Lemma \ref{noof4arrs2}. The probability that $x_1$ and $x_3$ are both x-normal for $\be'$ is at least $1-2p^{k-t}$, so the probability that they are both x-normal for $\be'$ given that $(P_4,P_8)\in\cQ^B$ is at least $1-4p^{8k-t}$. (Again, this can be improved, since with a little more effort one can show that $x_1$ and $x_3$ are approximately uniformly distributed even conditioned on $(P_4,P_8)\in\cQ^B$, but again the improvement is not important.) 

We therefore have the following statement. If $(P_4,P_8)$ is a random element of $\cQ^B$, then with probability at least $1-4p^{8k-t}$ the density of $(h_1,h_3)$ such that $\be'(x_1,h_1)=\be'(x_3,h_3)=0$ is $p^{-2k}$. That is, with probability at least $1-4p^{8k-t}$, the density of $(h_1,h_3)$ such that $(P_4,P_8)\in\cQ^B(h_1,h_3)$ is $p^{-2k}$. 

Similarly, if $(P_4,P_8)$ and $(P_4',P_8')$ are two random elements of $\cQ^B$, then the pairs $(x_1,x_1')$ and $(x_3,x_3')$ are x-normal for $\be'$ with probability at least $1-8p^{16k-t}$, and if they are normal, then the density of $(h_1,h_3)$ such that both $(P_4,P_8)$ and $(P_4',P_8')$ are in $\cQ(h_1,h_3)$ is $p^{-4k}$. 

We now apply Lemma \ref{qrsample} with $\a=p^{-2k}$, $\e_1=4p^{8k-t}$, $\e_2=8p^{16k-t}$ and $\theta>0$ to be chosen. For each $(P_4,P_8)\in\cQ^B$ let $f(P_4,P_8)=d\bigl(\phi(P_4),\phi(P_8)\bigr)$. Then Lemma \ref{qrsample} implies that
\[\Bigl|\E_{(P_4,P_8)\in\cQ^B}d\bigl(\phi(P_4),\phi(P_8)\bigr)\b1_{\cQ^B(h_1,h_3)}(P_4,P_8)-p^{-2k}\E_{(P_4,P_8)\in\cQ^B}d\bigl(\phi(P_4),\phi(P_8)\bigr)\Bigr|\leq\theta\]
for all $(h_1,h_3)$ outside a subset of density at most $(8p^{6k-t}+8p^{16k-t})/\theta^2$. Since $d\bigl(\phi(P_4),\phi(P_8)\bigr)=0$ when $(P_4,P_8)\notin\cQ^B$, and since $\cQ^B$ has density at most $2p^{-7k}$ in $\cQ$, it follows that
\[\Bigl|\E_{(P_4,P_8)\in\cQ}\,d\bigl(\phi(P_4),\phi(P_8)\bigr)\b1_{\cQ^B(h_1,h_3)}(P_4,P_8)-p^{-2k}\E_{(P_4,P_8)\in\cQ}\,d\bigl(\phi(P_4),\phi(P_8)\bigr)\Bigr|\leq 2\theta p^{-7k}.\]

If $(P_4,P_8)\in Q^B(h_1,h_3)$, then, as we have already pointed out, the expectation of $b(h,\x,\y)$ given $(P_4,P_8)$ is $p^{-4k}\pm 3p^{4k-t}$, and otherwise it is zero, except that if $(h,h_2)$ is not y-normal for $\be'$ and $(P_4,P_8)\in Q^B(h_1,h_3)$, then we just know that the expectation is between 0 and 1. The last event happens with probability at most $p^{2k-t}$. It follows that
\begin{align*}
\Bigl|p^{-4k}\E_{(P_4,P_8)\in\cQ}d\bigl(\phi(P_4),\phi(P_8)\bigr)&\b1_{\cQ^B(h_1,h_3)}(P_4,P_8)\\
&-\E_{(P_4,P_8)\in\cQ}d\bigl(\phi(P_4),\phi(P_8)\bigr)\E[b(h,\x,\y)|(P_4,P_8)]\Bigr|\leq 4p^{4k-t}.\\
\end{align*}
The second term is equal to $\E_{h,\x,\y}b(h,\x,\y)d\bigl(\phi(P_4),\phi(P_8)\bigr)$, so by the triangle inequality we deduce that
\[\E_{h,\x,\y}b(h,\x,\y)d\bigl(\phi(P_4),\phi(P_8)\bigr)\leq p^{-6k}\E_{(P_4,P_8)\in\cQ}\,d\bigl(\phi(P_4),\phi(P_8)\bigr)+ 3p^{9k-t}+2\theta p^{-11k}.\]

The same argument works for the other two terms (indeed, after suitable changes of variable we can use precisely the same words). Putting everything together then gives us the inequality
\begin{align*}\E_{\x,\y}&d\bigl(\phi(P_1)\phi(P_2),\phi(P_3)\bigr)\\
&\leq p^{-2k}\E_{h,\x,\y}\bigl(d\bigl(\phi(P_4,P_8)\bigr)+d\bigl(\phi(P_5,P_6)\bigr)+d\bigl(\phi(P_7,P_9)\bigr)\bigr)
+p^{8k-t}+9p^{13k-t}+6\theta p^{-7k}\\
\end{align*}
for all pairs $(h_1,h_2)$ outside a subset of density at most $(24p^{6k-t}+24p^{16k-t})/\theta^2+3p^{k-t}$, which is equivalent to the result claimed.
\end{proof}

As we have already noted, the ``trivial" upper bound for $\E_{h,\x,\y}d\bigl(\phi(P_4),\phi(P_8)\bigr)$ is $2p^{-7k}$, since this is an upper bound for the probability that a random 4-arrangement of width $w$ lies in $B$. Therefore, for the inequality just proved to be of any interest we need to choose $\theta$ small compared with $p^{-2k}$.

It remains to bound the expectation of $d\bigl(\phi(P_1),\phi(P_2)\bigr)$ over 4-arrangements $(P_1,P_2)$ of width $w$ by the expectation over \emph{all} 4-arrangements. 

For the next lemma, let $\cQ(w)$ be the set of all 4-arrangements of width $w$ and let $\cQ^B(w)$ be the set of 4-arrangements of width $w$ that lie in $B$.

\begin{lemma}\label{averagehw}
Let $\theta>0$. Then for all $w$ outside a set of density at most $(4p^{7k-t}+4p^{16k-t})/\theta^2$, we have the inequality
\[\E_{(P_1,P_2)\in\cQ(w)}d\bigl(\phi(P_1),\phi(P_2)\bigr)\leq 2\E_{(P_3,P_5)\in\cQ}\,d\bigl(\phi(P_3),\phi(P_5)\bigr)+2p^{-2k}\theta+4p^{k-t}.\]
\end{lemma}

\begin{proof}
Let $x_1,x_2,y_1,y_2,y_3,y_4$ and $h$ be random elements of $\F_p^n$, and let $P_1=P(w,h;x_1,y_1,y_2)$ and $P_2=P(w,h;x_2,y_3,y_4)$. 
Now let $y_5,y_6$ and $u$ be three more random elements of $\F_p^n$ and define vertical edges $E_1=\bigl((x_1+u,y_5),(x_1+u,y_5+h)\bigr)$ and $E_2=\bigl((x_2+u,y_6),(x_2+u,y_6+h)\bigr)$. Using $E_1$ and $E_2$ and the vertical edges of $P_1$ and $P_2$ we can create four more vertical parallelograms, namely $P_3=P(u,h;x_1,y_1,y_5)$, $P_4=P(w-u,h;x_1+u,y_5,y_2)$, $P_5=P(u,h;x_2,y_3,y_6)$ and $P_6=P(w-u,h;x_2+u,y_6,y_4)$.

Write $\x$ for $(x_1,x_2)$ and $\y$ for $(y_1,y_2,y_3,y_4,y_5,y_6)$. Then for all choices of $h,u,\x$ and $\y$ we have that
\begin{align*}
d\bigl(\phi(P_1),\phi(P_2)\bigr)&=d\bigl(\phi_L(P_1)\phi_L(P_2)^*,\phi_R(P_1)\phi_R(P_2)^*\bigr)\\
&\leq d\bigl(\phi_L(P_1)\phi_L(P_2)^*,\phi(E_1)\phi(E_2)^*\bigr)+d\bigl(\phi(E_1)\phi(E_2)^*,\phi_R(P_1)\phi_R(P_2)^*\bigr)\\
&=d\bigl(\phi_L(P_1)\phi(E_1)^*,\phi_L(P_2)\phi(E_2)^*\bigr)+d\bigl(\phi(E_1)\phi_R(P_1)^*,\phi(E_2)\phi_R(P_2)^*\bigr)\\
&=d\bigl(\phi(P_3),\phi(P_5)\bigr)+d\bigl(\phi(P_4),\phi(P_6)\bigr).\\
\end{align*}
Note that the 4-arrangements $(P_3,P_5)$ and $(P_4,P_6)$ are uniformly distributed amongst all 4-arrangements.

As usual, we now introduce a function that tells us when the entire collection of parallelograms belongs to $B$. Given $u,h,\x,\y$, let $b(u,h,\x,\y)=1$ if all the vertices of all the parallelograms above are in $B$ and let $b(u,h,\x,\y)=0$ otherwise. Then from the above calculation we deduce that
\[\mathop{\E}_{u,h,\x,\y}b(u,h,\x,\y)d\bigl(\phi(P_1),\phi(P_2)\bigr)\leq\mathop{\E}_{u,h,\x,\y}b(u,h,\x,\y)\bigl(d\bigl(\phi(P_3),\phi(P_5)\bigr)+d\bigl(\phi(P_4),\phi(P_6)\bigr)\bigr).\]

Now we estimate the three terms above, beginning with the term on the left-hand side. Fix a 4-arrangement $(P_1,P_2)$ of width $w$ and height $h$. If $h$ is y-normal for $\be''$, which it is unless it lies in an exceptional set of density at most $p^{k-t}$, then the proportion of $u$ with $\be''(u,h)=0$ is $p^{-k}$. If $x_1+u$ and $x_2+u$ are x-normal for $\be$, which they are unless $u$ lies in an exceptional set of density at most $2p^{k-t}$, then the proportion of $(y_5,y_6)$ such that $\be(x_1+u,y_5)=\be(x_2+u,y_6)=0$ is $p^{-2k}$. And if those equations all hold, then $b(u,h,\x,\y)=1$, since
\begin{align*}\be(x_1+u,y_5+h)&=\be(x_1+u,y_6)+\be'(x_1+u,h)\\
&=\be(x_1+u,y_5)+\be'(x_1,h)+\be''(u,h)\\
&=\be(x_1+u,y_5)+\be(x_1,y_1+h)-\be(x_1,y_1)+\be''(u,h)\\
&=0,\\
\end{align*}
and similarly $\be(x_2+u,y_6+h)=0$. 

It follows that
\[\mathop{\E}_{u,h,\x,\y}b(u,h,\x,\y)d\bigl(\phi(P_1),\phi(P_2)\bigr)=p^{-3k}\E_{h,\x,\y}d\bigl(\phi(P_1),\phi(P_2)\bigr)\pm 3p^{k-t}.\]

Now let us look at the term $\E_{u,h,\x\,y}b(u,h,\x,\y)d\bigl(\phi(P_3),\phi(P_5)\bigr)$. Once we know $P_3$ and $P_5$ the parameters $x_1,x_2,y_1,y_3,y_5,y_6,u$ and $h$ are determined, and it remains to choose $y_2$ and $y_4$. If we know that $(P_3,P_4)$ is in $B$, then for $b(u,h,\x,\y)$ to equal 1 we need 
\[\be(x_1+w,y_2)=\be(x_1+w,y_2+h)=\be(x_2+w,y_4)=\be(x_2+w,y_4+h)=0.\]
This is possible only if $\be'(x_1+w,h)=\be'(x_2+w,h)=0$. Since we already know that $\be'(x_1,h)=\be'(x_2,h)=0$ (from the fact that $(P_3,P_5)$ is in $B$), it follows that the one requirement that must be satisfied is that $\be''(w,h)=0$. Since this is not satisfied for all 4-arrangements $(P_3,P_5)$, we must again argue that we have a quasirandom sample. 

First, let us assume that $\be''(w,h)=0$ and estimate the expected value of $\be(u,h,\x,\y)$ given $(P_3,P_5)$. If $x_1+w$ and $x_2+w$ are x-normal for $\be$, then the density of $(y_2,y_4)$ such that $\be(x_1+w,y_2)=\be(x_2+w,y_4)=0$ is $p^{-2k}$. And if those equations hold, then we also have the equations $\be(x_1+w,y_2+h)=\be(x_2+w,y_4+h)=0$, since $\be'(x_1,h)=\be'(x_2,h)=0$ and we are assuming that $\be''(w,h)=0$. The density of $(x_1,x_2)$ such that $x_1+w$ and $x_2+w$ are x-normal for $\be$ is at least $1-2p^{k-t}$, so this implies that the expected value is $p^{-2k}\pm 2p^{k-t}$. 

Now we show the quasirandomness. Let us define $\cQ[w]$ to be the set of 4-arrangements of height $h$ with $\be''(w,h)=0$ and $\cQ^B[w]$ to be the set of all those that lie in $B$. Given a 4-arrangement $(P_3,P_5)\in\cQ^B$ of height $h$, if $h$ is y-normal for $\be''$, then the proportion of $w$ with $(P_3,P_5)\in\cQ^B[w]$ is $p^{-k}$. Given two 4-arrangements $(P_3,P_5)$ and $(P_3',P_5')$ in $\cQ^B$ of heights $h$ and $h'$, if $(h,h')$ is y-normal for $\be''$, then the proportion of $w$ such that both $(P_3,P_5)$ and $(P_3',P_5')$ are in $\cQ^B[w]$ is $p^{-2k}$. We may therefore apply Lemma \ref{qrsample} with $\a=p^{-k}$, the $A_i$ being the sets $\cQ^B[w]$, $X$ being $\cQ^B$, $\e_1=2p^{8k-t}$, and $\e_2=2p^{16k-t}$. (Again, these choices of $\e_1$ and $\e_2$ are inefficient: we are ignoring the fact that the height is roughly uniformly distributed when we choose a random $(P_3,P_5)\in\cQ^B$.) The lemma gives us that for every $\theta>0$ we have the inequality
\[\Bigl|\mathop{\E}_{(P_3,P_5)\in\cQ}\b1_{\cQ^B[w]}(P_3,P_5)d\bigl(\phi(P_3),\phi(P_5)\bigr)-p^{-k}\mathop{\E}_{(P_3,P_5)\in\cQ}d\bigl(\phi(P_3),\phi(P_5)\bigr)\Bigr|\leq\theta\]
as long as $w$ lies outside a set of  at most $(2p^{7k-t}+2p^{16k-t})/\theta^2$. But the estimate in the previous paragraph also gives us that
\[\Bigl|p^{-2k}\mathop{\E}_{(P_3,P_5)\in\cQ}\b1_{\cQ^B[w]}(P_3,P_5)d\bigl(\phi(P_3),\phi(P_5)\bigr)-\mathop{\E}_{u,h,\x,\y}b(u,h,\x,\y)d\bigl(\phi(P_3),\phi(P_5)\bigr)\Bigr|\leq 2p^{k-t}.\]
Therefore, by the triangle inequality we have
\[\Bigl|\mathop{\E}_{u,h,\x,\y}b(u,h,\x,\y)d\bigl(\phi(P_3),\phi(P_5)\bigr)-p^{-3k}\mathop{\E}_{(P_3,P_5)\in\cQ}d\bigl(\phi(P_3),\phi(P_5)\bigr)\Bigr|\leq p^{-2k}\theta+2p^{k-t}\]
as long as $w$ lies outside a set of  at most $(2p^{7k-t}+2p^{16k-t})/\theta^2$. (It is not obvious from the notation, but the dependence of the above statement on $w$ is via the function $b$, which we defined for a fixed $w$ but which is a different function for different values of $w$.)

The same argument applies to the term with $P_4$ and $P_6$, which gives the same answer because both $(P_3,P_5)$ and $(P_4,P_6)$ are uniformly distributed over all 4-arrangements. This proves the lemma.
\end{proof}

\subsection{Approximating $\psi$ by a single-valued function with good bilinearity properties}

In this subsection, we put together the results of the previous two subsections and find a function that is close to $\psi$, linear in each direction separately, and defined on almost all of $B$.

First, let us see what we now know about $\psi$, if we assume that $\phi$ is a $(1-\eta)$-bihomomorphism with respect to the characteristic function $b$ of $B$. Recall that this is equivalent to saying that
\[\E_{(P_1,P_2)\in\cQ}d\bigl(\phi(P_1),\phi(P_2)\bigr)\leq\eta\E_{(P_1,P_2)\in\cQ}b(P_1,P_2),\]
where as usual $b(P_1,P_2)$ is shorthand for the product of the values of $b$ over all the points in $P_1$ and $P_2$ (so it is 1 if the 4-arrangement $(P_1,P_2)$ lies in $B$ and 0 otherwise) and $\cQ$ is the set of all 4-arrangements in the whole of $(\F_p^n)^2$. Recall also that the left-hand side above is equal to $\E_{w',h'}d\bigl(\psi(w',h'),\psi(w',h')\bigr)$. 

Lemma \ref{averagew} gives us that
\[d\bigl(\psi(w_1,h)\psi(w_2,h),\psi(w_1+w_2,h)\bigr)\leq 2p^{-2k}\E_wd\bigl(\psi(w,h),\psi(w,h)\bigr)\pm 7p^{2k-t}\]
for all $w_1,w_2$ and for all $h$ outside a set of density at most $p^{k-t}$ such that $\be''(w_1,h)=\be''(w_2,h)=0$, those two equations being a trivial necessary condition for the left-hand side of the inequality not to be zero. Then Lemma \ref{averagewh} and our hypothesis tell us that for every $\theta>0$ we have the inequality
\[\E_wd\bigl(\psi(w,h),\psi(w,h)\bigr)\leq 2\eta\E_{(P_1,P_2)\in\cQ}b(P_1,P_2)+2\theta\]
for all $h$ outside a set of density at most $12p^{26k-t}/\theta^2$. 

Corollary \ref{noof4arrs2} tells us (provided that $t\geq 7k$) that $\E_{(P_1,P_2)\in\cQ}b(P_1,P_2)=p^{-7k}\pm 6p^{-t}$, so the above inequality implies that
\[\E_wd\bigl(\psi(w,h),\psi(w,h)\bigr)\leq 2\eta p^{-7k}+12p^{-t}+2\theta.\]
Plugging this bound into the first inequality gives us the inequality
\[d\bigl(\psi(w_1,h)\psi(w_2,h),\psi(w_1+w_2,h)\bigr)\leq 4\eta p^{-9k}+24p^{-2k-t}+2\theta p^{-2k}+7p^{2k-t}.\]
Therefore, if $\theta\leq\eta p^{-7k}$ and $t\geq 12k+\log_p(7\eta^{-1})$, we obtain the inequality
\[d\bigl(\psi(w_1,h)\psi(w_2,h),\psi(w_1+w_2,h)\bigr)\leq 8\eta p^{-9k}.\]
Recall that this holds for all $h$ outside a set of density at most $p^{k-t}+12p^{26k-t}/\theta^2$. 

The main point here is that we can choose some $\theta\leq\eta p^{-7k}$, and having made that choice we can choose $t$ such that the exceptional set of $h$ has some small density that can depend on $k$. For the time being we shall defer making an actual choice, but we note that the condition on $t$ will be that it is bounded below by some multiple of $k$ (by an absolute constant).

Now let us obtain a similar inequality that demonstrates linearity in the second variable for $\psi$. When $w$ is x-normal for $\be'$, Lemma \ref{averageh} gives us that
\[d\bigl(\psi(w,h_1)\psi(w,h_2),\psi(w,h_1+h_2)\bigr)\leq 2p^{-2k}\E_hd(\psi(w,h),\psi(w,h))+10p^{13k-t}+6\theta p^{-7k}\]
for all pairs $(h_1,h_2)$ outside a set of density at most $25p^{16k-t}/\theta^2$. Now for each $w$, if we expand out $\psi(w,h)$ and apply Lemma \ref{averagehw}, we obtain that
\begin{align*}
\E_hd(\psi(w,h),\psi(w,h))&=\E_{(P_1,P_2)\in\cQ(w)}d\bigl(\phi(P_1),\phi(P_2)\bigr)\\
&\leq 2\E_{(P_3,P_5)\in\cQ}\,d\bigl(\phi(P_3),\phi(P_5)\bigr)+2p^{-2k}\theta+4p^{k-t}\\
\end{align*}
provided that $w$ lies outside a set of density at most $(4p^{7k-t}+4p^{16k-t})/\theta^2$. Therefore, for the non-exceptional triples $(w,h_1,h_2)$, we have
\begin{align*}d\bigl(\psi(w,h_1)\psi(w,h_2),&\psi(w,h_1+h_2)\bigr)\\
&\leq 4p^{-2k}\E_{(P_3,P_5)\in\cQ}\,d\bigl(\phi(P_3),\phi(P_5)\bigr)+4p^{-4k}\theta+8p^{-k-t}+10p^{13k-t}+6\theta p^{-7k}\\
&\leq 4p^{-2k}(\eta(p^{-7k}+6p^{-t})+5p^{-4k}\theta+11p^{13k-t}\\
&\leq 4\eta p^{-9k}+5p^{-4k}\theta+12p^{13k-t}.\\
\end{align*}
If we choose $\theta$ to be at most $\eta p^{-5k}/5$, then we can choose $t$ in such a way that
\[d\bigl(\psi(w,h_1)\psi(w,h_2),\psi(w,h_1+h_2)\bigr)\leq 6\eta p^{-9k}\]
for all $w$ and for all $(h_1,h_2)$ outside exceptional sets of densities that can depend on $k$. 

Let us summarize this information in the form of a lemma.

\begin{lemma} \label{propertiesofpsi}
Let $\d>0$ and suppose that $t\geq 44k+\log(\eta^{-1})+\log(\d^{-1})$. Then 
\[d\bigl(\psi(w_1,h)\psi(w_2,h),\psi(w_1+w_2,h)\bigr)\leq 8\eta p^{-9k}\] 
for all triples $(w_1,w_2,h)$ outside a set of density at most $p^{-8k}/2500$, and 
\[d\bigl(\psi(w,h_1)\psi(w,h_2),\psi(w,h_1+h_2)\bigr)\leq 8\eta p^{-9k}\] 
for all triples $(w,h_1,h_2)$ outside a set of density at most $p^{-8k}/2500$.
\end{lemma}

\begin{proof}
Let us set $\theta=\eta p^{-7k}/5$. Then the first inequality holds outside a set of density at most $20p^{26k-t}/\theta^2=500p^{40k-t}/\eta^2$, and the second holds outside a set of density at most $25p^{16k-t}/\theta^2=625 p^{30k-t}/\eta^2$. These densities are both at most $\d$ when $t\geq 44k+\log(\eta^{-1})+\log(\d^{-1})$. (Here we used the fact that $p\geq 5$, so $p^4\geq 625$.)
\end{proof}

In order to use Lemma \ref{propertiesofpsi} to obtain a single-valued function, we prove a few more easy facts about our distance function.

\begin{lemma} \label{twoclosefunctions}
Let $f,g$ be non-negative functions in $\cA$ and suppose that $d(f,g)\leq\theta\|f\|_1\|g\|_1$. Then $d(f,f)\leq 2\theta\|f\|_1^2$ and $d(g,g)\leq 2\theta\|g\|_1^2$.
\end{lemma}

\begin{proof}
From the assumption and the bilinearity of $d$, we obtain the inequality
\[d(\|g\|_1f,\|f\|_1,g)\leq\theta\|f\|_1^2\|g\|_1^2.\]
It follows from the symmetry of $d$ and the triangle inequality that
\[d(\|g\|_1f,\|g\|_1f)\leq 2\theta\|f\|_1^2\|g\|_1^2,\]
and then by bilinearity again that
\[d(f,f)\leq 2\theta\|f\|_1^2.\]
By symmetry we have the inequality for $g$ as well.
\end{proof}

\begin{lemma} \label{almostdelta}
Let $f$ be a non-negative function in $\cA$, let $0\leq\theta<1/2$ and suppose that $d(f,f)\leq\theta\|f\|_1^2$. Then there exists a unique $x$ such that $f(x)\geq(1-\theta)\|f\|_1$.
\end{lemma}

\begin{proof}
The hypothesis implies that $\langle f,f\rangle\geq(1-\theta)\|f\|_1^2$. But we also know that $\langle f,f\rangle\geq\|f\|_1\|f\|_\infty$. It follows that $\|f\|_\infty\geq(1-\theta)\|f\|_1$, which proves the lemma (the uniqueness being obvious).
\end{proof}

\begin{corollary} \label{additive}
Let $0\leq\theta<1/25$ and let $f,g,h$ be three non-negative functions in $\cA$ and suppose that $d(fg,h)\leq\theta\|f\|_1\|g\|_1\|h\|_1$. Then there exist unique $x,y,z\in G$ such that $f(x)\geq(1-2\theta)\|f\|_1$, $g(y)\geq(1-2\theta)\|g\|_1$, $h(z)\geq(1-2\theta)\|h\|_1$, and $x+y=z$.
\end{corollary}

\begin{proof}
Since $\|fg\|_1=\|f\|_1\|g\|_1$ for any two non-negative functions $f,g\in\cA$, and since $\langle fg,h\rangle=\langle f,g^*h\rangle$, Lemma \ref{twoclosefunctions} implies that $d(f,f)\leq 2\theta\|f\|_1^2$, $d(g,g)\leq 2\theta\|g\|_1^2$ and $d(h,h)\leq 2\theta\|h\|_1^2$. The existence and uniqueness of $x,y,z$ satisfying the three inequalities now follows from Lemma \ref{almostdelta}. 

Write $\tilde f$ for the function that takes the value $\|f\|_1$ at $x$ and 0 elsewhere, and define $\tilde g$ and $\tilde h$ similarly. Then $\|f-\tilde f\|_1\leq 4\theta\|f\|_1$, with corresponding estimates for $g$ and $h$. From this it follows that
\[d(\tilde f\tilde g,\tilde h)=d(\tilde f\tilde g,\tilde h-h)+d(\tilde f\tilde g,h)\leq 4\theta\|f\|_1\|g\|_1\|h\|_1+d(\tilde f\tilde g,h),\]
and then that
\[d(\tilde f\tilde g,h)=d(\tilde g,\tilde f^*h)=d(\tilde g-g,\tilde f^*h)+d(g,\tilde f^*h)\leq 4\theta\|f\|_1\|g\|_1\|h\|_1+d(g,\tilde f^*h),\]
and thirdly that
\[d(g,\tilde f^*h)=d(\tilde f,g^*h)=d(\tilde f-f,g^*h)+d(\tilde f,g^*h)\leq 4\theta\|f\|_1\|g\|_1\|h\|_1+d(f,g^*h).\]
It follows that $d(\tilde f\tilde g,\tilde h)\leq 25\theta\|\tilde f\|_1\|\tilde g\|_1\|\tilde h\|_1$, and therefore that $\langle\tilde f\tilde g,\tilde h\rangle>0$, which can happen only if $x+y=z$, since $\tilde f,\tilde g$ and $\tilde h$ are multiples of delta functions.
\end{proof}

\begin{lemma} \label{noofvps}
For all $(w,h)\in B''$ outside a set of density at most $p^{k-t}$, the probability that a random vertical parallelogram of width $w$ and height $h$ lives in $B$ is $p^{-3k}\pm 2p^{k-t}$.
\end{lemma}

\begin{proof}
Every $h$ outside a set of density at most $p^{k-t}$ is y-normal for $\be'$. For each such $h$, the density of $x$ such that $\be'(x,h)=0$ is $p^{-k}$. We are assuming that $\be''(w,h)=0$. Therefore, if $\be'(x,h)=0$, then $\be'(x+w,h)=0$ as well. The density of $x$ such that $x$ and $x+w$ are x-normal for $\be$ is at least $1-2p^{k-t}$, and if they are both x-normal for $\be$, then the density of pairs $(y,y')$ such that $\be(x,y)=\be(x+w,y')=0$ is $p^{-2k}$. It follows that the probability that $(x,y),(x,y+h),(x+w,y')$ and $(x+w,y'+h)$ all lie in $B$ is $p^{-3k}\pm 2p^{k-t}$, as claimed.
\end{proof}

Using these results, we can now approximate $\psi$ by a single-valued function that is very close to being additive in each variable separately.

\begin{lemma} \label{propertiesoftpsi}
If $\eta<1/400$ and $t\geq 44k+\log(\eta^{-1})+\log(\d^{-1})$, then there is a function $\tpsi:B\to G$ with the following properties.
\begin{enumerate}
\item $\tpsi(w_1,h)+\tpsi(w_2,h)=\tpsi(w_1+w_2,h)$ for all triples $(w_1,w_2,h)$ outside a set of density at most $\d$.
\item $\tpsi(w,h_1)+\tpsi(w,h_2)=\tpsi(w,h_1+h_2)$ for all triples $(w,h_1,h_2)$ outside a set of density at most $\d$.
\item $d\bigl(\psi(w,h),\delta_{\tpsi(w,h)}\bigr)\leq 64\eta p^{-3k}$ for all $(w,h)$ outside a set of density at most $\d$.
\end{enumerate}
\end{lemma}

\begin{proof}
By Lemma \ref{propertiesofpsi}, $d\bigl(\psi(w_1,h)\psi(w_2,h),\psi(w_1+w_2,h)\bigr)\leq 8\eta p^{-9k}$
for all triples $(w_1,w_2,h)$ outside a set of density at most $\d$. For each such triple, apply Corollary \ref{additive} to the elements $\psi(w_1,h),\psi(w_2,h)$ and $\psi(w_1+w_2,h)$ of $\cA$ and let $\tpsi(w_1,h),\tpsi(w_2,h)$ and $\tpsi(w_1+w_2,h)$ be the elements of $G$ that it gives. 

By Lemma \ref{noofvps}, $\|\psi(w,h)\|_1=p^{-3k}\pm 2p^{k-t}$ for all $(w,h)$ outside a set of density at most $p^{k-t}\leq\d$, so except in a set of triples $(w_1,w_2,h)$ of density at most $\d$ we can also say that $\|\psi(w_1,h)\|_1,\|\psi(w_2,h)\|_1$ and $\|\psi(w_1+w_2,h)\|_1$ are all at least $9p^{-3k}/10$, which implies that
\[d\bigl(\psi(w_1,h)\psi(w_2,h),\psi(w_1+w_2,h)\bigr)\leq 16\eta\|\psi(w_1,h)\|_1\|\psi(w_2,h)\|_1\|\psi(w_1+w_2,h)\|_1\]
outside a set of triples $(w_1,w_2,h)$ of density at most $\d$. By Lemma \ref{additive}, for each such triple we have that $\tpsi(w_1,h)+\tpsi(w_2,h)=\tpsi(w_1+w_2,h)$, and the first statement is proved. 

The second statement follows by interchanging the roles of the two coordinates and using the symmetry of Lemma \ref{propertiesofpsi}.

As for the third, each $(w,h)\in B''$ outside a set of density at most $\d$ is $(w_1+w_2,h)$ for some non-exceptional triple $(w_1,w_2,h)$, and therefore by Corollary \ref{additive} satisfies the inequality
\[\psi(w,h)(\tpsi(w,h))\geq(1-32\eta)\|\psi(w,h)\|_1,\]
which implies the statement.
\end{proof}

\subsection{Approximating $\tpsi$ by a function defined on all of $B$}

We have just obtained a function $\tpsi:B''\to G$ that is ``almost always additive" in each variable separately. We now wish to approximate $\tpsi$ by a function defined on $B''$ that is additive in each variable separately. 

We begin by restricting $\tpsi$ to a function that is additive in each variable separately and still defined on almost all of $B''$. We shall then extend that to a function that is defined on all of $B''$.

By the first statement of Lemma \ref{propertiesoftpsi} it follows that for all $h$ outside a set of density at most $\d^{1/2}$ we have that $\tpsi(w_1,h)+\tpsi(w_2,h)=\tpsi(w_1+w_2,h)$ for all pairs $(w_1,w_2)$ outside a set of density at most $\d^{1/2}$. In order to exploit this, we shall use the following variants of Lemma \ref{linearstability}.

\begin{lemma} \label{additivestability}
Let $0\leq \g<1/8$, let $G$ be a finite Abelian group, and let $A\subset G$ and $\phi:G\to G$ be such that for a fraction at least $1-\g$ of pairs $(x,y)\in G^2$ we have that $x,y,x+y\in A$ and $\phi(x+y)=\phi(x)+\phi(y)$. Then there is a subset $A_1\subset A$ of density at least $1-4\g$ in $G$ such that $\phi(x+y)=\phi(x)+\phi(y)$ whenever $x,y,x+y\in A_1$. 
\end{lemma}

\begin{proof}
For all $z$ outside a set $B$ of density at most $1-4\g$ it is the case that for all $x$ outside a set of density at most 1/4 the equation $\phi(z)=\phi(x)+\phi(z-x)$ holds (and in particular both sides are defined).

If $z_1,z_2\in B$, then choose $x_1,x_2$ at random. Then each of the equations $\phi(z_1)=\phi(x_1)+\phi(z_1-x_1)$, $\phi(z_2)=\phi(x_2)+\phi(z_2-x_2)$ and $\phi(z_1+z_2)=\phi(x_1+x_2)+\phi(z_1+z_2-x_1-x_2)$ holds with probability at least 3/4, so with probability at least 1/4 all three equations hold. 

We also have each of the equations $\phi(x_1+x_2)=\phi(x_1)+\phi(x_2)$ and $\phi(z_1+z_2-x_1-x_2)=\phi(z_1-x_1)+\phi(z_2-x_2)$ with probability at least $1-\g$, so with non-zero probability all five equations hold, which implies that $\phi(z_1+z_2)=\phi(z_1)+\phi(z_2)$.
\end{proof}

\begin{lemma} \label{additivestability2}
Let $0\leq\theta<1/6$, let $G$ be a finite Abelian group, let $A_1$ be a subset of $G$ of density at least $1-\theta$, and let $\phi:A_1\to G$ be additive. Then $\phi$ has a unique extension to an additive function $\phi_1$ defined on all of $G$.
\end{lemma}

\begin{proof}
For each $z\in G$ there exist $x,y\in A_1$ such that $x+y=z$. Therefore if an extension exists, the value of $\phi_1(z)$ has to be $\phi(x)+\phi(y)$.

We now prove that this gives a well-defined function. To do this, suppose that $x+y=x'+y'$. Since $\theta<1/6$, we can find $d$ such that all of $x+d,y+d,x'+d,y'+d$ and $x+y+d=x'+y'+d$ and $x+y+2d=x'+y'+2d$ belong to $A_1$. From this and our hypothesis we obtain the three equations
\[\phi(x)+\phi(y+d)=\phi(x')+\phi(y'+d),\]
\[\phi(x+d)+\phi(y)=\phi(x'+d)+\phi(y')\]
and
\[\phi(x+d)+\phi(y+d)=\phi(x'+d)+\phi(y'+d).\]
Subtracting the last from the sum of the first two gives us that $\phi(x)+\phi(y)=\phi(x')+\phi(y')$, which establishes that the extension is well-defined.

A similar argument shows that it is additive. Indeed, let $z_1,z_2\in G$, pick $x_1,x_2$ at random, and let $y_1=z_1-x_1$ and $y_2=z_2-x_2$. Then with probability at least $1-6\theta>0$ all of $x_1,x_2,y_1,y_2,x_1+x_2$ and $y_1+y_2$ belong to $A$, and if they do, then
\begin{align*}
\phi_1(z_1+z_2)&=\phi(x_1+x_2)+\phi(y_1+y_2)\\
&=\phi(x_1)+\phi(x_2)+\phi(y_1)+\phi(y_2)\\
&=\phi_1(x_1+y_1)+\phi_1(x_2+y_2)\\
&=\phi_1(z_1)+\phi_1(z_2).\\
\end{align*}
This completes the proof.
\end{proof}

It follows that for all $h$ outside a set of density at most $\d^{1/2}$ we can remove at most $4\d^{1/2}|G|$ points $(w,h)$ from $B''_{\bullet h}$ in such a way that $\tpsi(w_1,h)+\tpsi(w_2,h)=\tpsi(w_1+w_2,h)$ whenever $(w_1,h),(w_2,h)$ and $(w_1+w_2,h)$ belong to $B''$ and have not been removed. And from that it follows that $B''$ has a subset $B_1$ such that the restriction of $\tpsi$ to $B_1$ is additive in the first variable, and such that $B''\setminus B_1$ has density at most $5\d^{1/2}$ in $G$.

Similarly, $B''$ has a subset $B_2$ with the same properties but with the roles of the coordinates exchanged. Now let $B_3=B_1\cap B_2$. Then $B''\setminus B_3$ has density at most $10\d^{1/2}$, and the restriction of $\tpsi$ to $B_3$ is additive in each variable separately.

The density of $(B''\setminus B_3)_{\bullet h}$ is at most $4\d^{1/4}$ for all $h$ outside a set of density at most $4\d^{1/4}$, and the density of $(B''\setminus B_3)_{w\bullet}$ is at most $4\d^{1/4}$ for all $w$ outside a set of density at most $4\d^{1/4}$. Therefore, we can find a subset $B_4\subset B_3$ with the following properties.
\begin{enumerate}
\item For every $h$, either $(B_4)_{\bullet h}=\emptyset$ or $(B''\setminus B_4)_{\bullet h}$ has density at most $8\d^{1/4}$, and the density of $h$ for which $(B_4)_{\bullet h}=\emptyset$ is at most $4\d^{1/4}$.
\item For every $w$, either $(B_4)_{w\bullet}=\emptyset$ or $(B''\setminus B_4)_{w\bullet}$ has density at most $8\d^{1/4}$, and the density of $w$ for which $(B_4)_{w\bullet}=\emptyset$ is at most $4\d^{1/4}$.
\item The density of $B''\setminus B_4$ is at most $8\d^{1/4}$.
\end{enumerate}
The set $B_4$ also inherits from $B_3$ the property that the restriction of $\tpsi$ is additive in each variable separately.

\begin{lemma}\label{extendtob}
Assume that $8\d^{1/4}<p^{-2k}/6$. Then $\tpsi$ can be extended to a function $\tpsi_2$ that is defined on all of $B''$ and is still additive in each variable separately.
\end{lemma}

\begin{proof}
Let $H$ be the set of all $h$ such that $(B_4)_{\bullet h}$ is non-empty, and let $\zeta=8\d^{1/4}$. Since each set $B''_{\bullet h}$ is a subspace of $\F_p^n$ of codimension at most $k$, and therefore density at least $p^{-k}$, our bound for $\zeta$, together with property (1) above, implies that the relative density of $(B_4)_{\bullet h}$ in $B''_{\bullet h}$ is greater than $5/6$, and therefore Lemma \ref{additivestability2} implies that $\tpsi_{\bullet h}$ can be extended to an additive function defined on all of $B''_{\bullet h}$. Putting all these functions together gives a function $\tpsi_1$ that is additive in the first variable and is defined on all of the set $\{(w,h)\in B'':h\in H\}$. 

We now check that $\tpsi_1$ is still additive in the second variable. This is where it is crucial that we were able to ensure that our error parameter $\zeta$ was small compared with $p^{-k}$. Let $w\in G$ and let $h_1,h_2\in H$ be such that $\be''(w,h_1)=\be''(w,h_2)=0$, which implies that $\be''(w,h_1+h_2)=0$. The set $B''_{\bullet h_1}\cap B''_{\bullet h_2}$ is a subspace of $\F_p^n$ of codimension at most $2k$ and therefore density at least $p^{-2k}$, and it is contained in the subspace $B''_{\bullet(h_1+h_2)}$. By our bound on $\zeta$ and the fact that $(B''\setminus B_4)_{\bullet h}$ has density at most $\zeta$ for each $h\in H$, if we choose a random $w_1\in B''_{\bullet h_1}\cap B''_{\bullet h_2}$ and set $w_2=w-w_1$, then with non-zero probability the points $(w_1,h_1),(w_2,h_1),(w_1,h_2),(w_2,h_2),(w_1,h_1+h_2)$ and $(w_2,h_1+h_2)$ all belong to $B_4$. That gives us the equations
\[\tpsi(w_1,h_1)+\tpsi(w_2,h_1)=\tpsi_1(w,h_1),\] 
\[\tpsi(w_1,h_2)+\tpsi(w_2,h_2)=\tpsi_1(w,h_2),\]
\[\tpsi(w_1,h_1+h_2)+\tpsi(w_2,h_1+h_2)=\tpsi_1(w,h_1+h_2),\]
\[\tpsi(w_1,h_1)+\tpsi(w_1,h_2)=\tpsi(w_1,h_1+h_2),\]
and
\[\tpsi(w_2,h_1)+\tpsi(w_2,h_2)=\tpsi(w_2,h_1+h_2).\]
Together these imply that $\tpsi_1(w,h_1)+\tpsi_1(w,h_2)=\tpsi_1(w,h_1+h_2)$, as required.

Let $B_5=\{(w,h)\in B'':h\in H\}$. Then $B_5$ has all the properties enjoyed by $B_4$, so we can run the same argument with the roles of the two variables interchanged. However, since $H$ has density at least $1-\zeta$, if we define $W$ to be the set of all $w$ such that $(B''\setminus B_5)_{w\bullet}$ has density at most $\zeta$, we find that $W$ is the whole of $\F_p^n$. Therefore, at the end of the process we obtain an extension $\tpsi_2$ of $\tpsi_1$ to the whole of $B''$ that is additive in each variable separately.
\end{proof}

We make one further observation, which we shall not need, but which is of some mild interest. Given a bilinear Bohr set $B''$ defined by a genuinely bilinear map $\be''$, as opposed to merely a bi-affine map, and given a function $\tpsi_2:B''\to G$ that is linear in each variable separately, it also respects all 4-arrangements in $B''$. Indeed, we have that 
\[\tpsi_2(w,h)-\tpsi_2(w,h+k)-\tpsi_2(w+u,h')+\tpsi_2(w+u,h'+k)=\tpsi_2(w+u,k)-\tpsi_2(w,k)=\tpsi_2(u,k),\]
so if $P$ is a vertical parallelogram in $B''$, then $\tpsi_2(P)$ depends only on the width and height of $P$.

\subsection{Putting everything together}

We now collect together the results of this section into a single theorem. 

\begin{theorem}\label{stability}
Let $0\leq\eta<1/400$ and let $k$ and $t$ be such that $t\geq 60k+9\log(\eta^{-1})$. Let $B$ be a bilinear Bohr set of codimension $k$ and rank $t$, let $\cA$ be the group algebra of $G$, and let $\phi:G^2\to\cA$ be a function such that $\phi(x,y)\in\cA$ if $(x,y)\in B$ and $\phi(x,y)=0$ otherwise. Let $\psi=\mc\phi$ be the mixed convolution of $\phi$. Suppose that $\phi$ is a $(1-\eta)$-bihomomorphism. Then there exists a function $\tpsi_2:B''\to G$ that is additive in each variable separately such that, setting $\psi_2(w,h)=\d_{\tpsi_2(w,h)}$ when $(w,h)\in B''$ and 0 otherwise, we have the inequality
\[\E_{w,h}d\bigl(\psi(w,h),\psi_2(w,h)\bigr)\leq 128\eta p^{-4k}.\]
\end{theorem}

\begin{proof}
Note first that by Lemma \ref{noofvps} we have that $\|\psi(w,h)\|_1\approx p^{-3k}$ for almost every $(w,h)\in B''$. We also know that $B''$ itself has density approximately $p^{-k}$ and that $\|\psi_2\|_1=1$ for $(w,h)\in B''$. Thus, $p^{-4k}$ is the trivial upper bound for the problem, and the theorem says that we can beat it by a constant factor that depends on $\eta$, as long as the rank of $B$ is sufficiently large.

Now let us set $\d=16\eta^4p^{-16k}$ and note that our choice of $t$ gives us the required inequality $t\geq 44k+\log(\eta^{-1})+\log(\d^{-1})$.

To prove the result, we first use Lemma \ref{propertiesoftpsi} to obtain a function $\tpsi$ with the properties stated there. Then the discussion just before Lemma \ref{extendtob} gives us a set $B_4$ such that $B''\setminus B_4$ has density at most $8\d^{1/4}$ and the restriction of $\tpsi$ to $B_4$ is additive in each variable separately. Lemma \ref{extendtob} tells us that this restriction of $\tpsi$ can be extended to a function $\tpsi_2$ that is defined on all of $B''$ and is additive in each variable separately. 

The functions $\tpsi$ and $\tpsi_2$ agree outside a set of density at most $16\d^{1/4}$, and statement (3) of Lemma \ref{propertiesoftpsi} tells us that $d\bigl(\psi(w,h),\d_{\tpsi(w,h)}\bigr)\leq 64\eta p^{-3k}$ for all $(w,h)$ outside a set of density at most $\d$. It follows that $d\bigl(\psi(w,h),\d_{\tpsi_2(w,h)}\bigr)\leq 64\eta p^{-3k}$ for all $(w,h)$ outside a set of density at most $17\d^{1/4}$. Since $B''$ has density at most $3p^{-k}/2$ and $17\d^{1/4}\leq 32\eta p^{-4k}$, the result follows.
\end{proof}

\section{Extending a bihomomorphism from a bilinear Bohr set to the whole of $G^2$}

Let $G=\F_p^n$ and let $H$ be a finite Abelian group (usually equal to $G$). In this section we shall prove that if $B$ is a high-rank bilinear Bohr set and $\phi:B\to H$ is affine in each variable separately, then $\phi$ can be extended to $G^2$ while retaining that property. Note that $\phi$ here is not the same as the function $\phi$ from earlier: it corresponds more closely to the function $\tpsi_2$, but this section is intended to be a free-standing result. We actually only need the slightly easier case where $B$ is defined by a bilinear (as opposed to bi-affine) function and $\phi$ is additive in each variable separately, but the more general statement is of some interest, so we give it here.

As a preparatory lemma, we shall show that $\phi$ is not only affine in each variable separately, but it also respects all 4-arrangements, which is equivalent to saying that the vertical derivative $\phi'$ is a Freiman homomorphism in each variable separately. We have already noted that this is very easy to show when $\phi$ is linear in each variable separately, but it becomes slightly trickier in the affine case, since then we do not know that every pair of cross-sections of $B$ has a non-empty intersection.

\begin{lemma} \label{weaktostrong}
Let $G=\F_p^n$ and let $H$ be a finite Abelian group. Let $B$ be a bilinear Bohr set of codimension $k$ and rank $t$ with $t>6k$. Let $\phi:B\to H$ be a Freiman homomorphism in each variable separately, and let $\phi'$ be the vertical derivative of $\phi$. Then $\phi'$ is a Freiman homomorphism in each variable separately.
\end{lemma}

\begin{proof}
Recall that $B'$ is the set of all pairs $(x,h)$ such that there exists $y$ with $(x,y)$ and $(x,y+h)$ in $B$. Also, the vertical derivative $\phi'(x,h)$ is defined to be $\phi(x,y+h)-\phi(x,y)$ whenever $(x,h)\in B'$.

It is trivial that $\phi'$ is a Freiman homomorphism on each column, since each $\phi_{x\bullet}$ is affine on $B_{x\bullet}$ and therefore $\phi'_{x\bullet}$ is even linear on $B'_{x\bullet}$. 

Now let us pick an arbitrary $(x,h)\in B'$. As we noted at the end of the previous section, since $B'_{\bullet h}$ is non-empty, it has density at least $p^{-k}$ in $G$. Also, $B_{x\bullet}$ has density at least $p^{-k}$ in $G$. 

By Lemma \ref{inductivehighrank}, if we choose $x_1$ uniformly at random from $B'_{\bullet h}$, the probability that the restriction of $\be_{x_1\bullet}$ to $B_{x\bullet}$ is a surjection is at least $1-p^{3k-t}$, and if it is a surjection, then $B_{x_1\bullet}\cap B_{x\bullet}$ has density $p^{-k}$ in $B_{x\bullet}$ and therefore $p^{-2k}$ in $G$. If we then pick $x_2$ uniformly at random from $B'_{\bullet h}$, the same argument gives us that $B_{x_1\bullet}\cap B_{x_2\bullet}\cap B_{x\bullet}$ has density $p^{-3k}$ with probability at least $1-p^{4k-t}$. And if we now choose $x_3$ such that $x_1-x_2=x_3-x$, we have that $B_{x_1\bullet}\cap B_{x_2\bullet}\cap B_{x\bullet}\subset B_{x_3\bullet}$, so $B_{x_1\bullet}\cap B_{x_2\bullet}\cap B_{x_3\bullet}\cap B_{x\bullet}$ has density $p^{-3k}$. 

In particular, it is non-empty, so we can pick $y$ that belongs to it. Then $y+h$ belongs to it as well, since the $x_i$ and $x$ all belong to $B'_{\bullet h}$. But
\[\phi(x_1,y)-\phi(x_2,y)=\phi(x_3,y)-\phi(x,y)\]
and
\[\phi(x_1,y+h)-\phi(x_2,y+h)=\phi(x_3,y+h)-\phi(x,y+h),\]
which implies that 
\[\phi'(x_1,h)-\phi'(x_2,h)=\phi'(x_3,h)-\phi'(x,h).\]

We have proved that for each $(x,h)\in B'$, out of all triples $(x_1,x_2,x_3)\in B'_{\bullet h}$ such that $x_1-x_2=x_3-x$, the proportion such that
\[\phi'(x_1,h)-\phi'(x_2,h)=\phi'(x_3,h)-\phi'(x,h)\]
is at least $1-p^{4k-t}$. Note that this is a stronger statement than saying that the proportion of additive quadruples in $B_{\bullet h}$ that are respected by $\phi'_{\bullet h}$ is at least $1-p^{4k-t}$, since the above statement applies to every $x$, and not just almost every $x$. 

Thus, for every $x\in B'_{\bullet h}$, $\phi'(x,h)$ is the most popular value of $\phi'(x_2,h)+\phi'(x_3,h)-\phi'(x_1,h)$. Since we also know, by Lemma \ref{linearstability}, that $\phi'_{\bullet h}$ agrees with an affine map on a subset of $B'_{\bullet h}$ of density at least $1-5p^{4k-t}$, it follows that $\phi'$ is equal to that affine map.

This proves that $\phi'_{\bullet h}$ is an affine map on $B'_{\bullet h}$ whenever $B'_{\bullet h}$ is non-empty, which completes the proof of the lemma.
\end{proof}

We need a second preparatory lemma as well. 

\begin{lemma} \label{connected}
Let $G=\F_p^n$ and let $B\subset G^2$ be a bilinear Bohr set of codimension $k$ and rank $t$ with $t>3k$. Then $B$, when considered as a bipartite graph, consists of one connected component together with isolated vertices. Moreover, the density of the set of isolated vertices in each vertex set is at most $p^{k-t}$.
\end{lemma}

\begin{proof}
By Lemma \ref{usinghighrank} the density of columns $B_{x\bullet}$ that fail to have density $p^{-k}$ is at most $p^{k-t}$, and similarly for rows. In bipartite-graph terms, that says that except in a subset of density at most $p^{k-t}$ on both sides, all vertices have neighbourhoods of density $p^{-k}$, and in particular are not isolated.

Let $x$ and $y$ be such that $B_{x\bullet}$ and $B_{\bullet y}$ are non-empty, and hence have density at least $p^{-k}$. Then by Lemma \ref{inductivehighrank}, if $u$ is chosen randomly from $B_{\bullet y}$, the probability that the restriction of $\be_{u\bullet}$ to $B_{x\bullet}$ is a surjection is at least $1-p^{3k-t}>0$, and if it is a surjection then there exists $v\in B_{x\bullet}$ such that $\be(u,v)=z$. This gives us an edge linking $B_{\bullet y}$ to $B_{x\bullet}$, and therefore a path of length 3 from $x$ to $y$.

Thus, every non-isolated vertex one one side is connected to every non-isolated vertex on the other side, from which the result follows easily.
\end{proof}

Now let us move on to the main theorem. Our approach is as follows. Let $B=\{(x,y)\in G^2:\be(x,y)=z\}$, where $\be:G^2\to\F_p^k$ is a bi-affine map of codimension $k$ and rank $t$. Then by Lemma \ref{inductivehighrank}, if we choose $x$ randomly, the probability that the linear map $\be_{x\bullet}$ is of full rank is at least $1-p^{k-t}$. Let $x_1$ be an element of $G$ with this property.

We can extend $\phi$ from $B$ to $B\cup G_{x_1\bullet}$ by extending $\phi_{x_1\bullet}$ arbitrarily from $B_{x_1\bullet}$ to an affine map on the whole of $G_{x_1\bullet}$. If we do that, the resulting extension is still affine in each variable separately. This is obvious for the columns. As for the rows, it follows since an affine map on an affine subspace can be extended arbitrarily to a function defined on the union of that subspace and an extra point, and it will remain affine.

Our aim now is to prove that there is a unique further extension to a bi-affine function defined on all of $G^2$. To do that we shall being by saying what the function is almost everywhere. Then we shall prove that it is well-defined and a Freiman homomorphism in each variable separately. Once that is done, Lemma \ref{extendtob} (in fact, an easy special case of that lemma) will remove the ``almost" and give us a bi-affine function defined on the whole of $G^2$.

We know that the set of $x$ such that the map $y\mapsto(\be(x_1,y),\be(x,y))$ has full rank has density at least $1-\d$. Let this set be $U$. For any $x\in U$ and any $y\in G$, we define $\psi(x,y)$ as follows. Let $\be(x,y)=u$. We begin by finding $h$ such that $\be(x,y-h)=z$ and $\be'(x_1,h)=u-z$, where $\be'$ is, as usual, the vertical derivative of $\be$. This we can do because the map $h\mapsto(\be'(x_1,h),\be(x,y-h))$ also has full rank. Next, we find $w$ such that $\be_{(x_1+w)\bullet}$ and $\be_{(x-w)\bullet}$ have full rank, and also such that $\be'(x_1+w,h)=0$. Since $\be'(x_1,h)=\be'(x,h)=u-z$ and $\be'$ is bi-affine, it follows that $\be'(x-w,h)=0$. 

We now pick $y_1$ arbitrarily, and $y_2,y_3$ such that $\be(x_1+w,y_2)=\be(x-w,y_3)=z$. This gives us a 4-arrangement whose points, in the usual order, are $(x_1,y_1)$, $(x_1,y_1+h)$, $(x_1+w,y_2)$, $(x_1+w,y_2+h)$, $(x-w,y_3)$, $(x-w,y_3+h)$, $(x,y-h)$, and $(x,y)$. 

We know the values of $\phi$ at all these points except the last. That is because the first two points have first coordinate $x_1$, and we have extended $\phi$ so that it is defined on all such points. As for the remaining points apart from $(x,y)$, they have been chosen to belong to $B$. For instance, $(x_2,y_3+h)$ belongs to $B$ because $\be(x-w,y_3+h)=\be(x-w,y_3)+\be'(x-w,h)=z+0=z$. So we now choose $\psi$ in the only way possible so that it will respect this 4-arrangement. That is,
\begin{align*}\psi(x,y)=\phi(x_1,y_1)&-\phi(x_1,y_1+h)-\phi(x_1+w,y_2)\\
&+\phi(x_1+w,y_2+h)-\phi(x-w,y_3)+\phi(x-w,y_3+h)+\phi(x,y-h).\\
\end{align*}

We now prove that $\psi$ is well-defined on $U\times G$.

\begin{lemma}
The value of the right-hand side in the formula just given for $\psi$ does not depend on our choices of $w,h,y_1,y_2$ and $y_3$.
\end{lemma}

\begin{proof}
Since $\phi$ is affine in the columns, it has a well-defined vertical derivative $\phi'$, which allows us to rewrite the formula as
\[\psi(x,y)=-\phi'(x_1,h)+\phi'(x_1+w,h)+\phi'(x-w,h)+\phi(x,y-h),\]
which is independent of $y_1,y_2$ and $y_3$. 

Note also that if we replace $h$ by $h'$, then the difference to the right-hand side is
\[-\phi'(x_1,h'-h)+\phi'(x_1+w,h'-h)+\phi'(x-w,h'-h)-\phi'(x,h'-h),\]
which is well-defined, since each column of $B'$ is a linear subspace and not just an affine subspace.

By Lemma \ref{weaktostrong}, $\phi'$ is affine in the rows, so this difference is zero.

Now suppose that we replace $w$ by $w'$. Then the difference to the right-hand side is
\[\phi'(x_1+w',h)-\phi'(x_1+w,h)+\phi'(x-w',h)-\phi'(x-w,h).\]
Since $\phi'$ is affine in the rows, and $(x_1+w')-(x_1+w)+(x-w')-(x-w)=0$, this difference is also zero.

That does not complete the proof, because a change to both $h$ and $w$ is not necessarily decomposable into a change to $h$ followed by a change to $w$. We deal with this by using Lemma \ref{connected}, which will provide us with a path from one choice of $(w,h)$ to another that changes only one variable at a time.

Note that $(w,h)$ is a possible pair if $\be'(x_1,h)=u-z$, $\be(x,y-h)=z$, and $\be'(x_1+w,h)=\be'(x-w,h)=0$. Because the map $y\mapsto(\be(x_1,y),\be(x,y))$ has full rank, the first two conditions tell us that $h$ lies in an affine subspace $V\subset G$ of codimension $2k$. Also, if they hold, then, as we have already commented, the condition $\be'(x_1+w,h)$ implies the condition $\be'(x-w,h)=0$. 

Therefore, the set of all possible $(w,h)$ is the intersection of $G\times V$ with the bilinear Bohr set $\{(w,h):\be'(x_1+w,h)=0\}$, which has rank $t$ in $G^2$ and therefore rank at least $t-k$ in $G\times V$. Therefore, by Lemma \ref{connected}, if $(w_1,h_1)$ and $(w_2,h_2)$ are possible values for $(w,h)$ we can find $w'$ and $h'$ such that $(w_1,h')$, $(w',h')$ and $(w',h_2)$ are possible values. This, together with the fact that changing just one variable does not alter the value of the formula for $\psi(x,y)$, completes the proof that $\psi$ is well-defined on $U\times G$.
\end{proof}

The next step is to prove that $\psi$ is a Freiman homomorphism in each row and column.

\begin{lemma}
Let $\psi$ be defined as above and assume that $t>6k$. Then $\psi$ is affine in each column.
\end{lemma}

\begin{proof}
Let $x\in U$ and let $y_1-y_2=y_3-y_4$. For $i=1,2,3$ let $h_i$ be such that $\be(x,y_i-h_i)=z$ and $\be'(x_1,h_i)=\be(x,y_i)-z$. Now let $h_4$ be such that $h_1-h_2=h_3-h_4$. Since $\be$ and $\be'$ are bi-affine, we have that $\be(x,y_4-h_4)=z$ and $\be'(x_1,h_4)=\be(x,y_4)-z$. 

Note that there are several ways of choosing $h_1,h_2,h_3$ above, and that for each $i$, the constraint on $h_i$ is that it should lie in a certain affine subspace of codimension $k$. From Lemma \ref{usinghighrank}, if we choose $h_1,h_2,h_3$ randomly from these affine subspaces, then the probability that the equations $\be'(x_1+w,h_1)=\be'(x_1+w,h_2)=\be'(x_1+w,h_3)=0$ are satisfied by all $w$ in a set of density $p^{-3k}$ is at least $1-p^{6k-t}$. (The lemma gives a bound of $1-p^{3k-t}$, and we then condition on an event of probability $p^{-3k}$.) 

If those equations are satisfied, then so is the equation $\be'(x_1+w,h_4)=0$. Furthermore, as we have already noted, the equations $\be'(x_1,h_i)=\be(x,y_i)-z=\be'(x,h_i)$ and $\be'(x_1+w,h_i)=0$ imply that $\be'(x-w,h_i)=0$. It follows that 
\[\psi(x,y_i)=\phi'(x_1+w,h_i)+\phi'(x-w,h_i)-\phi'(x_1,h_i)+\phi(x,y_i-h_i)\]
for each $i$, and now the equation $\psi(x,y_1)-\psi(x,y_2)=\psi(x,y_3)-\psi(x,y_4)$ follows because we have the corresponding linear relation for each term on the right-hand side.
\end{proof}

\begin{lemma}
Let $\psi$ be as defined above. Then for every $y$ and every triple $(x,x',d)$ outside a set of density at most $2p^{3k-t}$ such that $\be(u,y)=\be(u+d,y)$ for some (and hence every) $u$, we have that
\[\psi(x+d,y)-\psi(x,y)=\psi(x'+d,y)-\psi(x',y).\]
\end{lemma}

\begin{proof}
For all pairs $(x,d)$ outside a set of density at most $p^{3k-t}$ we can find $h$ such that $\be(x,y-h)=\be(x+d,y-h)=z$ and $\be'(x_1,h)=\be(x,y)-z$. By our condition on $d$, this implies that $\be'(x_1,h)=\be(x+d,y)-z$ as well. Furthermore, there are many possibilities for $h$ and for all of them outside a set of density at most $p^{k-t}$ we can find $w$ such that $\be'(x_1+w,h)=0$, which implies that $\be'(x-w,h)=\be'(x+d-w,h)=0$.

By definition, we have that
\[\psi(x,y)=\phi'(x_1+w,h)+\phi'(x-w,h)-\phi'(x_1,h)+\phi(x,y-h)\]
and 
\[\psi(x+d,y)=\phi'(x_1+w,h)+\phi'(x+d-w,h)-\phi'(x_1,h)+\phi(x+d,y-h).\]
Let us write $'\!\phi(w,y)$ for $\phi(x+w,y)-\phi(x,y)$ where this is defined. This is the \emph{horizontal derivative} of $\phi$. Then the above formulae give us that
\[\psi(x+d,y)-\psi(x,y)=\pphi'(d,h)+\pphi(d,y-h)=\pphi(d,y).\]
Since this does not depend on $x$, we have the result claimed.
\end{proof}


\begin{lemma}
Let $\psi$ be any function that is affine in columns and satisfies the conclusion of the previous lemma. Then for every $y$ and every quadruple $x_1-x_2=x_3-x_4$ outside a set of density at most $p^{3k-t}$ we have
\[\psi(x_1,y)-\psi(x_2,y)=\psi(x_3,y)-\psi(x_4,y).\]
\end{lemma}

\begin{proof}
Let $\be(x_i,y)=z_i$. Then $z_1-z_2=z_3-z_4$, since $\be$ is bi-affine. 

Now choose $h_1$ such that $\be(x_4,y-h_1)=\be(x_3,y-h_1)=z_3$ and $\be(x_2,y-h_1)=z_1$. This is possible for all $(x_2,x_3,x_4)$ outside a set of density at most $p^{3k-t}$ and it implies that $\be(x_1,y-h_1)=z_1$.

Next, choose $h_2$ such that $\be(x_4,y-h_2)=\be(x_2,y-h_2)=z_2$ and $\be(x_3,y-h_2)=\be(x_1,y-h_2)=z_1$. 

Note that since $\be$ is affine in columns and $z_1-z_2=z_3-z_4$, we can deduce from this that 
$\be(x_i,y-h_1-h_2)=z_1$ for each $i$. 

The previous lemma implies that 
\[\psi(x_1,v)-\psi(x_2,v)=\psi(x_3,v)-\psi(x_4,v)\]
whenever $v$ is one of $y-h_1,y-h_2$ or $y-h_1-h_2$. That combined with the fact that $\psi$ is affine in each column implies the result stated.
\end{proof}

\begin{theorem} \label{biaffineextension}
Let $G=\F_p^n$, let $t\geq 20k$ and let $B\subset G^2$ be a bilinear Bohr set of codimension $k$ and rank $t$. Then every function $\phi:B\to G$ that is affine in each variable separately can be extended to a function defined on all of $G$ that is affine in each variable separately.
\end{theorem}

\begin{proof}
We just sketch the argument. The results up to the last lemma prove that we can find an extension that is affine in the second variable and, for every choice of the second variable, very close to a Freiman homomorphism in the first variable. By the linear stability results of the previous section, we can restrict each row of $\psi$ to a set of density at least $1-p^{-2k}$ in such a way that $\psi$ is a Freiman homomorphism in the first variable. This Freiman homomorphism must be compatible with the original function $\phi$, so we do not have to remove any points where $\phi$ was defined. Therefore we have an extension of $\phi$ to a subset that has density at least $1-p^{-2k}$ in each row. This has a unique affine extension to the whole row, and it is easy to check that the columns remain affine after we have performed the extension.
\end{proof}

\begin{corollary} \label{bilinearextension}
Let $k$ and $t$ be as above, and let $B$ be a bilinear Bohr set defined by a bilinear (as opposed to bi-affine) function $\be$. Suppose that $\phi:B\to G$ is additive in each variable separately. Then $\phi$ can be extended to a bilinear map $\g:G^2\to G$.
\end{corollary}

\begin{proof}
For every $(x,y)$ such that either $x=0$ or $y=0$, we have that $\phi(x,y)=0$. Therefore the bi-affine extension we obtain from Theorem \ref{biaffineextension} has the same property, which makes it bilinear.
\end{proof}

\section{A stability theorem for near bihomomorphisms on high-rank bilinear Bohr sets}

We now use the results of the previous two sections to show that a function that is a near bihomomorphism on a high-rank bilinear Bohr set can be approximated by a function that is an exact bihomomorphism, which must itself come from a bilinear map from $G^2$ to $G$ and a function of the first variable only.

\begin{lemma} \label{gammatophi}
Let $B$ be a bilinear Bohr set of codimension $k$ and rank $t$, let $\tg:G^2\to G$ be a bilinear function, and let $0<\eta<1/480$. Let $\g:G^2\to\cA$ be given by the formula $\g(x,y)=\d_{\tg(x,y)}$, and let $\phi:B\to\Sigma(\cA)$ be a function such that $\E_{P}d\bigl(\phi(P),\g(w(P),h(P))\bigr)\leq\eta p^{-4k}$. Then there exist functions $\tilde\lambda:G\to G$ and $\tilde\theta:G\to G$ such that $\tilde\lambda$ is linear and such that 
\[\E_{(x,y)\in B}d\bigl(\phi(x,y),\g(x,y)\theta(x)\lambda(y)\bigr)\leq 180\eta,\]
where $\g(x,y)=\d_{\tg(x,y)}$, $\theta(x)=\d_{\tilde\theta(x)}$ and $\lambda(y)=\d_{\tilde\lambda(y)}$ for each $x,y$.
\end{lemma}

\begin{proof}
Define $\phi_1(x,y)$ to be $\phi(x,y)\g(x,y)^*$. Since $\tg$ is bilinear, $\g(P)=\g(w(P),h(P))$ for every vertical parallelogram, so the starting assumption can be rewritten as $\E_Pd\bigl(\phi_1(P),\d_0\bigr)\leq\eta p^{-4k}$. Expanding this out and rearranging gives us the inequality
\[\E_{x_1,x_2,y_1,y_2,h}d\bigl(\phi_1(x_1,y_1+h)\phi_1(x_1,y_1)^*,\phi_1(x_2,y_2+h)\phi_1(x_2,y_2)^*\bigr)\leq\eta p^{-4k}.\]
Now set $\lambda_1(h)=\E_{x,y}\phi_1(x,y+h)\phi_1(x,y)^*$. Then the above inequality can be rewritten as
\[\E_{x,y,h}d\bigl(\phi_1(x,y+h)\phi_1(x,y)^*,\lambda_1(h)\bigr)\leq\eta p^{-4k}.\]
Finally, let $\lambda_2(h)=\lambda_1(h)/\|\lambda_1(h)\|_1$, or an arbitrary function with $\ell_1$-norm 1 if $\|\lambda_1(h)\|_1=0$. For all $h$ outside a set of density $p^{k-t}$ we have that $\|\lambda_1(h)\|_1=p^{-2k}\pm p^{k-t}$, so
\[\E_{x,y,h}d\bigl(\phi_1(x,y+h)\phi_1(x,y)^*,\lambda_2(h)\bigr)\leq\eta p^{-2k}+2p^{k-t}.\]

For each $x,y,h_1,h_2$ let $b(x,y,h_1,h_1)$ be 1 if the points $(x,y),(x,y+h_1)$ and $(x,y+h_1+h_2)$ all belong to $B$. Then for every $x,y,h_1,h_2$, two applications of the triangle inequality give that
\begin{align*}
b(x,y,h_1,h_2)&d\bigl(\lambda_2(h_1)\lambda_2(h_2),\lambda_2(h_1+h_2)\bigr)\\
\leq\ &b(x,y,h_1,h_2)d\bigl(\lambda_2(h_1),\phi_1(x,y+h_1)\phi_1(x,y)^*\bigr)\\
&+b(x,y,h_1,h_2)d\bigl(\lambda_2(h_2),\phi_1(x,y+h_1+h_2),\phi_1(x,y+h_1)^*\bigr)\\
&+b(x,y,h_1,h_2)d\bigl(\lambda_2(h_1+h_2),\phi_1(x,y+h_1+h_2)\phi_1(x,y)^*\bigr).\\
\end{align*}
We now take expectations of all four terms. For each fixed $(h_1,h_2)$ outside a set of density at most $p^{2k-t}$, the probability that $x\in B'_{\bullet h_1}\cap B'_{\bullet h_2}$ is $p^{-2k}$, and for each $x$ outside a set of density at most $p^{k-t}$ the probability that $(x,y)\in B$ is $p^{-k}$. It follows that the expectation of the term on the left-hand side is 
\[p^{-3k}\E_{h_1,h_2}d\bigl(\lambda_2(h_1)\lambda_2(h_2),\lambda_2(h_1+h_2)\bigr)\pm 3p^{2k-t}.\]
For each fixed $x,y,h_1$ such that $(x,y)$ and $(x,y+h_1)$ are in $B$, the expectation of $b(x,y,h_1,h_2)$ is the probability that $h_2\in B'_{x\bullet}$, which is $p^{-k}$ for all $x$ outside a set of density at most $p^{k-t}$. It follows that the expectation of the first term on the right-hand side is 
\[p^{-k}\E_{x,y,h_1}d\bigl(\lambda_2(h_1),\phi_1(x,y+h_1)\phi_1(x,y)^*\bigr)+p^{k-t}\leq\eta p^{-3k}+2p^{k-t}.\]
Essentially the same argument gives the same upper bound for the other two terms, so it follows that
\[\E_{h_1,h_2}d\bigl(\lambda_2(h_1)\lambda_2(h_2),\lambda_2(h_1+h_2)\bigr)\leq 3\eta+5p^{5k-t}\leq 4\eta.\]
Now define a function $\tilde\lambda_3:G\to G$ by setting $\tilde\lambda_3(h)$ to be the most popular value of $\lambda_2(h)$. The inequality above implies that $d\bigl(\lambda_2(h_1)\lambda_2(h_2),\lambda_2(h_1+h_2)\bigr)>9/10$ with probability at least $1-40\eta$, and it is not hard to check that $\tilde\lambda_3(h_1)+\tilde\lambda_3(h_2)=\tilde\lambda_3(h_1+h_2)$ whenever that is the case. Therefore, by Lemmas \ref{additivestability} and \ref{additivestability2}, there is a linear map $\tilde\lambda:G\to G$ that agrees with $\tilde\lambda_3$ on a set of density at least $1-160\eta$. We therefore obtain the inequality
\[\E_{x,y,h}d\bigl(\phi_1(x,y+h)\phi_1(x,y)^*,\lambda(h)\bigr)\leq 170\eta p^{-2k}\]
and therefore
\[\E_{(x,y),(x,y+h)\in B}d\bigl(\phi_1(x,y+h)\phi_1(x,y)^*,\lambda(h)\bigr)\leq 180\eta.\]

Now let $\a$ be a random function from $G$ to $G$ with the property that $(x,\a(x))=0$ whenever $B_{x\bullet}\ne\emptyset$. Then the expectation of 
\[\E_{(x,h)\in B'}d\bigl(\phi_1(x,\a(x)+h)\phi_1(x,\a(x))^*,\lambda(h)\bigr)\]
is at most $180\eta$, so choose $\a$ such that we have this bound. Then setting $y=\a(x)+h$, we have
\[\E_{(x,y)\in B}d\bigl(\phi(x,y),\phi_1(x,\a(x))\lambda(\a(x))^*\lambda(y)\bigr)\leq 180\eta.\]
Setting $\tilde\theta(x)$ to be the most popular value of $\phi_1(x,\a(x))\lambda(\a(x))^*$ gives us the required result.
\end{proof}

For convenience, let us now give a version of the triangle inequality that works when the functions have different sizes.

\begin{lemma}\label{triangle2}
Let $f,g,h$ be non-negative functions defined on a finite set $X$. Then 
\[\|g\|_1d(f,h)\leq\|h\|_1d(f,g)+\|f\|_1d(g,h).\]
\end{lemma}

\begin{proof}
Apply the usual triangle inequality to the functions $f/\|f\|_1$, $g/\|g\|_1$ and $h/\|h\|_1$ and then multiply both sides by $\|f\|_1\|g\|_1\|h\|_1$. (Also, if any of $f,g,h$ is the zero function, then both sides of the inequality are zero.) 
\end{proof}

We are now ready to state and finish the proof of our stability theorem for bihomomorphisms on high-rank bilinear Bohr sets. It states that if $\eta$ is small, then every $(1-\eta)$-bihomomorphism on such a set can be approximated by a function that is essentially the sum of a bilinear function and a function of $x$ only.

\begin{theorem} \label{stability2}
Let $B$ be a bilinear Bohr set of codimension $k$ and rank $t$, let $\cA$ be the group algebra of some group $H=\F_p^m$ for some $m$, and let $\phi:G^2\to\Sigma(\cA)$ be a $(1-\eta)$-bihomomorphism with respect to the characteristic function $b$ of $B$. Then there exist a bilinear function $\g:G^2\to H$ and a function $\theta:G\to H$ such that, setting $\phi_2(x,y)=\delta_{\g(x,y)+\theta(x)}$ for all $(x,y)\in B$ and $\phi_2(x,y)=0$ otherwise, we have the inequality
\[\E_{(x,y)\in B}d\bigl(\phi(x,y),\phi_2(x,y)\bigr)\leq 27000\eta.\]
\end{theorem}

\begin{proof}
Let $\psi=\mc\phi$. The statement that $\phi$ is a $(1-\eta)$-bihomomorphism is the inequality $\langle\mc\phi,\mc\phi\rangle\geq(1-\eta)\langle\mc b,\mc b\rangle$. The right-hand side of this inequality is at least $(1-2\eta)p^{-7k}$. The left-hand side can be rewritten as $\E_P\langle\phi(P),\psi(w(P),h(P))\rangle$, where the average is over all vertical parallelograms in $G^2$. 

The probability that $(w(P),h(P))\in B''$ is at least $1-p^{k-t}$. We also have that $\|\phi(P)\|_1=1$ if $P$ is in $B$, and otherwise $\phi(P)=0$. By Lemma \ref{noofvps}, the probability of the first of these alternatives given that $(w(P),h(P))\in B''$ is $p^{-3k}\pm 2p^{k-t}$. By the same lemma, there is a conditional probability of at least $1-p^{k-t}$ that $\|\psi(w(P),h(P))\|_1=p^{-3k}\pm 2p^{k-t}$. It follows from this and our lower bound on $t$ that $\E_Pd\bigl(\phi(P),\psi(w(P),h(P))\bigr)\leq 3\eta p^{-7k}$.

Theorem \ref{stability} gives us a function $\tpsi_2:B''\to G$ that is additive in each variable separately such that, setting $\psi_2(w,h)=\d_{\tpsi_2(w,h)}$ for each $(w,h)\in B''$, we have that $\E_{w,h}d\bigl(\psi(w,h),\psi_2(w,h)\bigr)\leq 128\eta p^{-4k}$. It follows that $\E_Pd\bigl(\psi(w(P),h(P)),\psi_2(w(P),h(P))\bigr)\leq 128\eta p^{-4k}$.

Therefore, if we pick $P$ at random, then with probability at least $p^{k-t}$ we have the inequality
\begin{align*}(p^{-3k}\pm 2p^{k-t})d\bigl(\phi(P),&\psi_2\bigl(w(P),h(P)\bigr)\\
&\leq d\bigl(\phi(P),\psi(w(P),h(P)\bigr)+\b1_{P\subset B}d\bigl(\psi(w(P),h(P)),\psi_2(w(P),h(P))\bigr),\\
\end{align*}
by Lemma \ref{triangle2}. 

We now take the expectation of both sides over all vertical parallelograms $P$. The left-hand side gives us what we want to estimate, multiplied by $p^{-3k}\pm 2p^{k-t}$. The first term on the right-hand side gives, as we showed above, at most $3\eta p^{-7k}$. The last term gives us the expectation of $F(w,h)$, where $F(w,h)$ is $d\bigl(\psi(w,h),\psi_2(w,h)\bigr)$ times the probability that $P\subset B$ given that $P$ is width $w$ and height $h$. This probability is zero when $(w,h)\notin B''$ and $p^{-3k}\pm 2p^{k-t}$ for all other $(w,h)$ outside a set of density at most $p^{k-t}$. We therefore obtain at most $128\eta p^{-7k}+3p^{k-t}$. Therefore, using our lower bound on $t$ again, we deduce that
\[\E_Pd\bigl(\phi(P),\psi_2(w(P),h(P))\bigr)\leq 150\eta p^{-4k}.\]

By Corollary \ref{bilinearextension}, the function $\psi_2$ can be extended to a bilinear function $\g$ defined on all of $G$, so we may replace $\psi_2$ by $\g$ in the above inequality. The result now follows from Lemma \ref{gammatophi}.
\end{proof}

While the above theorem is the one that will be convenient for our application, it has a corollary that is simpler to state.

\begin{corollary}
Let $B$ be a bilinear Bohr set of codimension $k$ and rank $t$, and let $\phi:G^2\to G$ be a function that respects a proportion $(1-\eta)$-bihomomorphism of the 4-arrangements in $B$. Then there exist a bilinear function $\g:G^2\to G$ and a function $\theta:G\to G$ such that $\phi(x,y)=\g(x,y)+\theta(x)$ for every $(x,y)$ in a subset of $B$ of relative density at least $1-27000\eta$.
\end{corollary}

\section{Obtaining bilinear structure for the original function $\phi$}

Before we continue, let us take stock of what we have proved so far. We started with a function $f$, a constant $c_1>0$, a set $A\subset G^2$ of density at least $c_1$, and a function $\tphi:A\to\hat G$ such that $|\widehat{\partial_{a,b}f}(\tphi(a,b))|\geq c_1^{1/2}$ for every $(a,b)\in A$. (At that stage we called it $\phi$.) We then passed to a large subset $A'\subset A$ such that the restriction of $\tphi$ to $A'$ respected most second-order 4-arrangements. We then created the closely related function $\phi:G^2\to\cA$, defined by setting $\phi(x,y)$ to be $\delta_{\tphi(x,y)}$ if $(x,y)\in A'$ and $0$ otherwise. The condition on $\tphi$ gave us that $\phi$ was a $(1-\eta)$-bihomomorphism for some absolute constant $\eta$ that we were free to choose.

We then defined $\psi$ to be the mixed convolution $\mc\phi$ of $\phi$ and used a bilinear Bogolyubov method combined with Cauchy-Schwarz and averaging to prove that there is a high-rank bilinear Bohr set $B$ (that lives inside a product of two low-codimensional subspaces) such that the restriction of $\psi$ to $B$ is still a $(1-2\eta)$-bihomomorphism. Then we proved a stability theorem and an extensthat implies that the restriction of $\psi$ to $B$ is close to a function $\psi_2$ where $\psi_2(x,y)=\delta_{\g(x,y)+\theta(x)}$ for each $(x,y)\in B$, where $\g$ is a bilinear function defined on the whole of $G^2$ and $\theta$ is some arbitrary function from $G$ to $G$.

The next couple of steps are to show what this implies about the original function $\phi$. Our argument will be somewhat similar to the argument in the previous section, but this time it is a ``1\% version" instead of a ``99\% version". Also, since $\phi$ is a function from $G$ to $G$, the argument is somewhat simpler.

\begin{lemma} \label{backtophi}
Let $G=\F_p^n$, let $A'$ be a subset of $G^2$, let $\phi:A'\to G$ be a function such that $\phi(x,y)$ is the restriction of an affine function of $y$ for each fixed $x$, let $\psi:G^2\to G$ be a function of the form $\g+\theta$, where $\g$ is bi-affine and $\theta$ depends on $x$ only, and suppose that there are at least $\d|G|^5$ vertical parallelograms $P$ such that $\phi(P)=\psi(w(P),h(P))$, where $w(P)$ and $h(P)$ are the width and height of $P$. Then there exist functions $\theta,\lambda:G\to G$ such that $\lambda$ is affine, and a subset $A''$ of $A'$ of density at least $\d$ (in $G^2$) such that $\phi(a,b)=\g(a,b)+\theta(a)+\lambda(b)$ for every $(a,b)\in A''$.
\end{lemma}

\begin{proof}
Let us call a vertical parallelogram $P$ in $A'$ of width $w$ and height $h$ \emph{good} if $\phi(P)=\g(w,h)+\theta(w)$. Then our assumption implies that if $P$ is a random vertical parallelogram in $G$, then the probability that it belongs to $A'$ and is good is at least $\d$. It follows that if we choose two points $(x,y)$ and $(x',y')$ at random from $G^2$, then the expected number of $h\in G$ such that the vertical parallelogram with points $(x,y), (x,y+h), (x',y'), (x', y'+h)$ is good is $\d|G|$. From that it follows that if we choose a random function $\alpha:G\to G$, and then choose random $x,x'\in G$, then the expected number of $h$ such that $(x,\alpha(x)), (x,\alpha(x)+h), (x',\alpha(x'), (x', \a(x')+h)$ form a good parallelogram is at least $\d|G|$. Therefore by further averaging we can find $\a$ and $x$ such that for at least $\d|G|^2$ pairs $(x',h)$ that vertical parallelogram is good.

But if it is good, then 
\[\phi(x',\a(x')+h)=\phi(x,\a(x)+h)+\phi(x',\a(x'))-\phi(x,\a(x))+\g(x'-x,h)+\theta(x'-x).\] 
Here we have fixed $x$ and are interested in the dependence on $x'$ and $h$. Since $\phi$ is affine in the columns, the first term has an affine dependence on $h$ (and is independent of $x'$). The second term depends on $x'$ only. The third is constant. The fourth can be written as
\[\g(x'-x,h)=\g(x',h)-\g(x,h)+(x-x').b\]
for some $b\in G$, since $\g$ is bi-affine. And the fifth term depends on $x'$ only. 

We have now shown that there are at least $\d|G|^2$ pairs $(x',h)$ we have
\[\phi(x',\a(x')+h)=\g(x',h)+\theta_1(x')+\lambda_1(h)\]
for some function $\theta_1$ and affine map $\lambda$. Therefore,
\[\phi(x',y')=\g(x',y'-\a(x'))+\theta_1(x')+\lambda_1(y'-\a(x')),\]
and using the fact that $\g$ is bi-affine and $\lambda_1$ is affine, that enables us to rewrite the right-hand side as $\g(x',y')+\theta(x')+\lambda(y')$, for suitably defined $\theta$ and affine $\lambda$. The result follows.
\end{proof}

We would now like to get rid of the troublesome function $\theta$. In order to do that, we return to the origin of the function $\phi$. We now know that
\[|\widehat{\partial_{a,b}f}(\g(a,b)+\theta(a)+\lambda(b))|\geq c_1^{1/2}\]
for every $(a,b)\in A''$, and that $A''$ has density at least $\d$. 

Let us remove from $A''$ all columns $A''_{a\bullet}$ of density less than $\d/2$ in $G$. That leaves a subset $A'''$ of density at least $\d/2$ in $G^2$ such that every column has density at least $\d/2$ in $G$. Let $B$ be the set of all $a$ such that $A''_{a\bullet}$ has density at least $\d/2$. Then
\[\E_{a,b}\b1_B(a)|\widehat{\partial_{a,b}f}(\g(a,b)+\theta(a)+\lambda(b))|^2\geq\d c_1/2.\]

By averaging, we can find $b$ such that
\[\E_a\b1_B(a)|\widehat{\partial_{a,b}f}(\g(a,b)+\theta(a)+\lambda(b))|^2\geq\d c_1/2.\]
Writing $g$ for the bounded function $\partial_bf$ and $\zeta(a)$ for $\g(a,b)+\lambda(b)$, which has an affine dependence on $a$, we can rewrite this inequality as
\[\E_a\b1_B(a)|\widehat{\partial_ag}(\zeta(a)+\theta(a))|^2\geq\d c_1/2.\]
By Lemma \ref{linearaddquads} and the fact that $\zeta$ is affine, it follows that there are at least $(\d c_1/2)^4|G|^3$ quadruples $(a_1,a_2,a_3,a_4)\in B^4$ such that $a_1-a_2=a_3-a_4$ and $\theta(a_1)-\theta(a_2)=\theta(a_3)-\theta(a_4)$.

Lemma \ref{BSG} with $\a=(\d c_1/2)^4$ now gives us a s subset $B'\subset B$ of density at least $\d^{16}c_1^{16}/2^{19}$ such that, writing $\G'$ for the restriction of the graph of $\phi$ to $B'$, we have the inequality $|\G'-\G'|\leq 2^{86}|\G'|/\d^{64}c_1^{64}$. Let $K=2^{86}\d^{-64}c_1^{-64}$.

Applying Lemma \ref{sanders}, we obtain a subspace $V\subset(\F_p^n)^2$ of cardinality at most $|\G'|$ such that $|\G'\cap V|\geq\exp(-2^{42}(\log K+\log p)^6)|\G'|$.

Write $V=W+V'$, where $W\subset\{(x,y)\in(\F_p^n)^2:x=0\}$ and $W\cap V'=\{(0,0)\}$. Since $\G'$ is the graph of a function, $|G'\cap V|\leq|V|/|W|$. It follows that there is some translate $w+V'$ such that $|\G'\cap(w+V')|\geq\exp(-2^{43}(\log K+\log p)^6)|\G'|$. But $w+V'$ is the graph of an affine map restricted to a subspace. Let $\mu$ be some extension of this affine map to all of $G$. Then the restriction of $\theta$ to $B'$ agrees with $\mu$ on a set $B''$ of size at least $\exp(-2^{43}(\log K+\log p)^6)|B'|$.

We also have that $A''\cap(B''\times G)\geq(\d/2)|B''||G|$, so we have a subset $A'''$ of $A''$ of density at least $(\d/2)\exp(-2^{43}(\log K+\log p)^6)\d^{16}c_1^{16}/2^{19}\geq\exp(-2^{44}(\log K+\log p)^6)$ such that 
\[\phi(a,b)=\g(a,b)+\mu(a)+\lambda(b)\]
for every $(a,b)\in A'''$.

\section{A symmetry argument for trilinear forms} \label{symmetry}

The results up to now are sufficient to establish that if $G=\F_p^n$ and $f:G\to\C$ is a function with $\|f\|_\infty\leq 1$ and $\|f\|_{U^3}\geq\a$, then there is a bi-affine map $\phi_1:G^2\to G$ such that
\[\E_{a,b}|\widehat{\partial_{a,b}f}(\phi_1(a,b))|^2\geq\a'\]
where $\a'>0$ depends on $\a$ only (and in a ``reasonable" way). But the left-hand side is equal to
\[\E_{a,b}|\E_x\partial_{a,b}f(x)\omega^{-x.(\phi_1(a,b))}|^2=\E_{x,a,b,c}\partial_{a,b,c}f(x)\omega^{-c.(\phi_1(a,b))}.\]
Let $\phi$ be the linear part of $\phi_1$. Then 
\[\phi_1(a,b)=\phi(a,b)+\rho(a)+\sigma(b)\]
for some pair of affine maps $\rho,\sigma:G\to G$. It follows that 
\begin{equation} \label{ineq}
\E_x\E_{a,b,c}\partial_{a,b,c}f(x)\omega^{-\tau(a,b,c)-\rho(a).c-\sigma(b).c}\geq\a',
\end{equation}
where $\tau$ is the trilinear form given by the formula $\tau(a,b,c)=\phi(a,b).c$.

As in other proofs of $U^k$ inverse theorems, we would be in a very good position (as we shall see later) if we knew that $\tau$ was symmetric in the three variables $a,b,c$. This is not the case in general, but we shall prove that $\tau$ differs from a symmetric trilinear form by a form of low rank, where ``rank" is the analytic rank defined in \cite{gw} and briefly mentioned in subsection \ref{rank}.

The definition is very similar to the definition of analytic rank for bilinear maps. We define the \emph{(analytic) rank} of a trilinear form $\tau$ to be $-\log_p\E_{a,b,c}\omega^{\tau(a,b,c)}$. However, unlike in the bilinear case, this does not seem to have an equivalent algebraic definition: indeed, it does not even have to be an integer. We do, however, have the following very useful fact. A proof (of a result for general multilinear forms of which this is the trilinear case) can be found in \cite{gowerswolf}, where it appears as Lemma 5.9.

\begin{lemma} \label{subadditive}
Let $\sigma$ and $\tau$ be trilinear forms. Then 
\[\rank(\sigma+\tau)\leq 8(\rank(\sigma)+\rank(\tau)).\]
\end{lemma}

A second important fact about high-rank trilinear forms is the following. It is a consequence of two more basic facts: that $\E_{a,b,c}\omega^{\tau(a,b,c)}$ is equal to the eighth power of the box norm of $\omega^\tau$ and that if $u,v,w$ are three bounded functions from $G^2$ to $\C$ and $f:G^3\to C$, then we have the inequality
\[|\E_{a,b,c}u(a,b)v(b,c)w(a,c)f(a,b,c)|\leq\|f\|_\square,\]
where $\|.\|_\square$ denotes the three-variable box norm.

\begin{lemma}\label{rankcriterion}
Let $\tau$ be a trilinear form of rank $r$ and let $u,v,w:G^2\to\C$ be bounded functions. Then
\[|\E_{a,b,c}u(a,b)v(b,c)w(a,c)\omega^{-\tau(a,b,c)}|\leq p^{-r/8}.\]
\end{lemma}

A proof can be found, for example, in \cite{gowerswolf}, where a slightly stronger and more general statement appears as Lemma 5.4. 

If $\tau$ is a trilinear form, there is a natural candidate for the ``closest" symmetric trilinear form to $\tau$, which is
\[\sigma(a,b,c)=\frac 16(\tau(a,b,c)+\tau(a,c,b)+\tau(b,a,c)+\tau(b,c,a)+\tau(c,a,b)+\tau(c,b,a)).\]
(This is defined only if $p\geq 5$, but the inverse theorem as stated in this paper is known to be false in small characteristic.) We shall show that if $\tau$ satisfies the inequality (\ref{ineq}) for some choice of bounded functions $u,v,w,f$, then $\tau-\sigma$ has low rank. We shall then discuss the structure of low-rank trilinear forms and use it to obtain an inequality similar to (\ref{ineq}) but with the symmetric trilinear form $\sigma$ replacing $\tau$. This basic scheme of proof was invented by Green and Tao in their proof of an inverse theorem for the $U^3$ norm, which required a symmetry argument for bilinear forms. \cite{greentao1}.

\subsection{The difference between $\sigma$ and $\tau$ has low rank}

We shall prove this by showing that for each permutation $\pi$ of the variables, the trilinear form $\tau(a_1,a_2,a_3)-\tau(a_{\pi(1)},a_{\pi(2)},a_{\pi(3)})$ has small rank. The result will then follow from the additivity of analytic rank (together with the fact that the rank is not affected if we multiply by a scalar).

We begin by proving that the trilinear form $\tau(a,b,c)-\tau(a,c,b)$ has small analytic rank. The beginning of the argument is modelled on Green and Tao's bilinear symmetry argument, but the end of the proof is slightly different and works directly with analytic rank.

We say that a function $u:G\to\C$ is a \emph{linear phase function} if it is of the form $u(x)=\omega^{a.x}$ for some $a\in G$, and that $v:G^2\to\C$ is a \emph{bilinear phase function} if it is of the form $u(x,y)=\omega^{\xi(x,y)}$ for some bilinear form $\xi:G^2\to G$. Note that if this is the case, then $u(x,y_1+y_2)=u(x,y_1)u(x,y_2)$, and similarly for the other variable.

\begin{lemma} \label{partialsymmetry}
Let $G=\F_p^n$, let $f:G\to\C$, let $r,s,t:G\to\C$ be linear phase functions, let $u,v,w:G^2\to\C$ be bilinear phase functions, and let $\t:G^3\to\F_p$ be a trilinear form. Suppose that
\[|\E_{x,a,b,c}r(a)s(b)t(c)u(a,b)v(b,c)w(a,c)\partial_{a,b,c}f(x)\omega^{\t(a,b,c)}|\geq\alpha.\]
Then the trilinear form $\sigma_1$ defined by $\sigma_1(a,b,c)=\t(a,b,c)-\t(a,c,b)$ has analytic rank at most $\log_p(1/\a)$.
\end{lemma}

\begin{proof}
The initial assumption can be rewritten
\[|\E_a\E_{x,b,c}r(a)s(b)t(c)u(a,b)v(b,c)w(a,c)\partial_af(x)\partial_af(x-b)\partial_af(x-c)\partial_af(x-b-c)\omega^{\t(a,b,c)}|\geq\a.\]
Making the change of variables $z+p=x$, $z-p=x-b-c$, $z+q=x-b$, and $z-q=x-c$, we obtain the inequality
\begin{align*}
|\E_a\E_{z,p,q}r(a)s(p-q)t(p+q)u(a,p-q)&v(p-q,p+q)w(a,p+q)\partial_af(z+p)\\
&\partial_af(z+q)\partial_af(z-q)\partial_af(z-p)\omega^{\t(a,p-q,p+q)}|\geq\a.\\
\end{align*}
Averaging, we find $z$ such that
\begin{align*}
|\E_a\E_{p,q}r(a)s(p-q)t(p+q)u(a,p-q)&v(p-q,p+q)w(a,p+q)\partial_af(z+p)\\
&\partial_af(z+q)\partial_af(z-q)\partial_af(z-p)\omega^{\t(a,p-q,p+q)}|\geq\a.\\
\end{align*}

Expanding out the exponent, we get 
\[\t(a,u-v,u+v)=\t(a,u,v)-\t(a,v,u)\]
plus terms that depend on only two of the three variables $a,u,v$. Also, since $u$ and $w$ are bilinear phase functions we have that $u(a,p-q)=u(a,p)\overline{u(a,q)}$ and that $w(a,p+q)=w(a,p)w(a,q)$. It follows that we can rewrite the above inequality in the form
\[|\E_{a,u,v}g_1(a,u)g_2(a,v)g_3(u,v)\omega^{\t(a,u,v)-\t(a,v,u)}|\geq\a\]
for some triple of bounded functions $g_1,g_2,g_3:G^2\to\C$.  By Lemma \ref{rankcriterion}, it follows that if $r$ is the rank of $\t(a,u,v)-\t(a,v,u)$, then $p^{-r}\geq\a$, which proves the lemma.
\end{proof} 

\begin{corollary} \label{fullsymmetry}
Let $f,\t$ satisfy the conditions of Lemma \ref{partialsymmetry}. Then the trilinear form
\[\rho(a,b,c)=5\t(a,b,c)-\t(a,c,b)-\t(b,a,c)-\t(b,c,a)-\t(c,a,b)-\t(c,b,a)\]
has analytic rank at most $2^{12}\log_p(1/\a)$.
\end{corollary}

\begin{proof}
Let $r=\log_p(1/\a)$. Lemma \ref{partialsymmetry} tells us that the trilinear form $\t(a,b,c)-\t(a,c,b)$ has rank at most $r$. By symmetry, it also gives us that the trilinear forms $\t(a,b,c)-\t(b,a,c)$ and $\t(a,b,c)-\t(c,b,a)$ have rank at most $r$.

Recalling the definition, the second statement is that $|\E_{a,b,c}\omega^{\t(a,b,c)-\t(b,a,c)}|\geq p^{-r}$. By renaming $b$ as $c$ and $c$ as $b$, we deduce that $|\E_{a,b,c}\omega^{\t(a,c,b)-\t(c,a,b)}|\geq p^{-r}$ as well. That is, the trilinear form $\t(a,c,b)-\t(c,a,b)$ has rank at most $r$. By Lemma \ref{subadditive} and the fact that $\t(a,b,c)-\t(a,c,b)$ has rank at most $r$, we deduce that $\t(a,b,c)-\t(c,a,b)$ has rank at most $16r$. Similarly, we obtain that $\t(a,b,c)-\t(b,c,a)$ has rank at most $16r$.

We have shown that $\sigma$ is a sum of three trilinear forms of rank at most $r$ and two of rank at most $16r$. Applying Lemma \ref{subadditive} a few more times, we deduce the result.
\end{proof}

\begin{corollary} \label{symmetry}
Let $\tau$ be as in Lemma \ref{fullsymmetry}. Then there is a symmetric trilinear form $\sigma$ such that $\tau-\sigma$ has rank at most $2^{12}\log_p(1/\a)$.
\end{corollary}

\begin{proof}
We have proved that $6\tau$ differs from the trilinear form with formula
\[\t(a,b,c)+\t(a,c,b)+\t(b,a,c)+\t(b,c,a)+\t(c,a,b)+\t(c,b,a)\]
by a form of rank at most $2^{12}\log_p(1/\a)$.

We mentioned earlier that the rank of a trilinear form is not affected by scalar multiplication (by a non-zero scalar). That is because $\lambda\t(a,b,c)=\t(\lambda a,b,c)$, which can be seen to have the same rank as $\t$ by the obvious change of variables $a'=\lambda a$.
\end{proof}

\subsection{The structure of low-rank trilinear forms}

We now know that although our trilinear form $\t$ is not necessarily symmetric, it differs from a symmetric form $\sigma$ by a low-rank form $\rho$. We shall now prove that $\rho$ has a simple structure that enables us to deal with it. Results of this kind are known already, though with bad bounds, for general multilinear forms, and the main lemma, Lemma \ref{lowranksubspace} below, follows from a result of Meshulam \cite{meshulam}, but since we have a fairly simple argument, we give a complete proof.

For each $x\in G$, let $\rho_x$ be the bilinear map $(y,z)\mapsto\rho(x,y,z)$. By the trilinearity of $\rho$, we have that the map $x\mapsto\rho_x$ is linear. Also, for each $x$ we have a linear map $T_x$ such that $\rho_x(y,z)=y.T_xz$ for every $y,z$, and the map $x\mapsto T_x$ is also linear. Furthermore, the rank of the bilinear map $\rho_x$ is equal both to its analytic rank and to the rank of the linear map $T_x$. And a final simple observation is that if $r$ is the rank of $\rho$ and for each $x$ $r_x$ is the rank of $\rho_x$, then $p^{-r}=\E_xp^{-r_x}$. 

From this last observation, it follows that there is a set $A\subset G$ of density at least $p^{-r}/2$ such that $p^{-r_x}\geq p^{-r}/2$ for every $x\in A$, and therefore $r_x\leq r+\log_p2\leq 2r$ for every $x\in A$. Since rank is subadditive and equals analytic rank, it follows that $r_x\leq 8r$ for every $x\in 2A-2A$. But by Sanders's bounds for Bogolyubov's lemma, $2A-2A$ contains a subspace of codimension $k$, where $k$ is at most $2^{45}(\log p)^8r^4$. Let $V$ be such a subspace. 

We now prove a simple lemma that tells us that we can pass to low-codimensional subspaces. 

\begin{lemma} \label{passtosubspace}
Let $G=\F_p^n$, let $h$ be a product of linear and bilinear phase functions on $G$ and $G^2$, let $\tau:G^3\to G$ be a trilinear form, and let $V_0$ be a subspace of $G$. Let $f:G\to\C$ be such that
\[|\E_{x}\E_{a,b,c\in V_0}h(a,b,c)\partial_{a,b,c}f(x)\omega^{\tau(a,b,c)}|\geq\a.\]
Let $V$ be a subspace of $V_0$. Then there exists $w\in V_0$ and another product $h_1$ of linear and bilinear phase functions, possibly multiplied by a root of unity, such that
\[|\E_x\E_{a,b,c\in V}h_1(a,b,c)\partial_{a,b,c}f(x-w)\omega^{\tau(a,b,c)}|\geq\a.\]
\end{lemma}

\begin{proof}
Unsurprisingly, the proof is a simple averaging argument. Indeed, by averaging, we can find cosets $V_1,V_2,V_3$ of $V$ in $V_0$ such that
\[|\E_{x}\E_{a\in V_1}\E_{b\in V_2}\E_{c\in V_3}h(a,b,c)\partial_{a,b,c}f(x)\omega^{\rho(a,b,c)+\sigma(a,b,c)}|\geq\a.\]
Equivalently, we can find $a_0,b_0,c_0\in V_0$ such that
\[\E_{x}\E_{a,b,c\in V}h(a+a_0,b+b_0,c+c_0)\partial_{a,b,c}f(x-a_0-b_0-c_0)\omega^{\tau(a-a_0,b-b_0,c-c_0)}|\geq\a.\]
Since $\tau$ is trilinear, $\tau(a-a_0,b-b_0,c-c_0)$ differs from $\tau(a,b,c)$ by a sum of bilinear functions in two of the variables, linear functions in one of the variables, and a constant. Also, $h(a+a_0,b+b_0,c+c_0)$ is a product of linear and bilinear phase functions and a power of $\omega$. Therefore, setting $w=a_0+b_0+c_0$, there exists a suitable function $h_1$ with the property stated.
\end{proof}

It might look as though we have got something for nothing here, given that the lower bound $\a$ has not changed. However, the price we pay is that we will end up with a trilinear phase function that correlates with $f$ a little bit more locally than we want, and although this will imply correlation with a global trilinear phase function, that correlation decreases as the codimension of $V$ increases.

The next step will be to characterize spaces of linear maps when all the maps have low rank.

\begin{lemma} \label{lowranksubspace}
Let $X, Y$ and $V$ be finite-dimensional vector spaces over $\F_p$. For each $x\in V$ let $T_x\in L(X,Y)$ be a linear map of rank at most $k$ and suppose that the map $x\mapsto T_x$ is linear. Then there exist a $k^2$-codimensional subspace $W$ of $V$, a $k$-codimensional subspace $E\subset X$, and a $k$-dimensional subspace $F\subset Y$ such that $T_xu\in F$ for every $x\in W$ and every $u\in E$.
\end{lemma}

\begin{proof}
Without loss of generality there is some $x$ such that $T_x$ has rank equal to $k$. Pick bases of $X$ and $Y$ such that with respect to these bases the matrix of $T_x$ is $\begin{pmatrix}I_k&0\\ 0&0\\ \end{pmatrix}$, where $I_k$ is the $k\times k$ identity matrix.

Let $W$ be the set of all $u\in V$ such that the top left $k\times k$ part of the matrix of $T_u$ is zero. Then $W$ is a subspace of codimension at most $k^2$ in $V$.

Suppose now, with a view to finding a contradiction, that there exists a matrix $M\in W$ with a non-zero entry $M_{ij}$ such that $i,j>k$. By changing bases appropriately we may assume that the matrix of $T_x$ is as before and that $M_{k+1,k+1}=1$.

Now consider the rank of the matrix $T_x+tM$. We would like to prove that it is $k+1$ for at least one value of $t$. To prove this, it is enough to prove that the rank of its restriction to the first $k+1$ rows and columns is $k+1$. Thus, we are considering the rank of the matrix $\begin{pmatrix}I_k&tv\\ tw&t\\ \end{pmatrix}$, where $v$ is a column vector of height $k$ and $w$ is a row vector of length $k$.

For the matrix to be singular, the last row has to be a linear combination of the first $k$ rows. Let $w=(w_1,\dots,w_k)$ and $v=(v_1,\dots,v_k)^T$. Then in the linear combination the coefficient of the $i$th row has to be $tw_i$, so considering the last column we require $\sum_itw_itv_i=t^2w.v=t$, which implies that $t=0$ or $tw.v=1$. Since this cannot hold for more than two values of $t$, we have the desired contradiction.

This tells us that if $M$ is any matrix $T_u$ with $u\in V$, then $M_{ij}=0$ for every $i,j>k$. We may therefore take $E$ and $F$ to be the subspaces spanned by all but the first $k$ basis vectors in $X$ and~$Y$.
\end{proof}

Note that it is an immediate consequence of Lemma \ref{lowranksubspace} that if $X,Y$ and $V$ are finite-dimensional vector spaces over $\F_p$ and for each $x\in V$ there is a bilinear form $\be_x:X\times Y\to\F_p$ such that $\be_x$ depends linearly on $x$ and each $\be_x$ has rank at most $k$, then there are subspaces $W\subset V$, $E\subset X$ and $F_1\subset Y$ such that $\be_x$ is zero on $E\times F_1$ for each $x\in W$. Indeed, if $\be_x(a,b)=T_xa.b$, and each $T_x$ has rank at most $k$, then we can take the subspaces $W,E,F$ given by the lemma and set $F_1=F^\perp$. 

Now let us put this information together. We begin with the inequality
\[|\E_{x,a,b,c}h(a,b,c)\partial_{a,b,c}f(x)\omega^{\rho(a,b,c)+\sigma(a,b,c)}|\geq\a,\]
where $\sigma$ is symmetric, $\rho$ has rank $r$, and $h$ is a product of lower-order multilinear phase functions. We then find a subspace $V$ of codimension at most $k$, which is polynomial in $r$, such that the bilinear map $\rho_a$ has rank at most $16r$ for every $a\in V$. By Lemma \ref{passtosubspace} with $V_0=G$, we can find a constant $w$ and a product $h_1$ of lower-order phase functions such that
\[|\E_{x}\E_{a,b,c\in V}h_1(a,b,c)\partial_{a,b,c}f(x-w)\omega^{\rho(a,b,c)+\sigma(a,b,c)}|\geq\a.\]
Then by the remarks following Lemma \ref{lowranksubspace} we can find subspaces $W,E,F_1$ of $V$ of codimensions at most $k^2$, $k$, and $k$, respectively, such that $\rho$ vanishes on $W\times E\times F_1$. Setting $U=W\cap E\cap F_1$, which has codimension at most $k^2+3k$ in $\F_p^n$, we therefore have that $\rho=0$ on $U\times U\times U$. Therefore, by Lemma \ref{passtosubspace} again, we obtain an inequality of the form
\[|\E_{x}\E_{a,b,c\in U}h_2(a,b,c)\partial_{a,b,c}f(x-w')\omega^{\sigma(a,b,c)}|\geq\a.\]
where $w'\in G$ and $h_2$ is a product of lower-order multilinear phase functions.

Setting $d$ to be the codimension of $U$, we can rewrite this inequality as follows.
\[|\E_{x,a,b,c}\b1_U(a)\b1_U(b)\b1_U(c)h_2(a,b,c)\partial_{a,b,c}f(x-w')\omega^{\sigma(a,b,c)}|\geq\a p^{-4d}.\]
But $\b1_U(x)=\E_{x'\in U^\perp}\omega^{x.x'}$ for every $x$. Therefore, we can rewrite the inequality further as
\[|\E_{x}\E_{a,a',b,b',c,c'}\omega^{a.a'+b.b'+c.c'}h_2(a,b,c)\partial_{a,b,c}f(x-w')\omega^{\sigma(a,b,c)}|\geq\a p^{-4d}.\]
By averaging, we may pick $a',b',c'$ such that, setting $h_3(a,b,c)=\omega^{a.a'+b.b'+c.c'}h_2(a,b,c)$, which is still a product of lower-order multilinear phase functions, we have
\[|\E_{x,a,b,c}h_3(a,b,c)\partial_{a,b,c}f(x-w')\omega^{\sigma(a,b,c)}|\geq\a p^{-4d},\]
which is equivalent to the inequality
\[|\E_{x,a,b,c}h_3(a,b,c)\partial_{a,b,c}f(x)\omega^{\sigma(a,b,c)}|\geq\a p^{-4d}.\]

\section{The final step}

We are now almost done. Let $\kappa(x)=\sigma(x,x,x)$. Since $\sigma$ is symmetric, we have (as may easily be checked) the identity
\begin{align*}\sigma(a,b,c)=-\kappa(x)&+\kappa(x-a)+\kappa(x-b)+\kappa(x-c)\\
&-\kappa(x-a-b)-\kappa(x-b-c)-\kappa(x-a-c)+\kappa(x-a-b-c).\\
\end{align*}
Set $g(x)=f(x)\omega^{-\kappa(x)}$. Then the inequality at the end of the previous section becomes
\[|\E_{x,a,b,c}h_3(a,b,c)\partial_{a,b,c}g(x)|\geq\a p^{-4d}.\]

To complete the proof, we shall show that $g$ has a large $U^3$ norm. Then $g$ will correlate with a quadratic phase function, by the inverse theorem for the $U^3$ norm (which comes with quantitative bounds), which implies, since $f(x)=g(x)\omega^{-\kappa(x)}$, that $f$ correlates with a cubic phase function.

Note that if $h_3$ is identically 1, then we are done, since then the left-hand side is, by definition, the eighth power of the $U^3$ norm of $g$. To deal with general functions $h_3$ we need the following lemma.

\begin{lemma} \label{u3normbig}
Let $g$ be a function from $\F_p^n$ to $\C$ with $\|g\|_\infty\leq 1$, and let $u,v,w$ be functions from $(\F_p^n)^2$ to $\C$, with $\|u\|_\infty\leq 1, \|v\|_\infty\leq 1$ and $\|w\|_\infty\leq 1$. Suppose that
\[|\E_{x,a,b,c}\partial_{a,b,c}g(x)u(a,b)v(b,c)w(a,c)|\geq \a.\]
Then $\|g\|_{U^3}\geq\a$.
\end{lemma}

\begin{proof}
As one might expect, the proof is by repeated use of Cauchy-Schwarz. Let us pull out the variables $x,a,b$ and leave in $c$. Then we have that 
\[\E_{x,a,b}|\E_c\overline{\partial_{a,b}g(x-c)}v(b,c)w(a,c)|\geq\a,\]
since $\partial_{a,b,c}g(x)=\partial_{a,b}g(x)\overline{\partial_{a,b}g(x-c)}$ and the first part of this product does not depend on $c$, and neither does $u(a,b)$.

By Cauchy-Schwarz, it follows that
\[\E_{x,a,b}|\E_c\overline{\partial_{a,b}g(x-c)}v(b,c)w(a,c)|^2\geq\a^2.\]
Expanding out the modulus squared, we get
\[\E_{x,a,b,c_1,c_2}\partial_{a,b}g(x-c_1)\overline{\partial_{a,b}g(x-c_2)}v(b,c_2)\overline{v(b,c_1)}w(a,c_2)\overline{w(a,c_1)}\geq\a^2.\]

Setting $c_3=c_2-c_1$, we can rewrite the above inequality as
\[\E_{x,a,b,c_1,c_3}\partial_{a,b,c_3}g(x-c_1)v_1(b,c_1,c_3)w_1(a,c_1,c_3)\geq\a^2\]
for suitably defined bounded functions $v_1$ and $w_1$. 

Averaging over $c_1$ and writing $c$ for $c_3$, we obtain for some $c_1$ an inequality of the form
\[|\E_{x,a,b,c}\partial_{a,b,c}g(x-c_1)v_2(b,c)w_2(a,c)|\geq\a^2,\]
which is equivalent to the inequality
\[|\E_{x,a,b,c}\partial_{a,b,c}g(x)v_2(b,c)w_2(a,c)|\geq\a^2.\]

By symmetry we can get rid of $v_2$ and $w_2$ in the same way, each time squaring the right-hand side. It follows that $\|g\|_{U^3}^8\geq\a^8$, and the lemma is proved.
\end{proof}

\section{Putting everything together}

Recall that the inverse theorem we are proving states that for every $c>0$ and $p\geq 5$ there exists $c'>0$ such that if $f:\F_p^n\to\C$, $\|f\|_\infty\leq 1$ and $\|f\|_{U^4}\geq c$, then there is a cubic polynomial $\kappa$ such that $|\E_xf(x)\omega^{-\kappa(x)}|\geq c'$. We shall now see how the arguments of the previous sections fit together to give a proof of this statement, and obtain an explicit bound for the dependence of $c'$ on $c$. We shall present the argument in full, but make heavy use of the lemmas and theorems we have proved earlier in the paper.

Suppose, then, that $c>0$, that $p\geq 5$, and that $f:\F_p^n\to\C$ is a function with $\|f\|_\infty\leq 1$ and $\|f\|_{U^4}\geq c$. We continue to write $G$ for the group $\F_p^n$. Recall also a few definitions: we set $\partial_af(x)=f(x)\overline{f(x-a)}$ and $\partial_{ab}=\partial_a\partial_b$; also, we write $\cA$ for the group algebra of $G$, and $\Sigma(\cA)$ for the subset of $\cA$ that consists of non-negative functions that sum to 1.

As noted in Section \ref{overview}, our hypothesis implies that there is a set $A\subset G^2$ of density at least $c/2$ such that $\|\partial_{ab}f\|_{U^2}\geq c/2$ for every $(a,b)\in A$. Since $\|f\|_{U^2}\leq\|\hat f\|_\infty^2$ (because $\|f\|_{U^2}^4=\|\hat f\|_4^4$ and $\|\hat f\|_2^2=\|f\|_2^2\leq\|f\|_\infty^2\leq 1$), this gives us a function $\phi:A\to G$ such that $|\widehat{\partial_{a,b}f}(\phi(a,b))|\geq c_1=(c/2)^{1/2}$ for every $(a,b)\in A$. 

By Lemma \ref{somearrangements}, $\phi$ respects at least $c_2|G|^8$ 4-arrangements in $A$, where $c_2=(c/2)^{16}c_1^{48}=(c/2)^{40}$. Corollary \ref{manyarr2s} then shows that $\phi$ respects at least $c_3|G|^{32}$ second-order 4-arrangements in $A$, where $c_3=c_2^8(c/2)^{-12}=(c/2)^{308}$. And then by Lemma \ref{densification}, for every $\eta>0$, $A$ has a subset $A'$ that contains at least $c_4|G|^{32}$ second-order 4-arrangements, where $c_4=2^{-2^{37}(\log(\eta^{-1})+308\log(2c^{-1}))}$, such that the proportion of these second-order 4-arrangements that are respected by $\phi$ is at least $1-\eta$. A back-of-envelope calculation shows that $c_4\geq 2^{-2^{46}}\eta^{2^{37}}c^{2^{46}}$. Note that the first point of a random second-order 4-arrangement in $G^2$ has a probability $|A'|/|G|$ of belonging to $A'$, so the density of $A'$ is at most $c_4$. (It is possible to improve this bound to a smaller power of $c_4$, but the improvement does not make an interesting difference to our final result.) 

Now let us reinterpret $\phi$ as a function from $G^2$ to $\cA$. More precisely, if $(x,y)\in A'$ and $\phi(x,y)=z$, then we instead set $\phi(x,y)$ to be $\d_z$, and if $(x,y)\notin A'$ we set $\phi(x,y)$ to be 0. The conclusion of Lemma \ref{isoonmixedconv} is equivalent to the statement that the mixed convolution $\psi=\mc\phi$ is a $(1-\eta)$-bihomomorphism, in the sense that $\langle\mc\psi,\mc\psi\rangle\geq\langle\mc\mu,\mc\mu\rangle$, where $\mu=\mc\b1_{A'}$ (or equivalently, $\mu(x,y)=\|\psi(x,y)\|_1$).

We then use a bilinear version of Bogolyubov's method, Theorem \ref{bihomwrtbilinear}, to conclude that there is a bi-affine map $\be$ of codimension $k$ such that if we define $P_\be$ to be the averaging projection to the level sets of $\be$, let $\nu=P_\be\mu$, and let $\psi_1(x,y)=\psi(x,y)/\|\psi(x,y)\|_1$ for each $(x,y)$ (that is, we normalize so that $\psi_1$ takes values in $\Sigma(\cA)$), then $\|\mu-\nu\|_2\leq\zeta$ and $\psi_1$ is a $(1-2\eta)$-bihomomorphism with respect to $\nu$. We may take $k$ to be $2^{10}m^32^m/c_4^4\eta$, where $m=\exp(2^{69}(\log(\zeta^{-1})+\log p)^6)$. We are free to choose $\zeta$.

For any $t$, Corollary \ref{gethighrank} gives us a bilinear Bohr decomposition into sets $B_{v,w,z}$, each of which is a bilinear Bohr set of rank $t$ that lives inside a product of affine subspaces of codimension $r\leq tk$. We shall let $t=\lceil 60k+9\log(\eta^{-1})\rceil$. With this value, all the inequalities that $t$ is required to satisfy are indeed satisfied.

We now apply Corollary \ref{restricttooneset} with $\psi_1$ for $\phi$, $\nu$ for $\mu$, $2\eta$ for $\eta$, and $\zeta^2$ for $\zeta$. We define a function $\xi$ as follows. Let $(x,y)\in G^2$. Then $\xi(x,y)$ is the average of $(\mu(x',y')-\nu(x',y'))^2$ over all $(x',y')\in B_{v,w,z}$, where $B_{v,w,z}$ is the set from the bilinear Bohr decomposition that contains $(x,y)$. Note that the average of $\xi$ is $\|\mu-\nu\|_2^2\leq\zeta^2$.

For a $\g\leq\eta[\mu]^8/8$ of our choice, we obtain some $B_{v,w,z}$ such that the restriction of $\nu$ to $B_{v,w,z}$ takes the value at least $[\nu]^8/2$, which in turn is at least $\|\nu\|_1^8/2=\|\mu\|_1^8/2\geq c_4^{32}/2$, the restriction of $\psi_1$ to $B_{v,w,z}$ is a $(1-8\eta)$-bihomomorphism, and the restriction of $\xi$ to $B_{v,w,z}$ takes the value at most $\g^{-1}\zeta^2$.

Next, we apply Theorem \ref{stability2} with $\psi_1$ replacing $\phi$ and $8\eta$ replacing $\eta$, taking $B$ to be the set $B_{v,w,z}$ that we have just chosen, replacing $G$ by $\F_p^{n-r}$, where $r$ is, as above, the codimension of the two linear spaces $V,W$ that come from the bilinear Bohr decomposition and whose product contains $B_{v,w,z}$, and replacing $H$ by $G$. It gives us a bilinear function $\g:V\times W\to G$ and a function $\theta:V\to G$ such that, setting $\psi_2(x,y)=\d_{\g(x,y)+\theta(x)}$ for every $(x,y)\in B_{v,w,z}$ and 0 otherwise, we have the inequality
\[\E_{(x,y)\in B_{v,w,z}}d\bigl(\psi_1(x,y),\psi_2(x,y)\bigr)\leq 216000\eta.\]

Recall now that $\psi=\mc\phi$ and $\psi_1$ is the normalization of $\psi$, so $\psi=\mu\psi_1$. We also have that $\nu$ averages at least $c_4^4/2$ on $B_{v,w,z}$ and that $(\mu-\nu)^2$ averages at most $\g^{-1}\zeta$, from which it follows that $|\mu-\nu|$ averages at most $\g^{-1/2}\zeta^{1/2}$. As mentioned above, $[\mu]^8\geq c_4^{32}/2$, from which one can check that if we take $\g$ to be $\eta c_4^{32}/16$ and $\zeta$ to be $\eta c_4^{40}/512$, then the average of $|\mu-\nu|$ is less than a quarter the average of $\nu$. By Markov's inequality, it follows from this that $\mu(x,y)\geq\nu(x,y)/2$ for at least half the $(x,y)\in B_{v,w,z}$.

Now let us think about what $\psi_1(x,y)$ actually is, as a probability distribution. We have that $\psi(x,y)$ is the average of $\phi(x,y)$ over all vertical parallelograms of width $x$ and height $y$, and that $\mu(x,y)$ is the probability that such a vertical parallelogram $P$ belongs to $A'$, and hence that $\phi(P)$ (considered as a function to $G$) is well-defined. So $\psi_1(x,y)$ is the probability distribution of possible values of $\phi(P)$ given that it is well-defined, while $\psi(x,y)(z)$ is the probability that $\phi(P)$ is well-defined and equal to $z$.

It follows that $\langle\psi_1(x,y),\psi_2(x,y)\rangle$ is the probability that $\phi(P)=\g(x,y)+\theta(x)$ given that $P$ has width $x$ and height $y$ and belongs to $A'$. The inequality above tells us that this probability is on average at least $1-216000\eta$, so for a fraction at least $3/4$ of the points of $B_{v,w,z}$ it is at least $1-864000\eta$. Let us now take $\eta$ to be $1/1728000$. Since $\mu(x,y)\geq c_4^4/4$ with probability at least $1/2$, it follows that if $(x,y)$ is chosen at random from $B_{v,w,z}$, then with probability at least $1/4$ we have that a random vertical parallelogram $P$ of width $x$ and height $y$ has a probability at least $c_4^4/8$ of belonging to $A'$ and satisfying $\phi(P)=\g(x,y)+\theta(x)$.

It follows that there are at least $c_5|G|^5$ vertical parallelograms $P$ such that $\phi(P)=\g(w(P),h(P))+\theta(w(P))$, where we can take $c_5$ to be $c_4^4/8$ times the density of $B_{v,w,z}$. The density of $B_{v,w,z}$ is at least $p^{-k-r}\geq p^{-k(1+t)}$.

Lemma \ref{backtophi} can now be applied with $\d=c_5$, and it gives us a subset $A''$ of $G$ of density at least $c_5$ in $G^2$ such that $\phi(a,b)=\g(a,b)+\theta(a)+\lambda(b)$ for every $(a,b)\in A''$, where $\lambda$ is affine.

Let $K=2^{86}c_5^{-64}c_1^{-128}$. By the remarks following Lemma \ref{backtophi} we can pass to a subset $A'''$ of density at least $c_6=\exp(-2^{44}(\log K + \log p)^6)$ such that the restriction of $\theta$ (considered as a function of $a$ and $b$ with no $b$-dependence) is affine as well. We therefore have some affine function $\g'$ such that 
\[\E_{a,b}|\widehat{\partial_{a,b}f}(\g'(a,b))|^2\geq (c/2)c_6.\]

This gives rise to a trilinear form $\tau$, and the symmetry argument of Section \ref{symmetry} gives us a symmetric trilinear form $\sigma$ such that the analytic rank of $\tau-\sigma$ is $r\leq 2^{12}\log_p(2/cc_6)\leq 2^{13}\log_p(c_6^{-1})$.

The subsequent argument, culminating in Lemma \ref{u3normbig}, in which we can set $\a=(cc_6/2)p^{-d}$, where $d\leq k_1^2+k_1$ and $k_1\leq 2^{45}r^4(\log p)^8$, gives us a cubic polynomial $\kappa$ such that, setting $g(x)=f(x)\omega^{-\kappa(x)}$, we have the lower bound $\|g\|_{U^3}\geq (cc_6/2)p^{-d}\geq p^{-2d}$.

This statement is our main result, but to convert it into an inverse theorem for the $U^4$ norm, we need as a final step to apply the $U^3$ inverse theorem. Rather than giving an explicit bound at this point, we simply mention that with the help of Sanders's bounds for Bogolyubov's method, one can obtain a quasipolynomial dependence for the $U^3$ inverse theorem, so there is some quadratic polynomial $q$ such that $|\E_xg(x)\omega^{-q(x)}|\geq p^{-A(d+\log p)^C}$ for some absolute constants $A$ and $C$. Since $\E_xg(x)\omega^{-q(x)}=\E_xf(x)\omega^{-\kappa(x)-q(x)}$ and $\kappa+q$ is a cubic, the proof is finished.

It remains to give some idea of how $d$ depends on the initial constant $c$. Tracking back through the calculations above, a painful process that we will not display to the reader, one can obtain an upper bound of $\exp\exp(2^{500}(\log(c^{-1})+\log p)^6)$. Treating $p$ as a constant we can think of this as exponentiating a quasipolynomial function of $c^{-1}$. We then raise this to a further power and exponentiate again to obtain the correlation between $f$ and a cubic phase function, so the final bound is doubly exponential in a quasipolynomial function of $c^{-1}$, so it is ``almost" doubly exponential. (If one had the polynomial Freiman Ruzsa conjecture, one could remove the ``quasi" and obtain a polynomial followed by a double exponential.) 

\section{Concluding remarks}

A natural question is whether the methods of this paper can be used to prove other quantitative inverse theorems. The three most obvious directions of generalization are to groups other than $\F_p^n$ (most notably $\Z_N$), to $\F_p^n$ for $p<5$, and to higher $U^k$ norms.

There do not appear to be fundamental obstacles to generalizing most of the above argument from bilinear functions to multilinear functions and thereby to obtaining a quantitative inverse theorem for the $U^k$ norm for $\F_p^n$ when $p\geq k$. However, there is one aspect of the proof that may possibly cause problems, which is where we get rid of low-rank objects. This happens in two places: when we refine a bilinear Bohr decomposition in such a way that all its parts have high rank, and when (during the symmetry argument) we decompose a trilinear form into a symmetric part and a low-rank part and get rid of the low-rank part. In both places we made use of algebraic arguments concerning ranks of matrices, and it is not obvious that these arguments have higher analogues, which would have to make use of analytic rank instead. There are results in the literature that are closely related to this issue, which show that low-rank polynomials can be expressed economically in terms of polynomials of lower degree, but the bounds obtained are poor \cite{greentao2, KL}. We do not know for sure that it would be necessary to improve these bounds (or solve a similar problem of comparable difficulty), but it does seem to be a distinct possibility.

There also do not appear to be fundamental obstacles to adapting the argument to give a quantitative $U^4$ inverse theorem for functions defined on $\Z_N$. Of course, the role of subspaces in this paper would have to be played by Bohr sets, but it is possible to write down a reasonable (though not completely straightforward) definition of a high-rank bilinear Bohr set, and it seems likely that using standard Bohr-sets technology one could obtain an inverse theorem. The statement that the proof would naturally yield would not be the one involving nilsequences, but it would characterize functions with large $U^4$ norm and would probably be translatable into the nilsequences version.

We have not yet thought about the low-characteristic case of the theorem, but most of the proof does not require $p$ to be large. The only point where we needed this was during the symmetry argument, which involved dividing by 6 inside $\F_p$. It may well be that with a suitable use of non-classical polynomials, one can use most of the argument of this paper and diverge from it only at the end.

Although our bound is ``reasonable", it is very unlikely to be anywhere close to best possible. For all we know, the correct dependence of the final correlation on the $U^4$ norm is of power type. It is possible that one of the exponentials could be removed quite cheaply, since the use of the inclusion-exclusion formula in the proof of Corollary \ref{restricttosubspace}, which is responsible for one of them, is rather crude and not obviously necessary.

Another exponential arises in roughly the same place. Recall that in the proof of the bilinear Bogolyubov lemma, the codimension of the bilinear Bohr set we obtain depends on the number of graphs of affine functions we need to cover the graph of a certain subset of $G^2$. Even if we did not have to use of the inclusion-exclusion formula, without an extra idea, the codimension would be at least a power of the initial constant $c$, and therefore the density would be exponentially small in that power. It is conceivable that one could convert this exponential into a quasipolynomial function by generalizing Sanders's proof, which achieves this in the linear case, but that would be a challenging project.

\end{document}